\documentclass[11pt,reqno,aos]{imsart}
\usepackage{amsfonts}
\usepackage{amssymb, amsmath}

\usepackage[mathscr]{eucal}
\usepackage{graphics}
\usepackage{graphicx}
\usepackage{verbatim}
\usepackage{cases}
\usepackage{float}
\usepackage{url}
\usepackage[font=small,format=plain,labelfont=bf,up,textfont=it,up]{caption}

    \newcommand{\argmin}{\mathop{\rm argmin}}
    \newcommand{\argmax}{\mathop{\rm argmax}}

    \newcommand{\R}{\mathfrak{R}}
    
    \newcommand{\D}{\mathfrak{D}}
    
    \newcommand{\Cov}{\mathrm{Cov}}
    
    \newcommand{\wh}{\widehat}
    \newcommand{\X}{\mathfrak{X}}

    \newcommand{\Y}{\mathfrak{Y}}
    \newcommand{\Z}{\mathfrak{Z}}
    \newcommand{\lefts}[2]{{\vphantom{#2}}_{#1}{#2}}
    \newcommand{\bbR}{{\mathbb R}}

    \newcommand{\blangle}{\big\langle}
    \newcommand{\brangle}{\big\rangle}

\newcommand{\note}[1]
{$^{(!)}$\marginpar[{\hfill\tiny{\sf{#1}}}]{\tiny{\sf{(!) #1}}}}

\newtheorem{theorem}{Theorem}[section]
\newtheorem{lemma}{Lemma}[section]
\newtheorem{proposition}{Proposition}[section]
\newtheorem{definition}{Definition}[section]
\newtheorem{corollary}{Corollary}[section]

\numberwithin{equation}{section}

\newenvironment{proof}[1][\sc Proof.]{\begin{trivlist}
\item[\hskip \labelsep {\bfseries #1}]}{\end{trivlist}}

\newenvironment{remark}[1][Remark]{\begin{trivlist}
\item[\hskip \labelsep {\bfseries #1}]}{\end{trivlist}}

\newcommand{\qed}{\nobreak \ifvmode \relax \else
      \ifdim\lastskip<1.5em \hskip-\lastskip
      \hskip1.5em plus0em minus0.5em \fi \nobreak
      \vrule height0.75em width0.5em depth0.25em\fi}

\begin{document}
\begin{frontmatter}
\title{Theoretical Analysis of Nonparametric Filament Estimation}
\runtitle{Nonparametric filament estimation}
%
\author{\fnms{Wanli} \snm{Qiao}}
\and
    \author{ \fnms{Wolfgang} \snm{Polonik} \ead[label=wp]{wpolonik@ucdavis.edu}}
    \affiliation{University of California, Davis}
    \address{Department of Statistics\\
    University of California\\
    One Shields Ave.\\
    Davis, CA, 95616,
    USA}
     \thankstext{t1}{This research was supported by the NSF-grant DMS \#1107206\newline
}

\date{\today}
\maketitle
\begin{abstract} \noindent This paper provides a rigorous study of the nonparametric estimation of filaments or ridge lines of a probability density $f$. Points on the filament are considered as local extrema of the density when traversing the support of $f$ along the integral curve driven by the vector field of second eigenvectors of the Hessian of $f$. We `parametrize' points on the filaments by such integral curves, and thus both the estimation of integral curves and of filaments will be considered via a plug-in method using kernel density estimation. We establish rates of convergence and asymptotic distribution results for the estimation of both the integral curves and the filaments. The main theoretical result establishes the asymptotic distribution of the uniform deviation of the estimated filament from its theoretical counterpart. This result utilizes the extreme value behavior of non-stationary Gaussian processes indexed by manifolds $M_h, h \in(0,1]$ as $h \to 0$.
\end{abstract}
%
%
%

\end{frontmatter}
\section{Introduction}
Intuitively, a filament or a ridge line is a curve or a lower-dimensional manifold at which the height of a density is higher than in surrounding areas when looking in the `right direction' - a precise definition is given below. For instance, blood vessels, road system, and fault lines can be modeled as filaments. One of the most prominent instances of data sets modeled by means of filaments is the so-called cosmic web, consisting of location of galaxies (Novikov et al. 2006). Cosmologist are very interested in finding a rigorous topological description of this geometric structure because of its relation to the existence of dark matter (Dietrich et al. 2012). In fact, a large body of work on the estimation of filaments and the extraction of their topological structures exists in the corresponding cosmology literature, such as Barrow et al. (1985), Bharadwaj et al. (2004) and Pimbblet et al. (2004). Much of this work is missing theoretical underpinning, however. \\[8pt]
The goal of this paper is to theoretically study the nonparametric estimation of filaments and to develop rigorous theory, in particular distributional results, supporting the proposed estimation approach based on kernel density estimation. Integral curves (cf. (\ref{def-integral-curve}) below) are used to find and to `parametrize' filaments, and thus the estimation of integral curves comes into play here naturally. \\[8pt]
Earlier work on ridge estimation in a statistical context includes  Hall et al. (1992), where several geometric measures of `ridgeness' are defined and investigated. More recent work includes Genovese et al. (2009, 2012a,c) and Chen et al. (2013). Filament estimation is related to several other geometrically motivated concepts, such as manifold learning (Genovese et al. 2012b), investigating modality, edge detection, principle curves (Hastie et al. 1989), locally defined principal curves and surfaces (Ozertem et al. 2011), etc. More recently, the concept of persistent homology explicitly combines statistical mode and antimode estimation with topological concepts (e.g. chapter 5 of Genovese et al. 2013). From a more general perspective all these methods are attempting to find structure in multivariate data with geometric and topological ideas entering the definition of the methodology explicitly (cf. Genovese et al. 2012a).\\[8pt]
The lack of supporting theory, which we address in this paper, is only one challenge of filament estimation. Other challenges include the design of algorithms for tracking filaments. While the design of algorithms was part of this research, it is not included in this paper, but will be published elsewhere. However, geometric algorithms for finding modes or ridge points tend to be based on estimating integral curves (e.g. the well-known mean-shift algorithm estimates the integral curves driven by the gradient (Fukunaga et al. 1975, Cheng 1995, and Comaniciu et al. 2002)). This motivated our study of the estimation of filaments through the lens of estimating integral curves.\\[8pt]
While the notion of a filament has an intuitive geometric interpretation, a rigorous definition is needed here. The definiton of filaments used here is intimately related to integral curves driven by the second eigenvectors of the Hessian matrices of the density function. Only the two-dimensional space will be considered here so that filaments and integral curves are curves in the plane. Extension to higher-dimensional space is possible but some technical problems may come into play. Also, as can be seen from the examples given above, the two-dimensional case covers many important applications of filament estimation. The following definition of filament points can for instance be found in Eberly (1995).
\begin{definition}[filament points in ${\mathbb R}^2$] \label{FilaDef1} Let $f: \mathbb{R}^2\mapsto\mathbb{R}$ be a twice differentiable function with gradient  $\nabla f(x)$ and Hessian matrix $\nabla^2f(x)$. Let $\lambda_2(x) \le \lambda_1(x)$ denote the eigenvalues of the Hessian with corresponding eigenvectors $V(x)$ and $V^\perp(x)$, respectively. A point $x$ is said to be a filament point if
\begin{align}\label{DEF1}
\blangle\nabla f(x),V(x)\brangle=0 \qquad \text{and} \qquad \lambda_2(x)<0.
\end{align}
\end{definition}
Geometrically, $\blangle\nabla f(x),V(x)\brangle$ and $V(x)^T\nabla^2 f(x)V(x)=  \lambda_2(x)\|V(x)\|^2 $ are first and second order directional derivative of $f(x)$ along $V(x)$. Condition (\ref{DEF1}) thus means that a filament point $x$ is a local mode of $f(x)$ along the direction $V(x)$. The idea of using the above characterization of a point on a filament for statistical purposes has been used independently by Genovese et al. (2014) and Chen et al. (2013, 2014a and 2014b).\\[8pt]
%
%
%
By our definition a point on a filament is an extremal point of $f$ when traversing along an integral curve driven by $V(x)$. The integral curve $\X_{x_0}: [-T_{min}, T_{max}] \to {\mathbb R}^2$ (with $T_{min}, T_{max} \geq 0$ and $T_{min}+T_{max}>0$) starting in $x_0 \in {\mathbb R}^2$ driven by $V(x)$ is given by the solution to the differential equation
\begin{align}\label{def-integral-curve}
\frac{d  \X_{x_0}(t)}{dt} = V(\X_{x_0}(t)),\qquad \X_{x_0}(0) = x_0.
\end{align}
Note that $-V(x)$ is also an eigenvector of $H(x)$ and for $\tilde{\X}_{x_0}$ satisfying 
\begin{align}\label{def-integral-curve-reverse}
\frac{d  \tilde{\X}_{x_0}(t)}{dt} = -V(\tilde{\X}_{x_0}(t)),\qquad \tilde{\X}_{x_0}(0) = x_0,
\end{align}
we have $\X_{x_0}(t) = \tilde{\X}_{x_0}(-t)$ for $t\in [-T_{min}, T_{max}].$\\[8pt]
Genovese et al. (2009) use integral curves driven by the gradient field to define a `path density', whose level sets then contain large portions of the filament. Rather than integral curves of gradients, we here use integral curves of the second eigenvector of the Hessian. \\[8pt]
The sampling model considered here consists of independent observations from the underlying pdf. Other sampling models for filament estimation or detection have been used in the statistical literature as well. For instance, Arias-Castro et al. (2006) define a filament as a specific curve (of finite length). Data are then sampled according to a uniform distribution on the curve and background noise is added. Genovese et al. (2012a) also start out with a sample from the curve (filament) but then allow some (small) deviation of the data from the filament.\\[8pt]
Suppose a filament $\mathcal{L}$ exists in the support of a density function $f:\mathbb{R}^2\rightarrow \mathbb{R}^+$. The goal is to find an estimate of $\mathcal{L}$ from a random sample $X_1, X_2, \cdots, X_n$ drawn from $f$, and to assess the reliability of the estimation. Let the first `time point' $t$ at which $\X_{x_0}(t)$ hits the filament $\mathcal{L}$  be denoted by $\theta_{x_0}$, i.e.
$$\X_{x_0}(\theta_{x_0})\in \mathcal{L}.$$
Starting points $x_0$ corresponding to different trajectories lead to different corresponding filament points. The estimation of the filament can be divided into two steps: Estimation of the trajectory $\X_{x_0}(t)$ and estimation of the parameter $\theta_{x_0}$ corresponding to the filament point defined through the trajectory $\X_{x_0}(t)$. Both these quantities will be estimated by plug-in estimates using a kernel density estimator. The corresponding estimates are denoted by $\hat{\X}_{x_0}$ and $\hat\theta_{x_0}$, respectively. The assessment of the uncertainty in the estimation of the filament point $\X_{x_0}(\theta_{x_0})$ through $\wh \X_{x_0}(\wh \theta_{x_0})$ is also based on these two sources of uncertainty, as illustrated in Figure~\ref{FilaEsti}. \\
\begin{figure}[H]
\centering
\includegraphics[width=9cm]{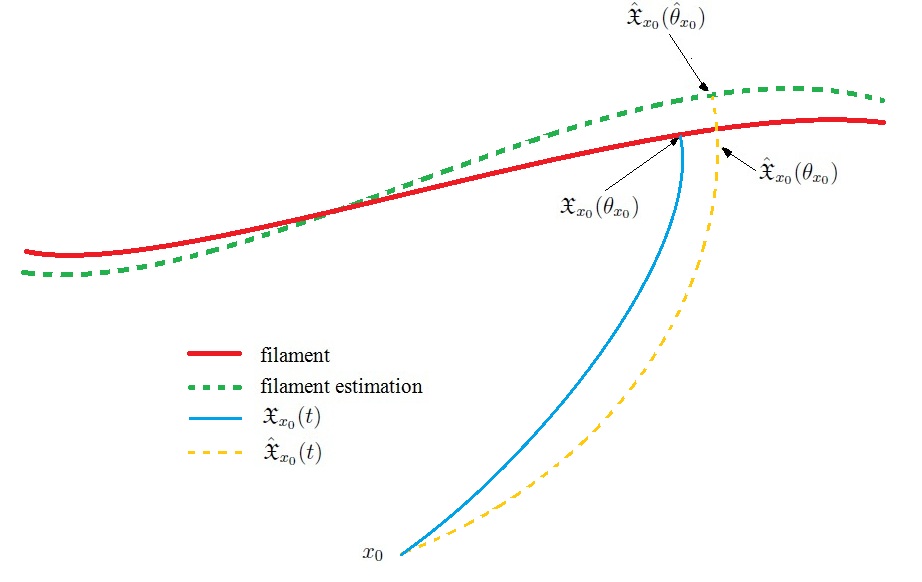}
\caption{Illustration of the integral curve $\X_{x_0}(t)$, its estimate $\hat\X_{x_0}(t)$, and of the estimation of a filament point $\X_{x_0}(\theta_{x_0})$ and its estimate $\hat\X_{x_0}(\hat\theta_{x_0})$.} \label{FilaEsti}
\end{figure}

In this paper we present three different types of results:
\begin{itemize}
\item[(i)] the estimation of the integral curve itself, i.e. we consider the asymptotic behavior of the properly normalized process $\widehat\X_{x_0}(t) - \X_{x_0}(t), t \in [-T_{min},T_{max}]$ with $T_{min}, T_{max}\geq 0$ and $T_{min}+T_{max}>0;$
\item[(ii)] the large sample behavior of  the estimator $\widehat{\theta}_{x_0};$ and
\item[(iii)] by combining results of type (i) and (ii) we will derive large sample behavior of the filament estimate $\wh \X_{x_0}(\wh\theta_{x_0})$. Our main result on filament estimation (Theorem~\ref{ConfBand}) gives the asymptotic distribution of the uniform deviation of the filament estimator. More precisely, we will provide conditions ensuring that:
\end{itemize}
%
{\em
There exists a function $g(x)$ depending on $f$ and on the kernel $K$  used to define our estimators, such that for any fixed $z$, we have

\begin{align*}
\lim_{n\rightarrow\infty}&P\bigg(\sup_{x_0\in\mathcal{G}}\bigg\|g(\X_{x_0}(\theta_{x_0}))\,\sqrt{nh^6}\Big(\hat \X_{x_0}(\hat \theta_{x_0})-\X_{x_0}(\theta_{x_0})\Big) \bigg\|< b_h(z)\bigg) = e^{-2\,e^{-z}},
\end{align*}
where
$ b_h(z)=\sqrt{2\log{h^{-1}}}+\frac{1}{\sqrt{2\log{h^{-1}}}}\big[z+c\big]$ with $c >0$ depending on $f, K$ and $\cal L$,  and ${\mathcal{G}}$ is some properly chosen subregion of $\bbR^2$ such that ${\cal L} = \{\X_{x_0}(\theta_{x_0}): \; x_0\in\cal G\}$, and $h$ denotes the bandwidth (see below).
}

\vspace*{0.5cm}

{\em Discussion:} (a) The results of type (i) and (ii) used to derive this main result are of independent interest. Note that Bickel and Rosenblatt (1973) discussed uniform absolute deviation of the univariate kernel density estimator from the density function and developed a confidence band for the density function. Rosenblatt (1976) extended the result to the multidimensional case. Our main result bears some similarity with these results, and we will borrow some ideas from this classical work for the proof of our result. \\[8pt]
(b) Notice that the filament points $\X_{x_0}(\theta_{x_0})$ and $\X_{x_1}(\theta_{x_1})$ are the same if the starting points $x_0$ and $x_1$ both lie on the same integral curve. However, the estimates for these two quantities that correspond to the same starting points are not the same. In other words, when $x_0$ is ranging over a (large) set we will have an entire class of estimates for each filament point! Of course we don't know which of the starting points lie on the same integral curve. However, asymptotically the maximum deviation over all these estimates behaves as if there were only a single starting point from each integral curve. In fact, as it turns out, the extreme value distribution in our main result (see above) only depends on the filament ${\mathcal{L}}$. This dependence is given through the constant $c > 0$ that is completely determined by ${\mathcal{L}}$ (cf. end of the proof of the main result Theorem~\ref{ConfBand}). \\[8pt]
The paper is organized as follows. Sections 2 presents the definition of our estimators. The main results on filament estimation and the estimation of integral curves driven by the second eigenvector of the Hessian are given in Section 3. Specifically, Theorem \ref{ConfBand} precisely states the main result indicated above. The proof uses an application of a limit result on the extreme value distribution of a sequence of non-stationary Gaussian fields on a growing manifold, which is proven in a companion paper by Qiao and Polonik (2015a) (see Theorem~\ref{ProbMain}). Section 3 also contains several other key results needed for the proof of the main result. The logical sequence of the results follows the order of Theorems \ref{Gaussian}, \ref{UniformPath}, \ref{Approx-FilaDiff}, \ref{UniformXDiff}, and then Theorem \ref{ConfBand} with auxiliary Lemma \ref{Theta} in-between. Another consequence of these results is the pointwise asymptotic normality of our filament estimator with rates depending on whether the gradient at the filament point is zero or not (Corollaries~\ref{distr-filament-est1} and \ref{distr-filament-est2}). Section~\ref{summary} presents a summary and some discussion, and all the proofs are delegated to section~\ref{proofs} and the technical supplement (Qiao and Polonik, 2015b), respectively.
\section{Notation and definition of the estimators}
Let $f:\mathbb{R}^2\rightarrow \mathbb{R}^+$ be a four times differentiable probability density function with corresponding cdf $F$, and let $f^{(i,j)}(x) = \frac{\partial^{i+j}f(x)}{\partial{x}^i_1\partial x^j_2}$ for $i, j\in \{0,1,2,\ldots\}$ and $i+j\leq 4$. Then we write the gradient of $f$ as $\nabla f(x) = (f^{(1,0)}(x),f^{(0,1)}(x))^T$, and the Hessian matrix of $f$ as
\vspace*{-0.7cm}

\begin{align*}
\nabla^2f(x)\equiv\left(
                                \begin{array}{cc}
                                  f^{(2,0)}(x) & f^{(1,1)}(x) \\
                                  f^{(1,1)}(x) & f^{(0,2)}(x)\\
                                \end{array}
\right).
\end{align*}
Further write
\begin{align*}
d^2f(x) = \left( f^{(2,0)}(x),\,f^{(1,1)}(x),\,f^{(0,2)}(x)\right)^T.
\end{align*}
Let $V(x)$ denote a second eigenvector of $\nabla ^2 f(x)$, which is assumed to have two distinct eigenvalues on an appropriate subset of the support of $f$. In this paper we will use the specific form of $V(x)$ given by
$$V(x) = G(d^2(f(x))),$$
where $G= (G_1,G_2)^T: \mathbb{R}^3\mapsto \mathbb{R}^2$ is
\begin{align}\label{GDef}
G(u,v,w)=\left(\begin{array}{c} 2u-2w+2v-2\sqrt{(w-u)^2+4v^2} \\w-u+4v-\sqrt{(w-u)^2+4v^2} \end{array}\right).
\end{align}
%
Details of constructing $G$ can be found in the supplemental material (Qiao and Polonik, 2015b). Since $f$ is four times continuously differentiable, $V(x)$ is twice continuously differentiable as long as the eigenvalues of Hessian $\nabla^2f(x)$ are distinct. There are different ways of choosing $V$ via $G$, e.g., $V(x)$ could have norm 1. All that matters here is that $V(x)$ is smooth and that $\|V(x)\|$ is bounded away from zero (and infinity). It is not difficult to verify that $V(x)$ so defined is in fact an eigenvector of the Hessian $\nabla^2 f(x)$ corresponding to its second eigenvalue\\[-20pt] 
\begin{align*}
\lambda_2(x)=J(d^2 f(x))
\end{align*}
where\\[-20pt]
\begin{align}\label{FDef}
J(u,v,w)=\frac{u+w-\sqrt{(u-w)^2+4v^2}}{2}.
\end{align}
%
%
%
%
%
{\bf The estimator of the integral curve $\X_{x_0}(t)$.}  Our estimator $\hat \X_{x_0}(t)$ of the integral curve $\X_{x_0}(t)$  is based on a plug-in estimator of the second eigenvector $V(x)$ of the Hessian, i.e. $\hat \X_{x_0}(t)$ is the solution to
\begin{align*}
\frac{d \hat\X_{x_0}(t)}{d t}=\hat V(\hat \X_{x_0}(t)), \quad \hat\X_{x_0}(0)=x_0,
\end{align*}
where $\hat V(x)$ is defined via a kernel estimator of the density. To be explicit, let $X_1, X_2,\dots$ 
be independent and identically distributed with density function $f$. The kernel density estimator of $f$ based on $X_1, \cdots ,X_n, n\geq 1$ is
\begin{align}
\hat f(x)=\frac{1}{nh^2}\sum_{i=1}^nK\bigg(\frac{x-X_i}{h}\bigg), \quad x\in\mathbb{R}^2,
\end{align}
where $K:\mathbb{R}^2\rightarrow \mathbb{R}^+$ is a four times differentiable kernel function and $h$  is the positive bandwidth (sometimes we also write $h_n$ rather than $h$ to indicate its dependence on $n$).
%
%
%
%
The corresponding plug-in kernel estimators of $V(x)$ and $\lambda_2(x)$ then are
\begin{align}
\hat V(x)=G(d^2\hat f(x)) \; \textrm{and} \; \hat \lambda_2(x)= J(d^2\hat f(x)),\quad x\in\mathbb{R}^2.
\end{align}
%
%
{\bf The estimator of the parameter $\boldsymbol{\theta}_{\mathbf {x_0}}$.}
%
%
Consider a compact set $\mathcal{H}$ such that $f(x) > 0$ on $\mathcal {H}^{\epsilon_0}$ for some $\epsilon_0>0$, where $\mathcal {H}^{\epsilon_0}$ denotes the $\epsilon_0$-enlarged set of ${\cal H}$, i.e. the union of all the open balls of radius $\epsilon_0$ with midpoints in ${\cal H}$. 
Let further $\mathcal{L}$ denote the target filament. For any $a,b\in\mathbb{R}$, we denote $a\wedge b=\min(a,b)$ and $a\vee b = \max(a, b)$. \\


We denote $x_0\rightsquigarrow\mathcal{L}$ if there exists a $t_0\in\mathbb{R}$ with ${\X}_{x_0}(t_0) \in {\cal L}$ and $\{{\X}_{x_0}(t): \; 0 \wedge t_0 \le t \le 0\vee t_0\} \subset {\cal H}$. We define
%
%
\begin{align*}
\Theta_{x_0}=\big\{t: \; \blangle\nabla f(\X_{x_0}(t)), V(\X_{x_0}(t))\brangle = 0,\, \lambda_2(\X_{x_0}(t))<0\big\}.
\end{align*}
For $a^*>0$ let 
\begin{align}\label{DefG}
\mathcal{G}({\cal L}, a^*) := \big\{\X_{x_0}(t): x_0 \in  {\cal L}, -a^* \leq t \leq a^* \big\},
\end{align}
which is assumed to be a subset of $\cal H$. For simplicity we write $\mathcal{G}({\cal L}, a^*)$ as $\mathcal{G}$. By definition of $x_0 \rightsquigarrow {\cal L}$, for $x_0 \in \mathcal{G}$ we have $\Theta_{x_0}\neq \emptyset$ and let
\begin{align}\label{theta_def}
\theta_{x_0} = \argmin_{t}\{|t|: \; t \in \Theta_{x_0}\}.
\end{align}
Suppose that we can choose $a^*$ such that $\theta_{x_0}$ is unique for any $x_0\in\mathcal{G}$. Note that $\mathcal{G}$ and $\mathcal{L}$ have such a relationship that
$\mathcal{L} = \{\X_{x_0}(\theta_{x_0}): \;\; x_0 \in \mathcal{G}\}$. This means that for $x_0 \in \mathcal{G}$ when traversing the path $\X_{x_0}$ we hit the filament ${\cal L}$ for the `first' time at `time' $|\theta_{x_0}|$. The estimator of $\theta_{x_0}$ is denoted by $\hat \theta_{x_0}$ and is defined as follows. Let
\begin{align*}
\wh \Theta_{x_0}=\big\{t: \; \blangle\nabla \hat f(\hat \X_{x_0}(t)), \hat V(\hat \X_{x_0}(t))\brangle = 0,\, \hat\lambda_2(\hat \X_{x_0}(t))<0\big\},
\end{align*}
and define\\[-25pt]
\begin{align} \label{def-thetahat}
\hat\theta_{x_0}=\begin{cases}
 \argmin_{t}\{|t|: \; t \in \wh \Theta_{x_0}\}, & \text{if } \wh \Theta_{x_0}\neq \emptyset\\
 0 & \text{if } \wh \Theta_{x_0}= \emptyset.
\end{cases}
\end{align}
If the minimizer here is not unique, then we just choose one of them as $\hat{\theta}_{x_0}.$ The probability of this happening is tending to zero under our assumptions. These assumptions also assure  that the probability of $\wh \Theta_{x_0}\neq \emptyset$ is tending to zero as $n\to \infty$ for $x\in \mathcal{G}$ (see Proposition~\ref{ThetaRate}).\vspace*{0.5cm} 

{\bf The estimator of a filament point  $\boldsymbol{\X_{x_0}(\theta_{x_0})}$} with $x_0 \rightsquigarrow {\cal L}$ is now given by
$$\hat\X_{x_0}(\hat{\theta}_{x_0}).$$
Our filament points (both estimates and theoretical) are parametrized by the starting value of the integral curves. In fact, we should rather think of the parametrization being induced by the corresponding integral curves, because any starting point on the same integral curve of course results in the same filament point. Since for each of the filament points there is exactly one integral curve passing through this point, this provides us a way to make pointwise comparisons. Our estimator of $\mathcal{L}$ is given by
\begin{align*}
\widehat{\mathcal{L}} = \{\hat \X_{x_0}(\hat\theta_{x_0}): x_0 \in \mathcal{G} \}.
\end{align*}
To formulate our main theorem we need the following additional notation and definitions. For a matrix $M$ and compatible vectors $v,w$ we denote $\langle v,w\rangle_M = v^T M w$. We further write $\|v\|^2_M = v^T M v$, which for $M$ the identity matrix is simplified to  $\|v\|^2$. For a vector field $W: \mathbb{R}^2\mapsto\mathbb{R}^3$ let $\mathbf{R}(W)$ denote the matrix given by $\mathbf{R}(W):=\int_{\mathbb{R}^2}W(x)W(x)^T dx \in \bbR^{3\times 3}$, assuming the integral is well defined, and let $\mathbf{R}:=\mathbf{R}(d^2 K).$ Further let
\begin{align}
\tilde G(x)&:=\nabla G(d^2 f(x)) \in \bbR^{2 \times 3} \qquad\text{and}\qquad A(x)=\tilde{G}(x)^T  \nabla f(x) \in \bbR^3,\label{GtildeDef}
\end{align}
and define the real-valued function $g(x)$ as
\begin{align}\label{def-g}
g(x) =  \frac{\tilde{a}^\prime(x)}{\sqrt{f(x)}\|V(x)\|\,\|A(x)\|_{\mathbf R}},
\end{align}
where\\[-25pt]
\begin{align}\label{def-atildeprime}
\tilde{a}^\prime(x) = \blangle\nabla f(x), V(x)\brangle_{\nabla V ( x)}+\lambda_2(x) \|V(x)\|^2.
\end{align}
Observe that $\tilde{a}^\prime(\X_{x_0}(\theta_{x_0})) = \frac{d}{dt}\, a_{x_0}(\theta_{x_0})$ with $a_{x_0}(t) = \blangle\nabla f(\X_{x_0}(t)), V(\X_{x_0}(t))\brangle.$ These quantities describe the behavior of $f(\X_{x_{0}}(t))$ at $t = \theta_{x_0}$, and thus they play an important role here.
Our assumptions given below assure that $g(x)$ is well defined on $\mathcal{H}$.

\section{Main Results}


\subsection{Assumptions and their discussion}\label{Assumptions}

\begin{itemize}
\item [(\bf{F}1)] $f$ is a four times continuously differentiable pdf. All of its first to fourth order partial derivatives are bounded.
\item[(\textbf{F}2)] $\mathcal {H}$ is compact such that $f(x) > 0$ on $\mathcal {H}^{\epsilon_0}$ for some $\epsilon_0>0$ and $\nabla^2f(x)$ has two distinct eigenvalues for $x\in\cal H$.  
%
%
\item[(\textbf{F}3)] $\mathcal{L}$ is a compact filament within $\mathcal{H}$ with $\mathcal{L} = \{\X_{x_0}(\theta_{x_0}): \;\; x_0 \in \mathcal{G}\}$, where $\theta_{x_0}$ is defined in (\ref{theta_def}) and $\mathcal{G}$ defined in (\ref{DefG}) is a subset of $\mathcal{H}$. We choose $a^*$ in (\ref{DefG}) such that $\theta_{x_0}$ is unique for any $x_0\in\mathcal{G}$.
%
\item[(\textbf{F}4)] There exists a $\gamma>0$ such that\\[-20pt] 
\begin{align*}
\inf_{x_0\in \mathcal{L}}\;\;\inf_{-a^*\leq s< u\leq a^*}\bigg\|\frac{1}{u-s}\int_s^uV(\X_{x_0}(\lambda))d\lambda\bigg\|\geq\gamma.
\end{align*}

%
%
\item[(\textbf{F}5)] $\blangle \nabla \blangle\nabla f(x), V(x)\brangle, \; V(x)\brangle \neq 0$ for all $x\in \mathcal {L}$.
\item[(\textbf{F}6)] $\{x\in\mathcal{H}: \lambda_2(x)=0, \blangle\nabla f(x), V(x)\brangle=0\}=\emptyset$.
\item[(\textbf{F}7)] $\nabla f(x)^T \tilde{G}(x)\neq0$ for $x\in{\cal L}$.\\[-10pt]
\item[(\textbf{K}1)] The kernel K is a symmetric probability density function with support being the unit ball in $\bbR^2$. All of its first to fourth order partial derivatives are bounded and $\int_{\mathbb{R}^2}K(x)xx^T dx=\mu_2(K){\rm I}_{2 \times 2}$ with $\mu_2(K)<\infty$.
\item[(\textbf{K}2)] $\mathbf{R}(d^2 K)<\infty$ where for $g: \mathbb{R}^2\mapsto\mathbb{R}^3$, $\mathbf{R}(g):=\int_{\mathbb{R}^2}g(x)g(x)^T dx$.
%
%
\item[(\textbf{K}3)] $\int [K^{(3,0)}(z)]^2dz\neq\int [K^{(1,2)}(z)]^2dz$.
\item[(\textbf{K}4)] For any open ball $\cal S$ with positive radius contained in $\mathscr{B}(0,1)$ the component functions of $\textbf{1}_\mathcal{S}(s)d^2 K(s)$ are linearly independent. \\[-10pt]

\item[(\textbf{H}1)] As $n\rightarrow0$, $h_n \downarrow 0$, $nh_n^8/(\log n)^3\rightarrow\infty$ and $nh_n^9\rightarrow\beta$ for some $\beta\geq0$.
\end{itemize}

\vspace*{0.2cm}

{\bf Discussion of the assumptions.}\\[7pt]
1. Assumption (\textbf{F}1) implies that $V(x)$ is Lipschitz continuous on $\mathbb{R}^2$. Since $\blangle\nabla f(x), V(x)\brangle=0$ on a filament, $\nabla \blangle\nabla f(x), V(x)\brangle$ provides a direction normal to the filament. Therefore Assumption (\textbf{F}1) implies that $\mathcal{L}$ is twice differentiable and has bounded curvature.\\[7pt]
2. Assumptions (\textbf{F}2) and (\textbf{F}6) are imposed to avoid the existence of ``degenerate'' filament points. Specifically,  assumption (\textbf{F}6) ensures the exclusion of points at which the first and second order directional derivatives of $f(x)$ along $V(x)$ are both zero. By assumption (\textbf{F}2), there exists a $\delta>0$ such that $\{d^2 f(x): x\in {\cal H}\}\subset \mathcal {Q}_{\delta}$, where
\begin{align}\label{Q-def}
\mathcal{Q}_{\delta}=\{(u,v,w)\in\mathbb{R}^3:|u-w|>\delta\;\; \textrm{or} \;\; |v|>\delta\},
\end{align}
since two eigenvalues of a $2\times2$ symmetric matrix are equal iff the matrix is a scaled identity matrix.\\[7pt]
3. (\textbf{F}7) in particular excludes flat parts on the filaments, i.e. $\|\nabla f(x)\| \ne 0,$ for $x \in {\cal L}$.\\[7pt]
4. The set  ${\mathcal{G}}$ defined in (\ref{DefG}) denotes the set of starting points of the integral curves, each of which uniquely corresponds to a filament point on ${\mathcal{L}}$. The uniqueness follows from the well-known fact that integral curves are non-overlapping except possibly at their endpoints, and our assumptions exclude the latter case. The set $\mathcal{G}$ is compact, because $[-a^*, a^*]\times \mathcal{L}$ is compact by ({\textbf F}3) and that the mapping $(t, x_0) \mapsto \X_{x_0}(t)$ is continuous as shown in the technical supplement (Qiao and Polonik, 2015b). Note that the choice of $a^*$ does not affect the asymptotic distribution result in our main theorem (cf. Theorem \ref{ConfBand}).\\[7pt]
%
%
5. Since $\{\X_{x_0}(s): \; x_0 \in \mathcal{G},\theta_{x_0} - a^*\leq s \leq \theta_{x_0} + a^*\} = {\cal G}$, the two sets $\{\X_{x_0}(s): \; x_0\in {\cal G},\; \theta_{x_0}-a^*\leq s \leq \theta_{x_0}+a^*\}$ and $\{\X_{x_0}(s): \; x_0\in {\cal L},\; -a^*\leq s \leq a^*\}$ are equal. Therefore assumption (\textbf{F}4) is equivalent to
\begin{align*}
\inf_{x_0\in \mathcal{G}}\;\;\inf_{\theta_{x_0}-a^*\leq s< u\leq \theta_{x_0}+a^*}\bigg\|\frac{1}{u-s}\int_s^uV(\X_{x_0}(\lambda))d\lambda\bigg\|\geq\gamma.
\end{align*} 
It will be satisfied, for instance, under the condition that the convex hull of $\cal G$ is a subset of $\cal H$. An assumption similar to ({\bf F}4) can also be found in Koltchinskii et al. (2007).\\[7pt] 
%
%
%
%
%
6. The geometric meaning of assumption (\textbf{F}5) is that the second eigenvector $V(x)$ of the Hessian $H(x)$ is not orthogonal to the normal direction at the filament, which is represented by $\nabla\blangle\nabla f(x), V(x)\brangle$. Assumption (\textbf{F}5) implies that $\blangle \nabla f( \X_{x_0}(t)), V( \X_{x_0}(t))\brangle$ as a function of $t$ is strictly monotone at $\theta_{x_0}$, i.e. it changes signs at $\theta_{x_0}$.\\[7pt]
7. Assumption (\textbf{K}4) means that there is no linear combination of the component functions of $d^2 K(s)$ whose roots constitute a set of positive Lebesgue measure. A kernel function $K$ satisfying assumptions (\textbf{K}1)--(\textbf{K}4) is given by
\begin{align*}
K(z)=\frac{6}{\pi}(1-\|z\|^2)^5\textbf{1}_{\mathscr{B}(0,1)}(z), \quad z\in\mathbb{R}^2.
\end{align*}
Let $z=(z_1,z_2)^T$. Then assumption (\textbf{K}4) can be verified by observing that
\begin{align*}
d^2K(z)=\frac{15}{\pi}(1-z_1^2-z_2^2)^3 \left(\begin{array}{c}
9z_1^2+z_2^2-1 \\
8z_1z_2-2 \\
z_1^2+9z_2^2-1\\
\end{array}\right).
\end{align*}
%
%
8. Below we study the properties of the kernel $K$ under the given assumptions. First note that by the symmetry of $K(\cdot)$ we have
\begin{align}
\int [K^{(2,1)}(z)]^2dz&=\int [K^{(1,2)}(z)]^2dz,\label{KEquality1}\\
\int [K^{(3,0)}(z)]^2dz&=\int [K^{(0,3)}(z)]^2dz.\label{KEquality2}
\end{align}
Denote $I(\{c_1,c_2\}, \{c_3,c_4\}):=\int K^{(c_1,c_2)}(z)K^{(c_3,c_4)}(z)dz$. Using integration by parts and assumption (\textbf{K}1), the value of $I(\{c_1,c_2\},\{c_3,c_4\})$ is equal to the value of the integrals in (\ref{KEquality1}) for  $(\{c_1,c_2\}, \{c_3,c_4\}) \in \{ (\{4,0\},\{0,2\}),$\, $(\{3,1\},\{1,1\}),$\, $(\{2,2\}, \{0,2\}) \}$, and $I(\{4,0\},\{2,0\})$ equals the value of the integrals in (\ref{KEquality2}). \\[7pt]
9. By standard arguments, for the second derivatives of the density the bias of the kernel estimator is of order $O(h^2)$, which under assumption (\textbf{H}1) is faster than $O_p\big(\frac{\log{n}}{nh^6}\big)$, i.e., the convergence rate of the stochastic part. Therefore the bias is absorbed into the stochastic variation, and the rate of the former does not appear in our theorems. 
%

\subsection{Filament estimation}

We first present our main result on filament estimation, which gives the asymptotic distribution of the uniform absolute deviation of the estimator of the filament from the target filament that is assumed to exist under our set-up. This main result is in the same spirit as the classical results by Bickel and Rosenblatt (1973) and Rosenblatt (1976) for kernel density estimates.
\begin{theorem}\label{ConfBand}
Suppose that $({\mathbf F}1) - ({\mathbf F}7), ({\mathbf K}1) - ({\mathbf K}4)$, and $({\mathbf H}1)$ hold. Then there exists a constant $c \in \bbR$ depending on $K,f$ and ${\mathcal{L}}$ such that for any $z \in \bbR$ we have with
\begin{align}\label{CDef}
b_h(z)=\sqrt{2\log{h^{-1}}}+\frac{1}{\sqrt{2\log{h^{-1}}}}\big[z+c\big],
\end{align}
that as $n \to \infty$
\begin{align}\label{ConfBand1}
P\bigg(\sup_{x_0\in\mathcal{G}}\big|g(\X_{x_0}(\theta_{x_0}))\big| \bigg\|\sqrt{nh^6}\Big(\hat \X_{x_0}(\hat \theta_{x_0})-\X_{x_0}(\theta_{x_0})\Big)\bigg\|<b_h(z)\bigg) \to e^{-2\,e^{-z}}.
\end{align}
\end{theorem}
Notice that in particular the assumptions of this theorem assure that there is no flat part on the filament. The dependence of $c$ on $K,f$ and ${\mathcal{L}}$ is made explicit in the proof of Theorem~\ref{ConfBand} given in section~\ref{proofs}.
As already indicated, to prove Theorem \ref{ConfBand} we approximate the supremum distance between $\hat \X_{x_0}(\hat \theta_{x_0})$ and $\X_{x_0}(\theta_{x_0})$ by a supremum of a Gaussian random field over the rescaled filaments ${\mathcal{L}}_h = \{ x: xh \in {\mathcal{L}}\}$ as $h \to 0.$ The proof combines results for estimating the trajectory of the integral curve and the estimation of the parameter value at the filament points when traveling along the integral curve. These results, which are of independent interest, will be discussed in the following subsections. \\[8pt]
%
%
It will turn out that the estimation of  the integral curves can be accomplished at a faster rate than the estimation of the location of the mode along the integral curve, and so it is the latter that is determining the rate in Theorem~\ref{ConfBand}. It perhaps is not a surprise that the estimation of the trajectory itself turns out to be negligible, as the property of being a maximizer/maximum is of local nature. In other words, in our approach the estimation of the integral curves only serves as a means to an end. We will further see in the next section that for each fixed $x_0$ the deviation $\hat\theta_{x_0} - \theta_{x_0}$ can be approximated by a linear function of the deviations of the second derivatives of the kernel estimator (see Theorem~\ref{Theta}). Since under our assumptions we can estimate second derivatives of $f$ by the standard rate $\sqrt{nh^{d+4}} = \sqrt{nh^6},$ this then explains the normalizing factor in Theorem~\ref{ConfBand}. In fact, Genovese et al. (2014) derived the rate $O_p\big(\big(\frac{\log n}{nh^6}\big)^{1/2}\,\big) + O(h^2)$ for the Hausdorff distance between a filament and its kernel estimate, where $O(h^2)$ accounts for the rate of the bias term, which can be absorbed into $O_p\big(\big(\frac{\log n}{nh^6}\big)^{1/2}\,\big)$ under out assumption ({\bf H}1). Notice that $\sup_{x_0\in\mathcal{G}}\big\|\hat \X_{x_0}(\hat \theta_{x_0})-\X_{x_0}(\theta_{x_0})\big\|$ gives an upper bound on the Hausdorff distance between the sets ${\cal L}$ and $\widehat{\cal L}$. \\[8pt]
The proof of Theorem~\ref{ConfBand} requires the derivation of several results that are interesting in their own right. These results provide further insight about the behavior of our filament estimation approach and they also provide more details about the deviation $\hat \X_{x_0}(\hat \theta_{x_0})-\X_{x_0}(\theta_{x_0})$. In fact, if we decompose this deviation into the projection orthogonal to the filament, i.e. the projection onto $V(\X_{x_0}(\theta_{x_0}))$, and the projection onto $V^{\bot}\big( \X_{x_0}( \theta_{x_0})\big)$, then we will see that under the assumptions of the above theorem the estimation of the latter is asymptotically negligible. The key assumption here is that the filament has no flat part. Without this assumption the two projections (and thus the deviation of the filament itself) both are of the order $O_p(1/\sqrt{nh^5})$ (cf. Corollary~\ref{distr-filament-est2}).

\subsection{Estimation of integral curves}

This section discusses the estimation of the integral curve $\X_{x_0}(t)$. We will adapt the method from Koltchinskii et al. (2007) to our case. Koltchinskii et al. assume the availability of iid observations $(W_i,X_i)$ following the regression model $W_i = V(X_i) + \epsilon_i$ with $X_i$ and $\epsilon_i$ independent. In contrast to that, our model assumes the underlying vector field to be given by the eigenvector of the Hessian of a density $f$, and we have available iid $X_i$'s from $f$. Our first result considers the estimation of a single trajectory (i.e. we fix the starting point).

\begin{theorem}
\label{Gaussian}
Under assumptions (\textbf{F}1)--(\textbf{F}2), (\textbf{K}1)--(\textbf{K}2) and (\textbf{H}1), for any $x_0\in \mathcal{G}$, $0 < \gamma < \infty$ and $0 \leq T_{min},T_{max}< \infty,\,T_{min} + T_{max} \ne 0$ with $\{\X_{x_0}(t),\,t \in [-T_{min},T_{max}]\} \subset {\cal H}$ and
\begin{align}\label{Gammax0}
\inf_{-T_{min}\leq s< u\leq T_{max}}\bigg\|\frac{1}{u-s}\int_s^uV(\X_{x_0}(\lambda))d\lambda\bigg\|\geq\gamma,
\end{align}
the sequence of stochastic process $\sqrt{nh^5}(\hat \X_{x_0}(t)-\X_{x_0}(t))$, $-T_{min}\leq t\leq T_{max}$, converges weakly in the space $C[-T_{min},T_{max}]:=C([-T_{min},T_{max}],\mathbb{R}^2)$ of $\mathbb{R}^2$-valued continuous functions on $[-T_{min},T_{max}]$ to the Gaussian process $\omega(t), -T_{min}\leq t\leq T_{max}$, satisfying the SDE
\begin{align}\label{SDEMain}
d\omega(t)=&\frac{\sqrt{\beta}}{2}\tilde G(\X_{x_0}(t))v(\X_{x_0}(t))dt+\nabla V(\X_{x_0}(t))\omega(t)dt\nonumber\\
&\hspace{-1cm}+\bigg\{\tilde G(\X_{x_0}(t))\bigg[\int\int{\mathbb K}(\X_{x_0}(t),\tau,z)f(\X_{x_0}(t))dzd\tau\bigg]\tilde G(\X_{x_0}(t))^T \bigg\}^{1/2}dW(t)
\end{align}
with initial condition $\omega(0)=0$, where $W(t),t\geq0$ is a two-sided standard Brownian motion in $\mathbb{R}^2$,
\begin{align}
v(x) &=
\left(
\begin{array}{c}
\int K(z)z^T \nabla^2f^{(2,0)}(x)zdz \\
\int K(z)z^T \nabla^2f^{(1,1)}(x)zdz \\
\int K(z)z^T \nabla^2f^{(0,2)}(x)zdz \\
\end{array}
\right) \in \bbR^3, \label{BDefinition}\\[-15pt]
\intertext{and}
{\mathbb K}(x,\tau,z) &:=d^2 K(z)\big[d^2 K\big(\tau V(x)+z\big)\big]^T \in \bbR^{3 \times 3}.\label{PsiDefinition}
\end{align}
\end{theorem}
The proof of Theorem \ref{Gaussian} that can be found in the supplementary material (Qiao and Polonik, 2015b) is following ideas from Koltchinskii et al. (2007). We can see that $\sqrt{nh^5}$ is the appropriate normalizing factor under the assumption of the theorem. The heuristic behind that rate is given by the fact that the integral curve satisfies the integral equation\\[-10pt]
\begin{align*}
\X_{x_0}(t) = \int_0^t V(\X_{x_0}(s))\,ds + x_0.\\[-20pt]
\end{align*}
Our estimator $\hat{\X}_{x_0}(t)$ satisfies the similar equation with $V$ replaced by $\hat V$, and thus we have
\begin{align}\label{heuristic}
\hat{\X}_{x_0}(t) - \X_{x_0}(t) &= \int_0^t \big[\hat{V}(\hat \X_{x_0}(s)) - V(\X_{x_0}(s))\big]\,ds \nonumber\\
&
\approx  \int_0^t \big[\hat{V}(\hat \X_{x_0}(s)) - V(\hat \X_{x_0}(s))\big]\,ds.
\end{align}
Heuristically, the indicated approximation holds because the remainder term $\int_0^t [V(\hat\X_{x_0}(s)) - V(\X_{x_0}(s))]ds$ roughly behaves like the integrated difference  $\hat{\X}_{x_0}(t) - \X_{x_0}(t),$ which is of smaller order than $\hat{\X}_{x_0}(t) - \X_{x_0}(t)$ itself. Therefore we get from (\ref{heuristic}) that the rate of convergence of $\hat{\X}_{x_0}(t) - \X_{x_0}(t) $ is essentially determined by the integrated difference $\hat{V}(x) - V(x)$. Since $\hat{V}$ is a function of the second derivatives of the density estimator we obtain a standard rate of $1/\sqrt{nh^6}$ for the difference $\hat{V}(x) - V(x)$, and through integrating we gain one power of $h$, justifying the normalizing factor $\sqrt{nh^5}.$ The above theorem implies that as $n\rightarrow\infty$,
\begin{align*}
\sup_{t\in[- T_{min},T_{max}]}\|\hat\X_{x_0}(t)-\X_{x_0}(t)\|=O_p\bigg(\frac{1}{\sqrt{nh^5}}\bigg).
\end{align*}
In the next theorem we consider the behavior of $\hat\X_{x_0}(t)-\X_{x_0}(t)$ not only uniformly in $t$ but also uniformly in the starting point $x_0$.
\begin{theorem}\label{UniformPath}
Suppose for any $x_0\in \mathcal{G}$ there exist $T_{x_0}^{min}, T_{x_0}^{max} \geq 0$ with $T_{x_0}^{min} + T_{x_0}^{max} > 0$ such that $T_{x_0}^{min}$ and $T_{x_0}^{max}$ are continuous in $x_0 \in \mathcal{G},$ and $\{\X_{x_0}(t),\,t \in [-T_{x_0}^{min}, T_{x_0}^{max}]\} \subset {\cal H}$ for all $x_0 \in \mathcal{G}$. Further assume that for some $\gamma_{\mathcal{G}}>0$
\begin{align}\label{Gammax1}
\inf_{x_0\in \mathcal{G}} \; \inf_{-T_{x_0}^{min}\leq s< u\leq T_{x_0}^{max}}\bigg\|\frac{1}{u-s}\int_s^uV(\X_{x_0}(\lambda))d\lambda\bigg\|\geq\gamma_{\mathcal{G}}.
\end{align}
Then under assumptions (\textbf{F}1)--(\textbf{F}2), (\textbf{K}1)--(\textbf{K}2) and (\textbf{H}1),
\begin{align*}
\sup_{x_0\in \mathcal{G}} \; \sup_{t\in[-T_{x_0}^{min}, T_{x_0}^{max}]}\|\hat \X_{x_0}(t)-\X_{x_0}(t)\|=O_p\bigg(\sqrt{\frac{\log{n}}{nh^5}}\bigg).
\end{align*}
\end{theorem}
\subsection{Pointwise asymptotic distribution of filament estimates}

Our goal here is to find the pointwise asymptotic distribution of $\hat\X_{x_0}(\hat{\theta}_{x_0}) - \X_{x_0}(\theta_{x_0})$, the difference of the `true' and the estimated filament points corresponding to integral curves starting at $x_0$. To this end, we first approximate $\hat\X_{x_0}(\hat{\theta}_{x_0}) - \X_{x_0}(\theta_{x_0})$ by a linear function of the difference $\hat \theta_{x_0}-\theta_{x_0}$. Thus, finding good approximations for $\hat\X_{x_0}(\hat{\theta}_{x_0}) - \X_{x_0}(\theta_{x_0})$ can be accomplished by finding good approximations for $\hat \theta_{x_0}-\theta_{x_0}$. 
\begin{theorem}\label{Approx-FilaDiff}
If (\textbf{F}1)--(\textbf{F}6), (\textbf{K}1)--(\textbf{K}2) and (\textbf{H}1) hold, then
\begin{align*}
\sup_{x_0 \in \mathcal{G}}\big\| \big[\hat \X_{x_0}(\hat\theta_{x_0})-\X_{x_0}(\theta_{x_0})\big] -  V(\X_{x_0}(\theta_{x_0})) [\hat\theta_{x_0}-\theta_{x_0}] \big\| = O_p\Big(\textstyle{ \sqrt{\frac{\log{n}}{nh^5}}}\Big).
\end{align*}
\end{theorem}
Now we will utilize this approximation by deriving good approximations for  $\hat\theta_{x_0}-\theta_{x_0}$, which then lead to good approximations for $\hat\X_{x_0}(\hat{\theta}_{x_0}) - \X_{x_0}(\theta_{x_0}).$ Define
\begin{align}\label{def-a1n}
\hat \varphi_{1n}(x) &= \tfrac{1}{\tilde a^\prime(x)} \;\blangle \nabla f(x), \;d^2\hat f(x)-\mathbb{E}d^2\hat f(x)\brangle_{\tilde{G}(x)}\\
\hat \varphi_{2n}(x) &= \tfrac{1}{\tilde a^\prime(x))}\,\Big[\blangle V(x), \hat \X_{x_0}(\theta_{x_0})-x\brangle_{\nabla^2 f(x)} + \blangle({\mathbb E}\nabla \hat f - \nabla f)(\hat\X_{x_0}(\theta_{x_0})),  V(x)\brangle\Big] , \label{def-a2n}
\end{align}
where $\tilde a^\prime(x)$ and $\tilde{G}(x)$ are as in (\ref{def-atildeprime}) and (\ref{GtildeDef}), respectively. Notice that for each fixed $x$, the term $\hat \varphi_{1n}(x)$ is a linear function of the second derivatives of $\hat f$. The following result shows that $- \hat \varphi_{1n}(\X_{x_0}(\theta_{x_0}))$ serves as a good approximation of  $\hat \theta_{x_0}-\theta_{x_0}$. If $\|\nabla f(\X_{x_0}(\theta_{x_0}))\| = 0$ then $\hat \varphi_{1n}(\X_{x_0}(\theta_{x_0})) = 0$, and a better approximation is provided by $\hat \varphi_{2n}(\X_{x_0}(\theta_{x_0}))$, and we also have control over the exact asymptotic behavior of this approximation, mainly due to Theorem~\ref{Gaussian}.
\begin{lemma}\label{Theta}
Under assumptions (\textbf{F}1)--(\textbf{F}6), (\textbf{K}1)--(\textbf{K}2) and (\textbf{H}1), we have
\begin{align}
\sup_{x_0\in\mathcal{G}}\big|\big(\hat \theta_{x_0}-\theta_{x_0}\big)+ \hat \varphi_{1n}(\X_{x_0}(\theta_{x_0}))\big| =O_p\bigg(\frac{\log{n}}{nh^7}\bigg).\label{ThetaNormal}
\end{align}
If in addition, $\sup_{x_0\in\mathcal{G}}\|\nabla f(\X_{x_0}(\theta_{x_0}))\|=0$, then 
\begin{align}\label{ThetaPart}
\sup_{x_0\in\mathcal{G}}\big|\big( \hat \theta_{x_0}-\theta_{x_0}\big)+ \hat \varphi_{2n}(\X_{x_0}(\theta_{x_0}))  \big|=O_p\bigg(\frac{\log n}{n h^{\frac{13}{2}}\,}\,\bigg).
\end{align}
%
\end{lemma}
%
%
Note that under standard assumptions, both $\hat \varphi_{1n}(\X_{x_0}(\theta_{x_0}))$ and $\hat \varphi_{2n}(\X_{x_0}(\theta_{x_0})) $ become asymptotically normal. Due to Theorem~\ref{Approx-FilaDiff} this property will then translate to the asymptotic normality of $\hat \X_{x_0}(\hat\theta_{x_0})-\X_{x_0}(\theta_{x_0})$ (see below).\\


First we provide a uniform large sample approximation of the estimator of the filament point $\hat \X_{x_0}(\hat\theta_{x_0})$ from its target $\X_{x_0}(\theta_{x_0}).$ The result provides further insight into the behavior of our filament estimator. It is an immediate consequence of Theorem~\ref{Approx-FilaDiff} and Lemma~\ref{Theta}.


\begin{theorem}\label{UniformXDiff}
Under assumptions (\textbf{F}1)--(\textbf{F}6), (\textbf{K}1)--(\textbf{K}2) and (\textbf{H}1): 
\begin{align}\label{XDiffApp2}
\sup_{x_0\in\mathcal{G}}\|\hat \X_{x_0}(\hat\theta_{x_0})-\X_{x_0}(\theta_{x_0})+ \hat \varphi_{1n}(\X_{x_0}(\theta_{x_0}))V(\X_{x_0}(\theta_{x_0}))\|=O_p\Big(\frac{\log{n}}{nh^7}\Big).
\end{align}

\vspace*{-0.5cm}

If in addition, $\sup_{x_0\in\mathcal{G}}\|\nabla f(\X_{x_0}(\theta_{x_0}))\|=0$, then
\begin{align}\label{XDiffApp}
\sup_{x_0\in\mathcal{G}}\big\|\,\hat \X_{x_0}(\hat\theta_{x_0})-\X_{x_0}(\theta_{x_0})+&\Gamma(\theta_{x_0})\,(\hat \X_{x_0}(\theta_{x_0}) - \X_{x_0}(\theta_{x_0})) +\nonumber \\ 
&\hspace*{2cm}P_V(\theta_{x_0})\hat{b}\,\big\|=O_p\bigg(\frac{\log{n}}{nh^{\frac{13}{2}}}\bigg).
\end{align}
where $\hat{b} = \big({\mathbb E}\nabla \hat f - \nabla f\big)(\hat\X_{x_0}(\theta_{x_0})),$ $P_V(t) = V(\X_{x_0}(t))V(\X_{x_0}(t))^T$ and
\begin{align*}
\Gamma(t):= \big(\tilde{a}^\prime(\X_{x_0}(t))\big)^{-1} P_V(t)\nabla^2 f(\X_{x_0}(t))-{\rm I}_{2 \times 2} \;\in \;\;\bbR^{2\times 2}.
\end{align*}
\end{theorem}


 Recall that $\hat \varphi_{1n}(x)$ is a real-valued random variable. Thus (\ref{XDiffApp2}) says that the asymptotic distribution of $\hat \X_{x_0}(\hat\theta_{x_0})-\X_{x_0}(\theta_{x_0})$ is degenerate, concentrating on the one-dimensional linear subspace spanned by $V(\X_{x_0}(\theta_{x_0})).$ Also note that the approximating quantity in Theorem~\ref{UniformXDiff} only depends on the filament points. This then implies that the extreme value distribution of $\hat \X_{x_0}(\hat\theta_{x_0})-\X_{x_0}(\theta_{x_0})$ over all $x_0 \in \mathcal{G}$ in fact only depends on the filament $\mathcal{L}$ rather than $\mathcal{G}$ (cf. Theorem~\ref{ConfBand}). Moreover, the approximations given in the above theorem imply exact rates of convergence for our filament estimates for fixed $x_0$. The following corollary makes this precise.

\begin{corollary}\label{distr-filament-est1}
Under assumptions (\textbf{F}1)--(\textbf{F}6), (\textbf{K}1)--(\textbf{K}2) and (\textbf{H}1),
for every fixed $x_0 \in {\mathcal{G}}$
\begin{align*}
\sqrt{nh^6}[\hat \X_{x_0}(\hat\theta_{x_0})-\X_{x_0}(\theta_{x_0})]\rightarrow Z(\X_{x_0}(\theta_{x_0}))V(\X_{x_0}(\theta_{x_0})),
\end{align*}
where $Z(\X_{x_0}(\theta_{x_0}))$ is a mean zero normal random variable with variance
\begin{align*}
f(\X_{x_0}(\theta_{x_0}))\|W(\X_{x_0}(\theta_{x_0}))\|^2_{\mathbf{R}},
\end{align*}
where $ W(x) = \big(\tilde{a}^\prime(x)\big)^{-1}\tilde G(x)\nabla f(x)^T \in \bbR^3.$ 
\end{corollary}

Note that when $\|\nabla f(\X_{x_0}(\theta_{x_0}))\| = 0$,  the variance of $Z(\X_{x_0}(\theta_{x_0}))$ is zero, and thus the limit in the above corollary is degenerate. In this case we have the following result:

\begin{corollary}\label{distr-filament-est2}
Suppose that the assumptions from Corollary~\ref{distr-filament-est1} hold, and $\|\nabla f(\X_{x_0}(\theta_{x_0}))\| = 0$. Then there exists $m(\theta_{x_0}) \in \bbR^2$ and $\Sigma(\theta_{x_0}) \in \bbR^{2\times 2}$ such that with $\beta$ from assumption ({\textbf H}1),
\begin{align*}
\sqrt{nh^5}[\hat \X_{x_0}(\hat\theta_{x_0})-\X_{x_0}(\theta_{x_0})]\rightarrow \mathscr{N}\Big(\mu_{\theta_{x_0}}, \sigma^2_{\theta_{x_0}} \Big),
\end{align*}
with $\sigma^2_{\theta_{x_0}} = \Gamma(\theta_{x_0})\Sigma(\theta_{x_0})\Gamma(\theta_{x_0})^T$ and $\mu_{\theta_{x_0}} = -\Gamma(\theta_{x_0})\,m(\theta_{x_0}) - \beta P_V (\theta_{x_0})b(\theta_{x_0})$ where $b(t) = \frac{1}{2} \mu_2(K)\,\big(\big(f^{(3,0)} +f^{(1,2)}\big)(\X_{x_0}(t)),\,\big(f^{(0,3)}+ f^{(2,1)}\big)(\X_{x_0}(t))\big)^T$ with $\mu_2(K)$ from assumption ({\textbf K}1), and $\Gamma(\theta_{x_0})\in \bbR^{2 \times 2}$ as given in Theorem~\ref{UniformXDiff}.
\end{corollary}
The final corollary implies that the projection on the tangent direction to the filament, i.e. onto $V^\bot(\X_{x_0}(\theta_{x_0})),$ is of smaller order than the projection on the direction orthogonal to the filament (assuming that the filament is not flat at this point).

\begin{corollary}\label{filament-est3} Suppose that the assumptions of Corollary~\ref{distr-filament-est1} hold. Then
\begin{align}\label{dev-proj-grad}
\sup_{x_0\in\mathcal{G}}\Big|\blangle\hat \X_{x_0}(\hat\theta_{x_0})-\X_{x_0}(\theta_{x_0}),  V^\bot(\X_{x_0}(\theta_{x_0}))\brangle\Big| = O_p\Big(\frac{\log n}{nh^{7}} \Big).
\end{align}
\end{corollary}

\section{Summary and outlook}\label{summary}

In this paper we consider the nonparametric estimation of filaments. We compare the estimated point on a filament $\hat{\X}_{x_0}(\hat{\theta}_{x_0})$ obtained by following an estimated integral curve $\hat{\X}_{x_0}(t)$ with starting point $x_0$ to the corresponding population quantity $\X_{x_0}(\theta_{x_0}).$ Here $\X_{x_0}$ is an integral curve driven by the second eigenvector of the Hessian of the underlying pdf $f$, and $\hat{\X}_{x_0}$ its plug-in estimate obtained by using a kernel estimator. Our main result derives the exact asymptotic distribution of the appropriately standardized deviation of $\hat{\X}_{x_0}(\hat{\theta}_{x_0}) - \X_{x_0}(\theta_{x_0})$ uniformly over a set of starting points $x_0 \in \mathcal{G}$ with ${\cal L} = \{\X_{x_0}(\theta_{x_0}): \; x_0\in\cal G\}$ (Theorem~\ref{ConfBand}). Along the way we derive several useful results about the estimation of integral curves (Theorem~\ref{Gaussian} and Theorem~\ref{UniformPath}). \\[5pt]
The proof of our main result Theorem~\ref{ConfBand} rests on a probabilistic result on the extreme value behavior of certain non-stationary Gaussian fields indexed by growing manifolds.  The main reason for this approach to work is an approximation of the deviation $\hat{\X}_{x_0}(\hat{\theta}_{x_0}) - \X_{x_0}(\theta_{x_0})$  by a linear function of the second derivatives of the kernel density estimator (Theorem~\ref{UniformXDiff}), which in turn can be approximated by a Gaussian field.\\[5pt]
The same approach is expected to work in other situations, as long as we consider linear functionals of (derivates of) kernel estimators. One possible application that will be considered elsewhere is the estimation of level set of pdf's or regression functions.\\[5pt]
%
%
Algorithms for finding filaments, or perhaps more importantly, for finding filament structures (i.e. the union of possibly intersecting filaments) have been developed in Qiao (2013). This research will be published elsewhere.

\section{Proofs}\label{proofs}


The proof of Theorem~\ref{Gaussian} follows similar ideas as the proof of Theorem 1 from Koltchinskii et al. (2007). The necessary modifications are more or less straightforward. More details are available from the authors. In what follows we use the notation $\| A \|_F$ to denote the Frobenius norm of a matrix $A$.
\subsection{Proof of Theorem~\ref{UniformPath}}
Since
\begin{align*}
&\hspace*{1cm}\sup_{x_0\in \mathcal{G}}\; \sup_{t\in[-T_{x_0}^{min}, T_{x_0}^{max}]}\|\hat \X_{x_0}(t)-\X_{x_0}(t)\| \\
& = \max\Big\{\sup_{x_0\in \mathcal{G}} \; \sup_{t\in[0, T_{x_0}^{max}]}\|\hat \X_{x_0}(t)-\X_{x_0}(t)\|, \sup_{x_0\in \mathcal{G}}\; \sup_{t\in[-T_{x_0}^{min}, 0]}\|\hat \X_{x_0}(t)-\X_{x_0}(t)\|\Big\},
\end{align*}
it suffices to prove that
\begin{align}
\sup_{x_0\in \mathcal{G}} \; \sup_{t\in[0, T_{x_0}^{max}]}\|\hat \X_{x_0}(t)-\X_{x_0}(t)\| &= O_p\bigg(\sqrt{\frac{\log{n}}{nh^5}\;}\;\bigg),\label{UniformPath1}
\end{align}
and that the same result holds with $[0,T_{x_0}^{max}]$ replaced by $[-T_{x_0}^{max},0]$. The latter result can be proven similarly to (\ref{UniformPath1}) by considering the relationship between $\X_{x_0}$ and $\tilde{\X}_{x_0}$ defined in (\ref{def-integral-curve-reverse}). For simplicity we write $T_{x_0}^{max}$ as $T_{x_0}$.\\[5pt]
With $\hat \Z_{x_0}$ satisfying the differential equation
\begin{align}\label{Z1}
\frac{d\hat \Z_{x_0}(t)}{dt}=\hat V(\X_{x_0}(t))-V(\X_{x_0}(t))+\nabla V (\X_{x_0}(t))\hat \Z_{x_0}(t), \quad \hat \Z_{x_0}(0)=0,\\[-18pt]\nonumber
\end{align}
and $\hat \Y_{x_0}(t)=\hat \X_{x_0}(t)-\X_{x_0}(t)$, we denote $\hat \D_{x_0}(t):=\hat \Y_{x_0}(t)-\hat \Z_{x_0}(t)$. Following similar arguments as in the proof on page 1586 of Koltchinskii et al. (2007), we can show that
\begin{align*}
\sup_{x_0\in \mathcal{G}, t\in[0,T_{x_0}]}\|\hat \D_{x_0}(t)\|=o_p\bigg(\sup_{x_0\in \mathcal{G}, t\in[0,T_{x_0}]}\|\hat \Z_{x_0}(t)\|\bigg) \quad \textrm{as} \quad n\rightarrow \infty.
\end{align*}
%
%
To show the assertion of the theorem we now show  $\sup_{x_0 \in {\mathcal{G}},\,t \in [0,T_{x_0}]}\|\hat \Z_{x_0}(t)\| = O_p\big(\sqrt{\frac{\log n }{nh^5} }\big).$  Since
\begin{align*}
\|\hat \Z_{x_0}(t)\|\leq \bigg\|\int_0^t(\hat V-V)(\X_{x_0}(s))ds\bigg\|+\int_0^t\|\nabla V (\X_{x_0}(s))\|_F\|\hat \Z_{x_0}(s)\|ds,
\end{align*}
Gronwall's inequality (see Gronwall (1919)) implies
\begin{align*}
\|\hat \Z_{x_0}(t)\| \leq \sup_{x_0\in \mathcal{G}, t\in[0,T_{x_0}]}\bigg\|\int_0^{t}(\hat V-V)(\X_{x_0}(s)) ds\bigg\| \;e^{\int_0^t\|\nabla V (\X_{x_0}(s))\|_Fds},
\end{align*}
so that
\begin{align}\label{S3SEC2}
\sup_{x_0\in \mathcal{G}, t\in[0,T_{x_0}]}\|\hat \Z_{x_0}(t)\| \leq C\sup_{x_0\in \mathcal{G}, t\in[0,T_{x_0}]}\bigg\|\int_0^{t}(\hat V-V)(\X_{x_0}(s)) ds\bigg\|
\end{align}
for some constant $C$. We have
\begin{align*}
&\int_0^{t}\Big[(\hat V-V)(\X_{x_0}(s))-\nabla G(d^2  f(\X_{x_0}(s))) d^2 (\hat f - f) (\X_{x_0}(s))\Big]ds\\
&\hspace{2cm}=\left(\begin{array}{c}
\int_0^{t}d^2 (\hat f - f)  (\X_{x_0}(s))^T M_1(s)\;d^2 (\hat f - f)  (\X_{x_0}(s))ds\\
\int_0^{t}d^2 (\hat f - f)  (\X_{x_0}(s))^T M_2(s)\;d^2 (\hat f - f)  (\X_{x_0}(s))ds\end{array}\right),
\end{align*}
where
\begin{align}\label{MiInt}
M_i(s):=\int_0^1\nabla^2G_i(d^2  f(\X(s))+\tau d^2 (\hat f - f)  (\X(s)))d\tau, \quad i=1,2.
\end{align}
Thus, by using the fact that under our assumptions $\sup_{x\in\mathbb{R}^2}\|\nabla^2 \hat f(x) - \nabla^2f  (x)\|_F = O_P\Big(\sqrt{\frac{\log{n}}{nh^{6}}}\Big)$, 
\begin{align}\label{S4SEC2}
&\sup_{x_0\in \mathcal{G} \atop t\in[0,T_{x_0}]}\bigg\|\int_0^{t}\Big[(\hat V-V)(\X_{x_0}(s))-\nabla G(d^2  f(\X_{x_0}(s)))d^2 (\hat f - f)  (\X_{x_0}(s))\Big]ds\bigg\|\nonumber\\
&\hspace{0.3cm}\leq \sup_{x_0 \in \mathcal{G}}T_{x_0}\sup\limits_{\substack{i=1,2 \\ w\in \mathbb{R}^2}}\|\nabla^2 G_i(w)\|\bigg(\sup_{x\in\mathbb{R}^2}\|d^2 (\hat f - f)  (x)\|\bigg)^2=O_p\bigg(\frac{\log{n}}{nh^{6}}\bigg).
\end{align}
It remains to show that $\sup_{x_0 \in \mathcal{G}}\big|\int_0^{t}\nabla G(d^2  f(\X_{x_0}(s)))d^2 (\hat f - f)  (\X_{x_0}(s))ds\big| = O_p\Big(\sqrt{\frac{\log n}{nh^5}\,} \Big).$ We write this integral as the sum of two terms, a mean zero probabilistic part $\int_0^{t}\nabla G(d^2  f(\X_{x_0}(s)))[d^2 \hat f(\X_{x_0}(s))-\mathbb{E}d^2 \hat f(\X_{x_0}(s))]ds$ and a term caused by the bias, $\int_0^{t}\nabla G(d^2  f(\X_{x_0}(s)))[\mathbb{E}d^2 \hat f(\X_{x_0}(s))-d^2  f(\X_{x_0}(s))]ds$. We will discuss the uniform convergence rate for each of two terms. \\

As for the bias term, recall that we have assumed that all the partial derivatives of $f$ up to fourth order are bounded and continuous.  Then we have with ${\mathcal {Q}_{\delta}}$ from (\ref{Q-def})
{\allowdisplaybreaks
\begin{align*}
\sup_{x_0\in \mathcal{G}\atop t\in[0,T_{x_0}]}&\bigg\|\mathbb{E}\bigg[\int_0^t \nabla G(d^2  f(\X_{x_0}(s)))d^2 (\hat f - f)  (\X_{x_0}(s))ds\bigg]\bigg\|\\
%
%
%
&\hspace*{1cm}\leq \sup_{x \in Q_\delta}\|\nabla G(x)\| T_{\mathcal{G}}\sup_{x\in\mathbb{R}^2}\|\mathbb{E}(d^2 (\hat f - f)  (x))\|= O(h^2),
%
%
%
%
\end{align*}
where the order $O(h^2)$ of the bias of the kernel estimator of the second derivatives of the density follows by standard arguments. Under the assumption that $nh^9\rightarrow\beta\geq0$, we have
\begin{align}\label{S5SEC2}
\sup_{x_0\in \mathcal{G}\atop t\in[0,T_{x_0}]}\bigg\|\mathbb{E}\bigg[\int_0^t \nabla G(d^2  f(\X_{x_0}(s)))d^2 (\hat f - f)  (\X_{x_0}(s))ds\bigg]\bigg\|=O\bigg(\frac{1}{\sqrt{nh^5}}\bigg).
\end{align}
%
To complete the proof we now consider the mean zero stochastic part and show that for $j=1,2$ with $\overline{d^2} \hat f(\X_{x_0}(s)) =  d^2 \hat f(\X_{x_0}(s))-\mathbb{E}d^2 \hat f(\X_{x_0}(s))$
\begin{align}\label{ConvRateGine1}
&\sup_{x_0\in \mathcal{G} \atop t\in[0,T_{x_0}]}\bigg\|\int_0^{t}\nabla G_j(d^2  f(\X_{x_0}(s)))\overline{d^2} \hat f(\X_{x_0}(s))ds\bigg\|=O_p\bigg(\sqrt{\frac{\log{n}}{nh^5}}\bigg).
\end{align}
%
Let $K_1=K^{(2,0)}$, $K_2=K^{(1,1)}$, $K_3=K^{(0,2)}$ and write
\begin{align}\label{def-varpi}
\omega_j(x;x_0,t)& :=\int_0^{t}\nabla G_j(d^2  f(\X_{x_0}(s)))d^2  K\Big(\frac{\X_{x_0}(s)-x}{h}\Big)ds 
%
%
= \sum_{\ell = 1}^3 \omega_{j,\ell}(x;x_0,t),
\end{align}
where $\omega_{j,\ell}(x;x_0,t) = \int_0^{t}\frac{\partial G_j}{\partial x_\ell}\big(d^2  f(\X_{x_0}(s))\big)  K_\ell {\textstyle \big(\frac{\X_{x_0}(s)-x}{h}\big)}ds.$ The dependence of the functions $\omega_{j,\ell}$ on $h$ (and thus on $n$) is suppressed in the notation. It suffices to show that for $j=1,2$ and $\ell=1,2,3$,\\[-20pt]
%
\begin{align}\label{GSeperate}
\sup_{x_0\in \mathcal{G}, t\in[0,T_{x_0}]}\Big|\frac{1}{nh^4}\sum_{i=1}^n\Big[\omega_{j,\ell}(X_i;x_0,t)-\mathbb{E}\omega_{j,\ell}(X_i;x_0,t)\Big]\Big|=O_p\Big(\sqrt{\frac{\log{n}}{nh^5}}\Big).
\end{align}
In order to see this we will use some empirical process theory. Related results can also be found in Gin\'{e} et al. (2002) and Einmahl et al. (2005). Consider the classes of functions
\begin{align*}
\mathcal {F}_{j,\ell}=\big\{\omega_{j,\ell}(\cdot;x_0,t): x_0\in \mathcal{G}, t\in[0,T_{x_0}]\big\},\;j = 1,2,\;\;\ell = 1,2,3.
\end{align*}
Again note that the classes $\mathcal {F}_{j,\ell}$ depend on $n$ through $h$. Let $Q$ denote a probability distribution on ${\mathbb R}^2$. For a class of (measurable) functions $\mathcal{F}$ and $\tau >0$, let $N_{2,Q}(\mathcal{F},\epsilon)$ be the smallest number of $L_2(Q)$-balls of radius $\tau$ needed to cover $\mathcal{F}$. We call $N_{2,Q}(\mathcal{F},\epsilon)$ the covering number of ${\cal F}_{j,\ell}$ with respect to the $L_{2}(Q)$-distance. We now show that for some constants $A,v > 0$ (not depending on $n$), we have
\begin{align}\label{VCdef}
\sup_{Q} N_{2,Q}({\mathcal F}_{j,\ell},\tau) \le \bigg(\frac{A}{\tau}\bigg)^v\quad j=1,2,\;\;\ell = 1,2,3.
\end{align}
Empirical process theory then will imply (\ref{GSeperate}) once we have found an appropriate uniform bound for the variance of the random variables $\omega_{j,\ell}(X_i;x_0,t)$.  Property (\ref{VCdef}) follows from
\begin{align}\label{VCDef}
N_{\infty}\big(\mathcal{F}_{j,\ell}, \tau\big)\leq \bigg(\frac{A}{\tau}\bigg)^v\quad j=1,2,\;\;\ell = 1,2,3,
\end{align}
where $N_{\infty}\big(\mathcal{F}_{j,\ell}, \tau\big)$ denotes the covering number of ${\cal F}_{j,\ell}$ with respect to the supremum distance $d_\infty(f_1,f_2) = \sup_{x \in {\mathbb R}^2}|f_1(x) - f_2(x)|$.
Property (\ref{VCdef}) follows from (\ref{VCDef}), because for any probability measure $Q$ we trivially have $d_{2,Q}(f_1,f_2) = \big(\int |f_1 - f_2|^2 dQ\big)^{1/2} \le \sup|f_1 - f_2|,$ which implies that $\sup_{Q} N_{2,Q}({\mathcal F}_{j,\ell},\tau) \le N_{\infty}\big(\mathcal{F}_{j,\ell}, \tau\big).$ The proof of (\ref{VCDef}) can be found in the supplementary material (Qiao and Polonik, 2015b).\\[5pt]
Now that we have control over the covering numbers of the classes ${\cal F}_{j,\ell}$, all we need to apply standard results from empirical process theory is an upper bound for the maximum variance of the $\omega_{j,\ell}(X_i;x_0,t).$ We have with $K_{\ell,x_0,h}(s,x) = K_\ell\big(\frac{\X_{x_0}(s)-x}{h}\big)$ that
\begin{align*}
&{\mathbb E}\omega^2_{j,\ell}(X_i;x_0,t) \le  \int \Big[\int_0^{t} \widetilde G_j(\X_{x_0}(s)) K_{\ell,x_0,h}(s,x)ds\Big]^2 dF(x)\\
& = \int \int_0^t \int_0^t \Big[ \widetilde G_j(\X_{x_0}(s)) K_{\ell,x_0,h}(s,x)\Big]\Big[\ \widetilde G_j(\X_{x_0}(s^\prime))  K_{\ell,x_0,h}(s^\prime,x)\Big] ds ds^\prime\;dF(x)\\
& \le \int_0^t \int_0^t \big| \widetilde G_j(\X_{x_0}(s)) \widetilde G_j(\X_{x_0}(s^\prime))\big| \int  K_{\ell,x_0,h}(s,x) K_{\ell,x_0,h}(s^\prime,x)\;dF(x)\;ds ds^\prime.
\end{align*}
Now recall that the kernel $K$ is assumed to have support inside the unit ball, and notice that ${\bf 1}\big(\|\X_{x_0}(s)-x\| \le h \big)\cdot  {\bf 1}\big(\|\X_{x_0}(s^\prime)-x\| \le h \big) \le {\bf 1}\big( \|\X_{x_0}(s) - \X_{x_0}(s^\prime)\| \le 2h\big) \le {\bf 1}\big( |s - s^\prime| \le c\,h\big)$ for some $c > 0$ with $c$ not depending on $x_0, s$ or $s^\prime$. Therefore we can estimate the last integral above by
\begin{align*}
C \int_0^t \int_0^t {\bf 1}\big( |s - s^\prime| \le c\,h\big) \int {\bf 1}\big( \|\X_{x_0}(s) - x\| \le h\big) dF(x)\, dsds^\prime \le C^\prime h^3
\end{align*}
where $C^\prime = \sup_{u}|\widetilde G_j(u)|^2 M_\ell^2 $ and $C>0$ sufficiently large.  Notice here that $0 \le t \le T_{x_0} \le T_{\mathcal{G}} < \infty.$ Thus we can uniformly bound the variances of the sum in (\ref{GSeperate}) by $\sigma^2 = O\big(n\cdot \big(\frac{1}{nh^4}\big)^2 h^3\big) = O\big(\frac{1}{nh^5}\big)$. \\[8pt]
We now have control over the covering numbers of  the classes ${\cal F}_{j,\ell}$ along with the variances of the sums involved. It is known from empirical process theory (e.g. see van der Vaart and Wellner (1996), Theorem 2.14.1) that empirical processes indexed by (uniformly bounded) classes of functions satisfying (\ref{VCDef}) (even if the function classes depend on $n$) behave like $O_p(\sqrt{\sigma^2 \log 1/\sigma^2})$. Plugging in our bound for $\sigma$ gives the assertion of (\ref{GSeperate}).\\[-4pt]

The asserted uniform convergence rate of the difference $\hat \X_{x_0}(t)-\X_{x_0}(t)$ in (\ref{UniformPath1}) now follows from (\ref{S3SEC2}), (\ref{S4SEC2}), (\ref{S5SEC2}) and (\ref{ConvRateGine1}). \hfill$\square$\\[-4pt]

Below we will repeatedly apply Theorem~\ref{UniformPath} with $T_{x_0} = \theta_{x_0}$, and so we need to know that $x_0 \to \theta_{x_0}$ is continuous.  This in fact follows from the fact that we can interpret  $\theta_{x_0}$ as indexed by the integral curve itself, and that we use the fact that by our assumptions, integral curves are dense and non-overlapping (a formal proof can be found in the supplemental material (Qiao and Polonik, 2015b).

\subsection{A rate of convergence for $\boldsymbol{\sup_{x_0 \in \mathcal{G}}\big|\hat\theta_{x_0} - \theta_{x_0}\big|}$}

A rate of  convergence for $\sup_{x_0 \in \mathcal{G}}\big| \hat\theta_{x_0} - \theta_{x_0}\big|$ is given here that is needed in the proofs below. Recall that $\hat \theta_{x_0} = \argmin_t\{|t|: \; t\in \Theta_{x_0}\}$ if $\wh \Theta_{x_0}\neq \emptyset$ (see (\ref{def-thetahat})), where $\wh \Theta_{x_0}=\{t: \blangle\nabla \hat f(\hat \X_{x_0}(t)), \hat V(\hat \X_{x_0}(t))\brangle = 0, \;\;\hat\lambda_2(\hat \X_{x_0}(t))<0\}$. The proof of the following result also implies that $\wh \Theta_{x_0}$ is non-empty and that $\hat \theta_{x_0}$ is unique for all $x_0\in\mathcal{G}$ with high probability for large $n$. 
%
%
We will thus only consider the case of $\wh \Theta_{x_0}\neq \emptyset$ and unique $\hat \theta_{x_0}$ in what follows. 
\begin{proposition}\label{ThetaRate}
Under assumptions (\textbf{F}1)--(\textbf{F}6), (\textbf{K}1)--(\textbf{K}2) and (\textbf{H}1), we have with  $\alpha_n=\sqrt{\frac{\log{n}}{nh^6}}$
\begin{align*}
\sup_{x_0\in\mathcal{G}}|\hat \theta_{x_0}-\theta_{x_0}|=O_p(\alpha_n),
\end{align*}
If, in addition, $\sup\limits_{x_0\in\mathcal{G}}\|\nabla f(\X_{x_0}(\theta_{x_0}))\|=0$, then we can choose $\alpha_n=\sqrt{\frac{\log{n}}{nh^5}}$.
\end{proposition}
In addition to Theorem~\ref{UniformPath}, Proposition~\ref{ThetaRate} is an important igredient to the proofs of Theorem~\ref{Approx-FilaDiff} and Lemma~\ref{Theta}. The proofs of Proposition~\ref{ThetaRate}, Theorem~\ref{Approx-FilaDiff} and Lemma~\ref{Theta} as well as Corollaries~\ref{distr-filament-est1} -- \ref{filament-est3} can be found in the supplementary material (Qiao and Polonik, 2015b).

\subsection{Proof of Theorem~\ref{ConfBand}}

Similar to Bickel and Rosenblatt (1973) and Rosenblatt (1976), the proof of Theorem~\ref{ConfBand} consists in using a strong approximation by (nonstationary) Gaussian processes indexed by manifolds. This approximation, and in particular the indexing manifold itself, depend on the bandwidth $h$. In fact, the indexing manifold is growing when $h \to 0$. In a companion paper, Qiao and Polonik (2015a) derive the  extreme value behavior of Gaussian processes in such scenarios, and we will apply their result here. Below we state a special case of this result for convenience.\\[8pt]
Before stating the result on the extreme value behavior of Gaussian fields indexed by growing manifolds that has been mentioned above, we need a definition that extends the notion of local $D_t$-stationarity that is known from the literature, e.g., Mikhavela and Piterbarg (1996). Since we are dealing with growing indexing manifolds (as $h \to 0$), we need the local $D_t$-stationarity to hold uniformly over the bandwidth $h$. The following definition makes this uniformity precise:
\begin{definition}[Local equi-$D_t$-stationarity] Let $X_h(t),t\in \mathcal {S}_h\subset\mathbb{R}^2$ be a sequence of nonhomogeneous random fields indexed by $h\in \mathbb{H}$ where $\mathbb{H}$ is an index set. We say $X_h(t)$ has a local equi-$D_t$-stationary structure, or $X_h(t)$ is locally equi-$D_t^h$-stationary, if for any $\epsilon>0$ there exists a positive $\delta(\epsilon)$ independent of $h$ such that for any $s\in \mathcal {S}_h$ one can find a non-degenerate matrix $D_s^h$ such that the covariance function $r_h(t_1,t_2)$ of $X_h(t)$ satisfies
\begin{align*}
1-(1+\epsilon)\|D_s^h(t_1-t_2)\|^2\leq r_h(t_1,t_2)\leq1-(1-\epsilon)\|D_s^h(t_1-t_2)\|^2
\end{align*}
provided $\|t_1-s\|<\delta(\epsilon)$ and $\|t_2-s\|<\delta(\epsilon).$ 
\end{definition}
%
%
The following result generalizes Theorem~1 in Piterbarg and Stamatovich (2001) and Theorem~A1 in Bickel and Rosenblatt (1973).
\begin{theorem}[Qiao and Polonik, 2015a]\label{ProbMain}
Let $\mathcal{H}_1\subset\mathbb{R}^2$ be a compact set and $\mathcal{H}_h:=\{t:ht\in\mathcal{H}_1\}$ for $0<h\leq1$. Let $X_h(t), t\in \mathcal{H}_h, 0<h\leq1$ be a class of Gaussian centered locally equi-$D_t^h$-stationary fields with matrix $D_t^h$ continuous in $h\in (0, 1]$ and $t\in \mathcal{H}_h$. Let $\mathcal {M}_1\subset\mathcal{H}_1$ be a one-dimensional compact manifold with bounded curvature and $\mathcal {M}_h:=\{t:ht\in\mathcal{M}_1\}$ for $0<h\leq1$. Suppose $\lim_{h\rightarrow0,ht=t^*}D_t^h=D_{t^*}^0$ uniformly in $t^*\in \mathcal{H}_1$, where all the components of $D_{t^*}^0$ are continuous and bounded in $t^*\in \mathcal{H}_1$. Further assume there exists a positive constant $C$ such that
\begin{align}\label{Lambda2C}
\inf_{0<h\leq 1, hs\in \mathcal{H}_1}\lambda_2(\{D_s^h\}^T D_s^h)\geq C,
\end{align}
where $\lambda_2(\cdot)$ is the second eigenvalue of the matrix. Suppose for any $\delta>0$, there exists a positive number $\eta$ such that  the covariance function $r_h$ of $X_h$ satisfies
\begin{align}\label{SupGauss1}
\sup_{0<h\leq 1}\{|r_h(x+y,x)|: x+y\in\mathcal {M}_h, x\in\mathcal {M}_h, \|y\|>\delta\} < \eta<1.
\end{align}
In addition, assume that there exists a $\tilde\delta>0$ such that
\begin{align}\label{SupGauss2}
\sup_{0<h\leq 1}\{|r_h(x+y,x)|: x+y\in\mathcal {M}_h, x\in\mathcal {M}_h, \|y\|>\tilde\delta\} =0.
\end{align}
For any fixed $z$, define
\begin{align*}
\theta \equiv \theta(z) =\sqrt{2\log{h^{-1}}}+\frac{1}{\sqrt{2\log{h^{-1}}}}\bigg[z+\log\bigg\{\frac{1}{\sqrt{2}\,\pi}\int_{\mathcal {M}_1}\|D_s^0 M_s^1\|ds\bigg\}\bigg],
\end{align*}
where $M_s^1$ is the unit vector denoting the tangent direction of $\mathcal {M}_1$ at $s$. Then
\begin{align*}
\lim_{h\rightarrow0}P\Big\{\sup_{t\in\mathcal {M}_h}|X_h(t)|\leq \theta \Big\}=\exp\{-2\exp\{-z\}\}.
\end{align*}
\end{theorem}

\begin{remark}
Let $\lambda_1(A)$ is the first eigenvalue of $A.$ The assumptions of $h \to D_t^h$ being a continuous matrix function in $h\in (0, 1]$ and $t\in\mathcal{H}_h,$ and $\lim_{h\rightarrow0,ht=t^*}D_t^h=D_{t^*}^0$ uniformly in $t^*\in \mathcal{H}_1,$ imply the existence of $C^\prime \le 0$ with $\sup_{0<h\leq 1, hs\in\mathcal{H}_1}\lambda_1(\{D_s^h\}^T D_s^h)\leq C^\prime$, since the first eigenvalue is a continuous function of the entries of a matrix. Now we have
\begin{align}\label{EigenvalueBound} 
0<C\leq \inf_{0<h\leq 1 \atop hs\in\mathcal{H}_1}\lambda_2(\{D_s^h\}^T D_s^h)\leq\sup_{0<h\leq 1 \atop hs\in\mathcal{H}_1}\lambda_1(\{D_s^h\}^T D_s^h)\leq C^\prime.
\end{align}
Since\\[-20pt]
\begin{align}\label{MatrixBound}
\lambda_2(\{D_s^h\}^T D_s^h)\|t_1-t_2\|^2\leq \|D_s^h(t_1-t_2)\|^2\leq\lambda_1(\{D_s^h\}^T D_s^h)\|t_1-t_2\|^2,
\end{align}
local equi-$D_t$-stationarity of $X_h(t)$ implies that
\begin{align}\label{CovLoc}
r_h(t_1,t_2)=1-\|D_s^h(t_1-t_2)\|^2+o(\|t_1-t_2\|^2)
\end{align}
uniformly for $t_1,t_2\in\mathcal{H}_h$. On the other hand, (\ref{EigenvalueBound}) and (\ref{MatrixBound}) imply
\begin{align*}
\frac{1}{C^\prime}\|D_s^h(t_1-t_2)\|^2\leq \|t_1-t_2\|^2 \leq \frac{1}{C} \|D_s^h(t_1-t_2)\|^2.
\end{align*}
Hence, (\ref{CovLoc}) also implies the local equi-$D_t^h$-stationarity of $X_h(t)$.
\end{remark}
{\bf Proof of Theorem~\ref{ConfBand}} Observe that by using Theorem~\ref{UniformXDiff} we have
\begin{align*}
\sup_{x_0 \in \mathcal{G}}\|\hat\X_{x_0}(\hat \theta_{x_0}) - \X_{x_0}(\theta_{x_0}) - \hat\varphi_{1n}(\X_{x_0}(\theta_{x_0})) V(\X_{x_0}(\theta_{x_0}))\| = O_p\Big(\frac{\log n}{nh^7}\Big).
\end{align*}
Therefore, by using $({\bf H}1)$ we see that (\ref{ConfBand1}) will follow once we have shown that
\begin{align}\label{ConfBand3}
P\bigg(\sup_{x_0\in\mathcal{G}}\Big\| \sqrt{nh^6} g(\X_{x_0}(\theta_{x_0}))\,\hat\varphi_{1n}(\X_{x_0}(\theta_{x_0})) \,V(\X_{x_0}(\theta_{x_0}))\Big\|<b_h(z)\bigg) \to e^{-2\,e^{-z}}.
\end{align}
By definition of $g(x)$ and $\hat \varphi_{1n}(x)$ (see (\ref{def-g}) and (\ref{def-a1n}), respectively) we have
\begin{align*}
\Big\| \sqrt{nh^6} g(\X_{x_0}(\theta_{x_0}))\,\hat\varphi_{1n}(\X_{x_0}(\theta_{x_0})) \,V(\X_{x_0}(\theta_{x_0}))\Big\| = \big|Y_n(\X_{x_0}(\theta_{x_0}))\big|,
\end{align*}
where\\[-20pt]
\begin{align*}
Y_n(x)=\frac{\sqrt{nh^6}}{\sqrt{f(x)\,}\, \|A(x)\|_{\mathbf R}}\left\langle A(x),d^2 \hat f(x)-\mathbb{E}d^2 \hat f(x)\right\rangle,
\end{align*}
and $A(x) \in {\mathbb R}^3$ is defined in (\ref{GtildeDef}). In other words, (\ref{ConfBand3}) can be written as
\begin{align}\label{Replace1}
\lim_{n\rightarrow\infty}&P\Big(\sup_{x\in \mathcal{L}}|Y_n(x)|<b_{h}(z)\Big)=\exp\{-2\exp\{-z\}\}.
\end{align}
Note that for any $x\in \mathcal{H}$, we have $Y_n(x)\rightarrow {\cal N}(0,1)$ in distribution as $n\rightarrow\infty$. This immediately follows from the fact that under the present assumptions $\sqrt{nh^6}(d^2\hat f(x)-\mathbb{E}d^2 \hat f(x))\rightarrow \mathscr{N}(0,f(x)\mathbf{R})$ in distribution.\\[8pt]
Now we prove (\ref{Replace1}).  In what follows we denote $\mathcal{H}_h=\{x: hx\in\mathcal{H}\}$ and $\mathcal{L}_h=\{x: hx\in\mathcal{L}\}$ for $0<h\leq 1$. For $0 < h < 1$ we also use the notation\\[-15pt]
\begin{align}\label{defahAh}
A_h(x)=\tilde{G}(hx)^T  \nabla f(hx) \in {\mathbb R}^{3}\quad\text{and}\quad
a_h(x)=\frac{1}{\|A_h(x)\|_{\mathbf R}}.
\end{align}
(Recall that $\tilde G(x) = \nabla G(d^2f(x))$ with $G$ defined in (\ref{GDef}).) Note that plugging in $h = 1$ into the definition we obtain $A(x) = (A_1(x), A_2(x), A_3(x))^T$ used above already. For ease of notation we denote $a_1(x) = a(x).$ Let $W$ be a two-dimensional Wiener process and
\begin{align}\label{GaussianU}
U_h(x)=a_h(x)\int \big(A_h(x))^T  d^2   K(x-s) dW(s).
\end{align}
%
Following similar arguments in the proof of Theorem~1 in Rosenblatt (1976) (for more details see technical supplement (Qiao and Polonik, 2015b), it suffices to prove
\begin{align}\label{sufficiency}
\lim_{n\rightarrow\infty}&P\Big(\sup_{x\in \mathcal{L}_h}|U_h(x)|<b_h(z)\Big)=\exp\{-2\exp\{-z\}\}.
\end{align}
To complete the proof, we now show that $U_h(x), x \in {\mathcal{L}}_h$ satisfies the conditions of Theorem ~\ref{ProbMain}. For any $x,y\in\mathcal{L}_h$ let $r_h(x, y) := \Cov(U_h(x), U_h(y)).$ Then obviously
\begin{align}\label{covU}
r_h(x, y)=a_h(x)a_h(y)A_h(x)^T\int_{\mathbb{R}^2}d^2  K(x-s)d^2  K(y-s)^T \; dsA_h(y).
\end{align}
To show that (\ref{SupGauss1}) and (\ref{SupGauss2}) hold for this covariance function, we will calculate the Taylor expansion of the covariance function $r_h(x+y, x)$ as $y\rightarrow 0$. For any vector-valued function $g(\cdot)=(g_1(\cdot),g_2(\cdot),g_3(\cdot))^T :\mathbb{R}^2\to \mathbb{R}^3$, denote
\begin{align*}
\nabla^{\otimes2} g(x)=\begin{pmatrix}
g_1^{(2,0)}(x) & g_1^{(1,1)}(x) & g_1^{(1,1)}(x) & g_1^{(0,2)}(x)\\
g_2^{(2,0)}(x) & g_2^{(1,1)}(x) & g_2^{(1,1)}(x) & g_2^{(0,2)}(x)\\
g_3^{(2,0)}(x) & g_3^{(1,1)}(x) & g_3^{(1,1)}(x) & g_3^{(0,2)}(x)
\end{pmatrix}
\end{align*}
and $x^{\otimes2}=(x_1^2,x_1x_2,x_1x_2,x_2^2)^T $. For any $x\in \mathbb{R}^2$, we have $\|x^{\otimes2}\|=\|x\|^2$. A Taylor expansion of $r_h(x+y, x)$ (see technical supplement (Qiao and Polonik, 2015b) for details) gives
\begin{align}\label{RExpression}
r_h(x+y,x)=1-y\Lambda(h,x)y^T+o(\|y^{\otimes2}\|), 
%
%
\end{align}
where the little-$o$ term in (\ref{RExpression}) is independent of $h$ and equivalent to $o(\|y\|^2)$, and $\Lambda(h,x)=\Lambda_1(h,x)+\Lambda_2(hx)$ with
\begin{align*}
\Lambda_1(h,x)&=\frac{1}{2}(a_h(x))^2\Big[\;\nabla A_h(x)^T \mathbf{R}\nabla A_h(x)\\
&\hspace*{1cm} + 2\,\big(A_h(x)^T \mathbf{R}\nabla A_h(x)\big)^T \big(A_h(x)^T \mathbf{R}\nabla A_h(x)\big) \big]
\end{align*}
and the matrix $\Lambda_2(hx)$ implicitly defined through (an explicit expression for $\Lambda_2(h,x)$ is derived below)
\begin{align}\label{Lambda2}
y^T\Lambda_2(hx)y=-\frac{1}{2}(a_h(x))^2A_h(x)^T \int\big[\nabla^{\otimes2}d^2  K(s)\big]y^{\otimes2}[d^2  K(s)]^T dsA_h(x).
\end{align}
%
%
%
Notice that $\Lambda_2$ only depends on the product $hx$, while $\Lambda_1(\lambda,h)$ depends on both $hx$ and $h$ itself (because of the presence of $\nabla A_n(x)$). Obviously $\Lambda_1(h,x)$ is symmetric and we will see below that $\Lambda_2(hx)$ can also be chosen to be symmetric. 
The matrix $\Lambda_1(h,x)$ is positive semi-definite. If we keep $hx$ fixed, say as $x^*$, and let $h\rightarrow0$,
\begin{align*}
\lim_{hx=x^*,h\rightarrow0}\Lambda_1(h,x)=0
\end{align*}
uniformly in $x^*\in\mathcal{H}$. On the other hand, if $hx = x^*$ is fixed, then $\Lambda_2(xh) = \Lambda_2(x^*)$ stays fixed as well. We will in fact give an explicit expression of $\Lambda_2(hx)$ and show that it is strictly positive definite under the given assumptions. Using these two properties, our expansion (\ref{RExpression})  then implies  (\ref{CovLoc}) with $D_s^h(t_1-t_2) = \big(\Lambda(h, t_1-t_2)\big)^{1/2},$ and this implies local equi-$D_t$-stationarity of $U_n(x), x \in {\mathcal{L}}$ (see remark right after Theorem~\ref{ProbMain}).  It then only remains to verify conditions (\ref{Lambda2C}), (\ref{SupGauss1}) and (\ref{SupGauss2}) from Theorem~\ref{ProbMain}. The latter follows easily, however, because due to the boundedness of the support of the kernel $K$, we have $r_h(x+y,x)=0$ once $\|y\|>1$. \\[8pt]
To show (\ref{Lambda2C}) we first derive an explicit expression for $\Lambda_2(hx)$ by using the properties of $K$ discussed after the assumptions in section \ref{Assumptions}. We have
\begin{align}\label{PartCalc}
\int\nabla^{\otimes2}d^2  K(s)y^{\otimes2}[d^2  K(s)]^T ds=-\int K^{(1,2)}(z)^2dz\; \Delta(y),
\end{align}
where\\[-25pt]
\begin{align*}
\Delta(y):=\begin{pmatrix}
b_1 y_1^2+y_2^2 & 2 y_1y_2 & y_1^2 + y_2^2\\
2y_1y_2 & y_1^2 + y_2^2 & 2y_1y_2\\
y_1^2 + y_2^2 & 2y_1y_2 & y_1^2+b_1y_2^2\\
\end{pmatrix},
\end{align*}
with 
\begin{align}\label{Pdefinition}
b_1=\int \big[K^{(3,0)}(z)\big]^2dz\bigg/\int \big[K^{(1,2)}(z)\big]^2dz.
\end{align}
Note that $ \int [K^{(3,0)}(z)]^2dz+\int [K^{(1,2)}(z)]^2dz\geq 2\int K^{(3,0)}(z) K^{(1,2)}(z)dz=$  $2\int [K^{(2,1)}(z)]^2dz,$
where by assumption (\textbf{K}3) equality is impossible. Thus $b_1 > 1$. Plugging (\ref{PartCalc}) into (\ref{Lambda2}) gives
\begin{align*}
y\Lambda_2(hx)y^T&=\frac{1}{2}(a_h(x))^2\int K^{(1,2)}(z)^2dzA_h(x)^T \Delta(y) A_h(x)\\
&=\frac{1}{2}(a_h(x))^2\int K^{(1,2)}(z)^2dz\;y^T\Omega(hx)y,
\end{align*}
where $\Omega(\cdot)$ is explicitly given in (\ref{OmegaDef}) below. It is straightforward to see that $\Omega$ is positive definite (by using that $b_1 > 1$). Hence
\begin{align*}
\Lambda_2(hx)=\frac{1}{2}(a_h(x))^2\int K^{(1,2)}(z)^2dz\;\Omega(hx),
\end{align*}
is positive definite. Let $\lambda_2(\cdot)$  denote the second eigenvalue of a matrix, then,
{\allowdisplaybreaks
\begin{align*}
\inf_{0<h\leq 1, hx\in \mathcal{H}}\lambda_2(\Lambda(h,x))
&=\inf_{0<h\leq 1, hx\in \mathcal{H}} \;\inf_{\|y\|=1}\Big(y^T\Lambda_1(h,x)y+y^T\Lambda_2(hx)y\Big)\\
&\geq \inf_{0<h\leq 1, hx\in \mathcal{H}}\Big(\inf_{\|y\|=1}y^T\Lambda_1(h,x)y+\inf_{\|y\|=1}y^T\Lambda_2(hx)y\Big)\\
%
%
&\ge\inf_{0<h\leq 1, hx\in \mathcal{H}}\lambda_2(\Lambda_2(hx))>0,
\end{align*}}
validating (\ref{Lambda2C}). It remains to verify that $r_h(x,y)$ (defined in (\ref{covU})) satisfies (\ref{SupGauss1}). We first derive a lower bound for the following quantity. This bound will then lead to the desired result.
\begin{align}
&\inf\limits_{\substack{x\in\mathcal{L}_h,x+y\in\mathcal{L}_h \\ \lambda\in \mathbb{R}, 0<h\leq 1,\|y\|>\delta}}\int \Big|A_h(x+y)^Td^2  K(x+y-s)-\lambda[d^2  K(x-s)]^T A_h(x)\Big|^2ds\nonumber\\
&\hspace{1cm}\geq \inf\limits_{\substack{x\in\mathcal{L}_h,x+y\in\mathcal{L}_h \\ \lambda\in \mathbb{R}, 0<h\leq 1,\|y\|>\delta}}\int_{\mathscr{B}(x+y,1)\backslash \mathscr{B}(x,1)} \Big|A_h(x+y)^Td^2  K(x+y-s) \nonumber\\[-5pt]
&\hspace{7cm}
-\lambda[d^2  K(x-s)]^T A_h(x)\Big|^2ds\nonumber\\[5pt]
&\hspace{1cm}=\inf\limits_{\substack{x\in\mathcal{L}_h, x+y\in\mathcal{L}_h \\ \|y\|>\delta, 0<h\leq 1}}\int_{\mathscr{B}(x+y,1)\backslash \mathscr{B}(x,1)} \Big|A_h(x+y)^Td^2  K(x+y-s)\Big|^2ds\nonumber\\
&\hspace{1cm}=\inf\limits_{\substack{x\in\mathcal{L}_h, x+y\in\mathcal{L}_h \\ \|y\|>\delta, 0<h\leq 1}}\int_{\mathscr{B}(0,1)\backslash \mathscr{B}(-y,1)} \Big|A_h(x+y)^Td^2  K(s)\Big|^2ds\nonumber\\
&\hspace{1cm}\geq \inf_{z\in\mathcal{L}}\inf_{\|y\|>\delta}\int_{\mathscr{B}(0,1)\backslash \mathscr{B}(-y,1)} \Big|A(z)^Td^2  K(s)\Big|^2ds.\label{InfIneq}
\end{align}
There exist a finite number of balls ${\mathscr B}_1$, ${\mathscr B}_2,\cdots, {\mathscr B}_N$ such that for any $y$ with $\|y\|>\delta$, at least one of the these balls is contained in $\mathscr{B}(0,1)\backslash \mathscr{B}(-y,1)$. It follows that for any $z\in\mathcal{L}$,
\begin{align*}
\inf_{\|y\|>\delta}\int_{\mathscr{B}(0,1)\backslash \mathscr{B}(-y,1)} \Big|A(z)^Td^2  K(s)\Big|^2ds\geq \min_{i\in\{1,2\cdots,N\}}\int_{\mathscr B_i} \Big|A(z)^Td^2  K(s)\Big|^2ds.
\end{align*}
Note that under assumptions (\textbf{K}4) and (\textbf{F}7), for any $i\in\{1,2\cdots,N\}$ there exists a constant $C>0$ such that the Lebesgue measure of $\{s\in {\mathscr B}_i: |A(z)^Td^2  K(s)|^2>C\}$ is positive. Therefore,
\begin{align*}
\inf_{\|y\|>\delta}\int_{\mathscr{B}(0,1)\backslash \mathscr{B}(-y,1)} \Big|A(z)^Td^2  K(s)\Big|^2ds>0.
\end{align*}
Since $\cal L$ is a compact set by assumption ({\bf F}3), it follows that the integral 
%
$\int_{\mathscr{B}(0,1)\backslash \mathscr{B}(-y,1)} \Big|A(z)^Td^2  K(s)\Big|^2ds$ 
%
is bounded away from zero uniformly over $z\in\mathcal{L}$ and $\|y\|>\delta$, which by (\ref{InfIneq}) further implies that 
\begin{align*}
\inf\limits_{\substack{x,\,x+y\in\mathcal{L}_h\\ \lambda\in \mathbb{R}, 0 < h < 1,\|y\|>\delta}}\int \Big|A_h(x+y)^Td^2  K(x+y-s)-\lambda[d^2  K(x-s)]^T A_h(x)\Big|^2ds>0.
\end{align*}
Recalling the definition of $a_h$ (see (\ref{defahAh})) we can rewrite this inequality as
\begin{align}\label{QuadraticForm}
&\inf\limits_{\substack{x,\,x+y\in\mathcal{L}_h \\ \lambda\in \mathbb{R}, 0<h\leq 1,\|y\|>\delta}}\zeta_{x,y,h}(\lambda)>0,\\[-25pt] \nonumber
\end{align}
where 
%
$ \zeta_{x,y,h}(\lambda) =\frac{\lambda^2}{a_h(x)^2}-2A_h(x+y)^T  \int d^2  K(x+y-s)[d^2  K(x-s)]^T ds\;A_h(x)\lambda 
%
+\frac{1}{a_h(x+y)^2}.$ 
%
If we consider $\zeta_{x,y,h}(\lambda)$ as a quadratic polynomial in $\lambda$ then its
%
discriminant is given by
\begin{align*}
\sigma(x,y,h):&=4\bigg\{A_h(x+y)^T  \int d^2  K(x+y-s)[d^2  K(x-s)]^T ds\;A_h(x)\bigg\}^2\\
&\hspace*{5cm} -\frac{4}{a_h(x)^2a_h(x+y)^2}.
\end{align*}
Inequality (\ref{QuadraticForm}) says that the polynomials $\zeta(\lambda;x,y,h)$ are uniformly bounded away from zero, and thus their discriminants must satisfy
\begin{align*}
&\sup\limits_{\substack{x\in\mathcal{L}_h, x+y\in\mathcal{L}_h \\ \|y\|>\delta, 0<h\leq 1}}\sigma(x,y,h)<0,\\[-25pt]
\end{align*}
or equivalently,
%
\begin{align*}
&\sup\limits_{\substack{x\in\mathcal{L}_h, x+y\in\mathcal{L}_h \\ \|y\|>\delta, 0<h\leq 1}}|r_h(x+y,x)|<1,
\end{align*}
%
%
which is (\ref{SupGauss1}). Finally notice that the constant $c$ in (\ref{CDef}) corresponds to the quantity $\log\big\{\frac{1}{\sqrt{2}\,\pi}\int_{\mathcal{M}_1}\|D_s^0 M_s^1\|ds\big\}$ from Theorem~\ref{ProbMain}. Using the above, one can easily see that $c$ has the form:
\begin{align*}
c=\log\bigg\{\sqrt{\frac{b_2}{2}}\frac{1}{\pi}\int_{\mathcal{L}}\frac{\|\Omega^{1/2}(s)M_s\|}{\|A(s)\|_{\mathbf R}}\,ds\bigg\},
\end{align*}
where $b_2 = \frac{1}{2 }\int K^{(1,2)}(z)^2dz$, $M_s, \, s \in {\mathcal{L}}$ is the unit tangent vector to $\mathcal{L}$ at $s$, and $\Omega(s) = \big(\omega_{ij}\big)_{i,j = 1,2}$ a $(2 \times 2)$-matrix with
\begin{align}\label{OmegaDef}
\omega_{11}(s)&=b_1\,A_1(s)^2+A_2(s)^2+A_3(s)^2+2A_1(s)A_3(s),\nonumber\\
\omega_{12}(s)=\omega_{21}(s)&=2A_1(s)A_2(s)+2A_2(s)A_3(s),\\
\omega_{22}(s)&=b_1\,A_3(s)^2+A_2(s)^2+A_1(s)^2+2A_1(s)A_3(s). \nonumber
\end{align}
where $A(x) = (A_1(x), A_2(x), A_3(x))^T $  with $A(x)$ as above, and $b_1$ as in (\ref{Pdefinition}).\begin{flushright}\hfill$\square$\end{flushright}

\begin{supplement}[id=suppA]
  \stitle{SUPPLEMENTAL MATERIAL FOR: THEORETICAL ANALYSIS OF NONPARAMETRIC FILAMENT ESTIMATION}
  \slink[doi]{COMPLETED BY THE TYPESETTER}
  \sdatatype{.pdf}
  \sdescription{Due to page constraints on the main article this supplement presents the proofs of some technical results in this paper as well as some miscellaneous results (Appendix B) that are used in the proofs.\\}
\end{supplement}

{\bf Acknowledgement.} The authors would like to thank the associate editor and the referees for careful reading of the manuscript and for insightful comments that lead to a significant improvement of the manuscript.\\

{\bf References.}

\begin{description}
\item  Arias-Castro, E., Donoho, D.L. and Huo, X. (2006): Adaptive multiscale detection of filamentary structures in a background of uniform random points. {\em Ann. Statist.}, {\bf 34}(1), 326-349.
%
%
\item Barrow, J.D., Sonoda, D.H. and Bhavsar, S.P. (1985). Minimal spanning tree, filaments and galaxy clustering. {\em Monthly Notices of the Royal Astronomical Society}, {\bf 216}, 17-35.
%
\item Bharadwaj, S., Bhavsar, S.P. and Sheth, J.V. (2004): The size of the longest filaments in the universe. {\em The Astrophysical Journal}, {\bf 606}, 25-31.
\item Bickel, P. and Rosenblatt, M. (1973): On some global measures of the deviations of density function estimates. {\em Ann. Statist.}, {\bf 1}(6), 1071-1095.
\item Chen, Y-C., Genovese, C. R., and Wasserman, L. (2013): Uncertainty Measures and Limiting Distributions for Filament Estimation. {\em arXiv: 1312.2098.v1}.
\item Chen, Y-C., Genovese, C. R., and Wasserman, L. (2014a): Asymptotic theory for density ridges. {\em arXiv: 1406.5663}.
\item Chen, Y-C., Genovese, C. R., and Wasserman, L. (2014b): Generalized mode and ridge estimation. {\em arXiv: 1406.1803}.
\item Cheng, Y. (1995): Mean shift, mode seeking, and clustering. {\em IEEE Trans. Pattern Analysis and Machine Intelligence}, {\bf 17}(8), 790-799.
\item Comaniciu, D. and Meer, P. (2002): Mean shift: A robust approach toward feature space analysis. {\em IEEE Transactions on Pattern Analysis and Machine Intelligence}, {\bf 24}(5) 603-619.
%
%
%
%
%
%
%
%
\item Dietrich, J., Werner, N., Clowe, D., Finoguenov, A., Kitching, T., Miller, L. and Simionescu, A. (2012): A filament of dark matter between two clusters of galaxies. {\em Nature}, {\bf 487}, 202-204.
%

%
%
\item Duong, T., Cowling, A., Koch, I. and Wand, M. (2008): Feature Significance for multivariate kernel density estimation. {\em Comp. Statist. Data Anal.}, {\bf 52}(9), 4225-4242.
%
%
%
\item Eberly, D. (1996): {\em Ridges in image and data analysis.} Kluwer, Boston.
%
%
%
\item Einmahl, U., and Mason, D.M. (2005): Uniform in bandwidth consistency for kernel-type function estimators. {\em Ann. Statist.} {\bf 33}(3), 1380-1403.
%
%
\item Fukunaga, K., and Hostetler, L. (1975): The estimation of the gradient of a density function, with applications in pattern recognition. {\em IEEE Tran. Information Theory}, {\bf 21}(1), 32-40.
\item Genovese, C.R., Perone-Pacifico, M., Verdinelli, I. and Wasserman, L. (2009): On the path density of a gradient field. {\em Ann. Statist.} {\bf 37}, 3236-3271.
\item Genovese, C.R., Perone-Pacifico, M., Verdinelli, I. and Wasserman, L. (2012a): The geometry of nonparametric filament estimation. {\em J. Amer. Statist. Assoc.}, {\bf 107}(498), 788-799.
\item Genovese, C.R., Perone-Pacifico, M., Verdinelli, I. and Wasserman, L. (2012b). Manifold estimation and singular deconvolution under Hausdorff loss. {\em Ann. Statist.} {\bf 40}, 941-963.
\item Genovese, C.R., Perone-Pacifico, M., Verdinelli, I. and Wasserman, L. (2013): Nonparametric inference for density modes. {\em arXiv:1312.7567v1.}
\item Genovese, C.R., Perone-Pacifico, M., Verdinelli, I. and Wasserman, L. (2014): Nonparametric ridge estimation.  {\em Ann. Statist.} {\bf 42},  1511-1545.
%
%
\item Gin\'{e}, E. and Guillou, A. (2002): Rates of strong uniform consistency for multivariate kernel density estimators. {\em Annales de l'Institut Henri \\Poincar\'{e} (B)}, {\bf 38}(6), 907-921.
\item Gronwall, T. (1919): Note on the derivatives with respect to a parameter of the solutions of a system of differential equations. {\em Ann. of Math.}, {\bf 20}(2), 292-296.
\item Hall, P., Qian W. and Titterington, D.M. (1992): Ridge finding from noisy data. {\em J. Comp. Graph. Statist.}, {\bf 1}(3), 197-211.
\item Hastie, T. and Stuetzle, W. (1989): Principal curves. {\em J. Amer. Statist. Assoc.}, {\bf 84}(406), 502-516.
%
%
\item Koltchinskii, V., Sakhanenko, L. and Cai, S. (2007): Integral curves of noisy vector fields and statistical problems in diffusion tensor imaging: Nonparametric kernel estimation and hypotheses testing. {\em  Ann. Statist.}, {\bf 35}(4), 1576-1607.
\item Novikov, D., Colombi, S. and Dor\'{e}, O. (2006). Skeleton as a probe of the cosmic web: Two-dimensional case. {\em Monthly Notices of the Royal Astronomical Society}, {\bf 366}(4), 1201-1216.
\item  Mikhaleva, T.L. and Piterbarg, V.I. (1996): On the distribution of the maximum of a Gaussian field with constant variance on a smooth manifold. {\em Theory Probab. Appl.}, {\bf 41}(2), 367-379.
%
%
\item Ozertem, U. and Erdogmus, D. (2011): Locally defined principal curves and surfaces. {\em Journal of Machine Learning Research}, {\bf 12}, 1249-1286.
\item Pimbblet, K.A., Drinkwater, M.J. and Hawkrigg, M.C. (2004): Inter-cluster filaments of galaxies programme: abundance and distribution of filaments in
the 2dFGRS catalogue. {\em Monthly Notices of the Royal Astronomical Society}, {\bf 354}, L61-L65.
%
%
%
\item Piterbarg, V.I. and Stamatovich, S. (2001): On maximum of Gaussian non-centered fields indexed on smooth manifolds. In {\em Asymptotic methods in probability and statistics with applications (St. Petersburg)},  N. Balakrishnan, I.A. Ibragimov, V.B. Nevzorov (eds) 189-203, Birkh\"{a}user, Boston.
%
%
\item Qiao, W. (2013): On estimation of filament structures. {\em PhD Thesis, UC Davis}
\item Qiao, W. and Polonik, W. (2015a): Extrema of Gaussian fields on growing manifolds. {\em Submitted.}
\item Qiao, W. and Polonik, W. (2015b): Supplemental material for: theoretical analysis of nonparametric filament estimation. doi: COMPLETED BY THE TYPESETTER.
\item Rosenblatt, M. (1952): Remarks on a multivariate transformation. {\em Ann. Math. Statist.}, {\bf 23}(3), 470-472.
\item Rosenblatt, M. (1976): On the maximal deviation of $k$-dimensional density estimates. {\em Ann. Probab.}, {\bf 4}(6), 1009-1015.
%
%
%
\item van der Vaart, A.W. and Wellner, J.A. (1996): Weak Convergence and Empirical Process. Springer, New York.
\end{description}

\pagebreak

\begin{frontmatter}
\title{Supplemental material for: \protect \\ Theoretical Analysis of Nonparametric Filament Estimation}
%
\author{Wanli Qiao}
\and
\author{Wolfgang Polonik} %
\maketitle
%
\begin{abstract}
\noindent Due to page constraints on the main article this supplement presents the proofs of some technical results from Qiao and Polonik (2015) (Appendix A) as well as some miscellaneous results (Appendix B) that are used in the proofs. Appendix B also contains the derivation of the function $G$ and some of its properties that play an important role in the paper.
\end{abstract}
\end{frontmatter}
\begin{appendix}
\section{Technical proofs}
\section*{Theorem 3.2 and its proof} For convenience of the reader we first restate Theorem 3.2.\\

{\bf Theorem 3.2} {\em
Under assumptions (\textbf{F}1)--(\textbf{F}2), (\textbf{K}1)--(\textbf{K}2) and (\textbf{H}1), for any $x_0\in \mathcal{G}$, $0 < \gamma < \infty$ and $0 \leq T_{min},T_{max}< \infty,\,T_{min} + T_{max} \ne 0$ with $\{\X_{x_0}(t),\,t \in [-T_{min},T_{max}]\} \subset {\cal H}$ and
\begin{align}\label{Gammax0}
\inf_{-T_{min}\leq s< u\leq T_{max}}\bigg\|\frac{1}{u-s}\int_s^uV(\X_{x_0}(\lambda))d\lambda\bigg\|\geq\gamma,
\end{align}
the sequence of stochastic process $\sqrt{nh^5}(\hat \X_{x_0}(t)-\X_{x_0}(t))$, $-T_{min}\leq t\leq T_{max}$, converges weakly in the space $C[-T_{min},T_{max}]:=C([-T_{min},T_{max}],\mathbb{R}^2)$ of $\mathbb{R}^2$-valued continuous functions on $[-T_{min},T_{max}]$ to the Gaussian process $\omega(t), -T_{min}\leq t\leq T_{max}$, satisfying the SDE
\begin{align}\label{SDEMain}
d\omega(t)=&\frac{\sqrt{\beta}}{2}\tilde G(\X_{x_0}(t))v(\X_{x_0}(t))dt+\nabla V(\X_{x_0}(t))\omega(t)dt\nonumber\\
&+\bigg\{\tilde G(\X_{x_0}(t))\bigg[\int\int{\mathbb K}(\X_{x_0}(t),\tau,z)f(\X_{x_0}(t))dzd\tau\bigg]\tilde G(\X_{x_0}(t))^T \bigg\}^{1/2}dW(t)
\end{align}
with initial condition $\omega(0)=0$, where $W(t),t\geq0$ is a two-sided standard Brownian motion in $\mathbb{R}^2$,
\begin{align}
v(x) &=
\left(
\begin{array}{c}
\int K(z)z^T \nabla^2f^{(2,0)}(x)zdz \\
\int K(z)z^T \nabla^2f^{(1,1)}(x)zdz \\
\int K(z)z^T \nabla^2f^{(0,2)}(x)zdz \\
\end{array}
\right) \in \bbR^3, \label{BDefinition}\\[-15pt]
\intertext{and}
{\mathbb K}(x,\tau,z) &:=d^2 K(z)\big[d^2 K\big(\tau V(x)+z\big)\big]^T \in \bbR^{3 \times 3}.\label{PsiDefinition}
\end{align}
}

{\sc Proof.} The structure of the proof follows Koltchinskii et al. (2007). Let $\hat \Y_{x_0}(t)=\hat \X_{x_0}(t)-\X_{x_0}(t)$.  We will find sequences of stochastic processes $\hat \Z_{x_0}(t)\equiv \hat \Z_{x_0,n}(t)$ and $\hat \D_{x_0}(t)\equiv \hat \D_{x_0,n}(t)$ such that
\begin{align*}
\hat \Y_{x_0}(t)=\hat \Z_{x_0}(t)+\hat \D_{x_0}(t), \quad t\in[-T_{min},T_{max}],
\end{align*}
where $\sqrt{nh^5}\hat \Z_{x_0,n}(t), -T_{min}\leq t\leq T_{max}$, converges weakly to the Gaussian process $\omega(t)$ defined in (\ref{SDEMain}) and
\begin{align}
\sup_{-T_{min}\leq t\leq T_{max}}|\hat \D_{x_0}(t)|=o_p\bigg(\frac{1}{\sqrt{nh^5}}\bigg).\label{3.2}
\end{align}
This immediately implies the assertion of the theorem. For ease of notation we drop from now on the index $x_0$ in this proof. Write
\begin{align}\label{YIntegral}
\hat \Y(t)
=\int_0^t(\hat V-V)(\X(s))ds+\int_0^t\nabla V (\X(s))  \hat \Y(s)ds+\hat \R(t),
\end{align}
where
\begin{align*}
\hat \R(t)&:=\int_0^t[\hat V(\hat \X(s))-\hat V(\X(s))-\nabla V (\X(s))  \hat\Y(s)]ds.
\end{align*}

It is not difficult to see (by following the proof on page 1585 of Koltchinskii et al. 2007) that
\begin{align}\label{Reminder}
\sup_{-T_{min}\leq t\leq T_{max}}\|\hat \R(t)\|=o_p\bigg(\sup_{-T_{min}\leq t\leq T_{max}}\|\hat \Y(t)\|\bigg).
\end{align}

Suppose $\hat \Z$ satisfies the differential equation
\begin{align}\label{Z1}
\frac{d\hat \Z(t)}{dt}=\hat V(\X(t))-V(\X(t))+\nabla V (\X(t))\hat \Z(t), \quad \hat \Z(0)=0.
\end{align}
which means that
\begin{align}\label{ZIntegral}
\hat \Z(t)=\int_0^t[\hat V(\X(s)))-V(\X(s))]ds+\int_0^t\nabla V (\X(s))\hat \Z(s)ds.
\end{align}

Denote $\hat \D(t):=\hat \Y(t)-\hat \Z(t)$. Then from (\ref{YIntegral}) and (\ref{ZIntegral}) we have
\begin{align*}
\hat \D(t)=\int_0^t\nabla V (\X(s)) \hat \D(s)ds+\hat \R(t).
\end{align*}

Following the proof on page 1586 of Koltchinskii et al. (2007) we can show that
\begin{align*}
\sup_{-T_{min} \leq t \leq T_{max}}\|\hat \D(t)\|=o_p\bigg(\sup_{-T_{min} \leq t \leq T_{max}}\|\hat \Z(t)\|\bigg) \quad \textrm{as} \quad n\rightarrow \infty.
\end{align*}
As we will show below, the sequence $\sqrt{nh^5}\hat \Z(t), -T_{min} \leq t \leq T_{max}$, converges in distribution to the Gaussian process $\omega(t)$, it follows that
\begin{align*}
\sup_{-T_{min} \leq t \leq T_{max}}\|\hat \Z(t)\|=O_P\bigg(\frac{1}{\sqrt{nh^5}}\bigg).
\end{align*}
Immediately we get (\ref{3.2}).\\

We now show the asserted weak convergence of $\sqrt{nh^5}\hat \Z(t), -T_{min} \leq t \leq T_{max}$. Denote by $C_0^{(1)}[-T_{min},T_{max}]$ the set of all $\mathbb{R}^2$-valued continuously differentiable functions on $[-T_{min},T_{max}]$ with value zero at the point 0. We define a mapping $\mathscr{U}: C_0^{(1)}[-T_{min},T_{max}] \mapsto C[-T_{min},T_{max}]$ such that for any $S\in C_0^{(1)}[-T_{min},T_{max}]$, $\mathscr{U}S(t)$ is the solution of the following differential equation in $\mathbb{R}^2$:
\begin{align}\label{DiffEquDef}
\frac{du(t)}{dt}=\frac{dS(t)}{dt}+\nabla V (\X(t))u(t), \quad u(0)=0.
\end{align}
As indicated on page 1587 of Koltchinskii et al. (2007), $\mathscr{U}$ is a Lipschitz mapping with respect to the uniform distance.\\

Define a stochastic process $\chi$ satisfying the SDE
\begin{align}\label{EtaSDE}
d\chi(t)=&\frac{\sqrt{\beta}}{2}\tilde G(t) v(\X(t))dt +\bigg\{\tilde G(t)\bigg[\int\int\Psi(\X(t),\tau,z)f(\X(t))dzd\tau\bigg]\tilde G(t)^T \bigg\}^{1/2}dW(t)
\end{align}
with initial condition $\chi(0)=0$. Then $\mathscr{U}\chi$ satisfies (\ref{SDEMain}) with value zero at the point 0, i.e., $\mathscr{U}\chi=\omega$.\\

Denote two sequences of processes
\begin{align*}
&\chi_n(t):=\sqrt{nh^5}\int_0^t[\hat V(\X(s))-V(\X(s))]ds,\\
&\omega_n(t):=\sqrt{nh^5}\hat \Z(t).
\end{align*}
Then by (\ref{Z1}) we have $\omega_n=\mathscr{U}\chi_n$. As we will show below, the sequence $\chi_n$ converges weakly in the space $C[-T_{min},T_{max}]$ to $\chi$. Since $\mathscr{U}$ is Lipschitz, $\omega_n$ converges weakly to $\omega$, which is the assertion.\\

It thus remains to show the weak convergence of the sequence $\chi_n$.
In what follows we will write the explicit form of $\hat V(x)$ and $V(x)$, i.e., $\hat V(x)=G(d^2 \hat f(x))$ and $V(x)=G(d^2  f(x))$. 
%
%
Recall that $G=(G_1, G_2)^T$ and $\tilde G(x):=\nabla G(d^2  f(x))$. Denote $\mathcal{J}_n(t):=\int_0^t \tilde G(\X(s))d^2 \hat f(\X(s))ds$ and $\mathcal{J}(t):=\int_0^t\tilde G(\X(s))d^2  f(\X(s))ds$. A first order Taylor expansion with respect to the variables in $G$ gives
\begin{align}\label{LinearTaylor}
\int_0^t[\hat V(\X(s))-V(\X(s))]ds=\mathcal{J}_n(t) - \mathcal{J}(t)+\mathcal{R}_n(t),
\end{align}
where the remainder term
\begin{align*}
\mathcal{R}_n(t):=\left(\begin{array}{c} \int_0^td^2 (\hat f - f) (\X(s))^T M_1(s)d^2 (\hat f - f)  (\X(s))ds \\ \int_0^t d^2 (\hat f - f)  (\X(s))^T M_2(s)d^2 (\hat f - f)  (\X(s))ds\end{array}\right)
\end{align*}
with
\begin{align}\label{MiInt}
M_i(s):=\int_0^1\nabla^2G_i(d^2  f(\X(s))+\tau d^2 (\hat f - f)  (\X(s)))d\tau, \quad i=1,2.
\end{align}
For this remainder term we have 
\begin{align*}
\sup_{t \in [-T_{min},T_{max}]}\|\mathcal{R}_n(t)\|&\leq 2\sup\limits_{\substack{ t \in [-T_{min},T_{max}] \\ i=1,2}}\bigg|\int_0^t d^2 (\hat f - f)  (\X(s))^T M_i(s)d^2 (\hat f - f)  (\X(s))ds\bigg|\\
&\leq 2T\bigg(\sup_{x\in\mathbb{R}^2}\|d^2 (\hat f - f)  (x)\|\bigg)^2\sup\limits_{\substack{i=1,2 \\ w\in \mathcal{H}^\epsilon}}\|\nabla^2 G_i(w)\|=O_P\Big(\frac{\log n}{nh^6} \Big) = o_p\bigg(\frac{1}{\sqrt{nh^5}}\bigg),
\end{align*}
because $\sup_{x\in \mathcal{H}^\epsilon}\|d^2 (\hat f - f)  (x)\|=O_p(\sqrt{\log{n} /  nh^6})$  and $nh^8 / \log{n}\rightarrow\infty$ by assumption (\textbf{H}1).\\


The linear approximation $\mathcal{J}_n(t) - \mathcal{J}(t)$ in (\ref{LinearTaylor}) is a sum of iid random variables for each fixed $n$, indicating asymptotic normality. In fact, weak convergence of the process $\sqrt{nh^5}(\mathcal{J}_n(t) - \mathcal{J}(t))$ to $\chi(t)$ can be shown by proving convergence of finite dimensional distributions along with asymptotic stochastic equicontinuity by following the proof on pages 1591-1595 of Koltchinskii et al. (2007). Then we conclude the proof of this theorem. Further details are omitted.\hfill$\square$\\
%
\section*{The proof of (5.11)}
By our assumptions, the maps $K_\ell(z)$ are Lipschitz continuous. Let $c_\ell$ be the corresponding Lipschitz constants. Fix $\tau > 0,$ and let $x_{0},x^*_{0} \in {\mathcal{G}}$ be such that $\|x_{0} - x^*_{0}\| \le \tau$, and assume that $T_{x_0} \le T_{x_0^*}$.  We first show that there exists a constant $C > 0$ with
\begin{align}
&\sup_{x \in {\mathbb R}^2}\Big| \int_0^t K_\ell\Big(\frac{\X_{x_{0}} (s)-x}{h}\Big) -K_\ell\Big(\frac{\X_{x^*_{0}} (s)-x}{h}\Big) ds \Big| \le C\tau\qquad\ell = 1,2,3,\;\;0 < t < T_{x_0}.\label{fact1}
%
\end{align}
Recall that the support of $K(z)$ is contained in a unit ball $\mathscr{B}(0,1)$. Thus we have with $A_{x,x_0}(h) = \{s: \|\X_{x_{0}}(s) - x\| \le h\}$ that
\begin{align}
&\int_0^t \Big|K_\ell\Big(\frac{\X_{x_{0}}(s)-x}{h}\Big) -K_\ell\Big(\frac{\X_{x^*_{0}}(s)-x}{h}\Big)\Big| ds\nonumber \\
& = \int_0^t \Big|K_\ell\Big(\frac{\X_{x_{0}}(s)-x}{h}\Big) -K_\ell\Big(\frac{\X_{x^*_{0}}(s)-x}{h}\Big)\Big| \, {\bf 1}_{A_{x,x_{0}}(h) \cup A_{x,x^*_{0}}(h)}(s)\,ds\nonumber\\
&\le \int_0^t c_\ell \Big\| \frac{\X_{x_{0}}(s) -{\X_{x^*_{0} }(s)}}{h}\Big\|\,{\bf 1}_{A_{x,x_{0}}(h) \cup A_{x,x^*_{0}}(h)}(s)\,ds.\label{part1}
\end{align}
The Lebesgue measure of the set $A_{x,x_{0}}(h) \cup A_{x,x^*_{0}}(h)$ is of the order $O(h)$. To see that observe that for $s, s^\prime \in A_{x,x_{0}}(h)$ we have $\|\X_{x_{0}}(s) - \X_{x_{0}}(s^\prime)\| \le 2h$, so that with $\gamma_{\mathcal{G}} > 0$ from  (\ref{Gammax1}):
{\allowdisplaybreaks
\begin{align*}
2h \ge \|\X_{x_{0}}(s) - \X_{x_{0}}(s^\prime)\| & = \Big\| \int_s^{s^\prime} V(\X_{x_{0}}(t))\,dt\Big\| = \Big\| \frac{1}{s - s^\prime} \int_s^{s^\prime} V(\X_{x_{0}}(t)\,dt\Big\|\; |s - s^\prime| \ge \gamma_{\mathcal{G}} \,|s - s^\prime|.
\end{align*}}
It follows that ${\rm Leb}(A_{x,x_{0}}(h)) \le 2h/\gamma_{\mathcal{G}}$ and the same holds for $A_{x,x^*_{0}}(h)$, so that
\begin{align}\label{aaa}
{\rm Leb}(A_{x,x_{0}}(h) \cup A_{x,x^*_{0}}(h)) \le \frac{4h}{\gamma_{\mathcal{G}}}.
\end{align}
To continue the argument we will use the fact that $\X_{x_0}(s)$ is Lipschitz continuous in $x_0$ under the sup-norm. To see this note that for any $x_0, x_0^\prime\in\mathcal{G}$ and $s\in[0, \min{(T_{x_0}, T_{x_0^\prime})}]$,
\begin{align*}
&\|\X_{x_0}(s)-\X_{x_0^\prime}(s)\|=\bigg\|x_0-x_0^\prime+\int_0^s\big[ V(\X_{x_0}(t))-V(\X_{x_0^\prime}(t))\big] dt\bigg\| \\
&\hspace*{3cm}\leq \|x_0-x_0^\prime\|+\sup_{x\in\mathcal{H}}\|\nabla V(x)\|_F\int_0^s\|\X_{x_0}(t)-\X_{x_0^\prime}(t)\|dt.
\end{align*}
Applying Gronwall's inequality, we have for all $s \in [0, \min(T_{x_0}, T_{x^\prime_0})]$
\begin{align}\label{Lipschitz}
\|\X_{x_0}(s)-\X_{x_0^\prime}(s)\|\leq \|x_0-x_0^\prime\| \exp\bigg\{\min(T_{x_0}, T_{x^\prime_0})\sup_{x\in\mathcal{H}}\|\nabla V(x)\|_F\bigg\}.
\end{align}
%
By using (\ref{aaa}) and (\ref{Lipschitz}), the integral in (\ref{part1}) can now be bounded by
\begin{align*}
c_\ell\, \int_0^t \Big\| \frac{\X_{x_{0}} (s)- \X_{x^*_{0} }(s)}{h}\Big\| \,{\bf 1}_{A_{x,x_{0}} \cup A_{x,x^*_{0}}}(s)\,ds
\le &\;\;  c\,\frac{\tau}{h} c_\ell\, \int_0^t\,{\bf 1}_{A_{x,x_{0}} \cup A_{x,x^*_{0}}}(s)\,ds\\
\le &\;\; c\,\frac{\tau}{h}\; c_\ell \;\frac{4h}{\gamma_{\mathcal{G}}} = \frac{4 c\,c_\ell}{\gamma_{\mathcal{G}}}\;\tau,
\end{align*}
where $c =  \exp\bigg\{T_{\mathcal{G}}\sup_{x\in\mathcal{H}}\|\nabla V(x)\|_F\bigg\}$. We have verified (\ref{fact1}). Next we show (\ref{VCDef}).  Fix $0 < \tau \le 1.$ Since $x_0 \in {\mathcal{G}}$ and ${\mathcal{G}}$ is compact there exist points $x_{0,1},\ldots,x_{0,N_1}$ such that for all $x_0 \in {\mathcal{G}}$ we have $\min_{i = 1,\ldots,N_1}\|x_0 - x_{0,i}\| \le \tau$ and $N_1= N_1(\tau) \le A_1 \frac{1}{\tau}$ for some constant $A_1 > 0.$ Further, let $t_1,\ldots,t_{N_2}$ be such that for all $t \in [0,\max_{x_0 \in {\mathcal{G}}} T_{x_0}]$ we have $\min_{i=1,\ldots,N_2}|t - t_i| \le \tau$ and $N_2 = N_2(\tau) \le A_2\frac{1}{\tau}$ for $A_2 > 0$. With these definitions let
\begin{align*}
\widetilde {\cal F}_{j,\ell}(\tau) = \bigg\{\omega_{j,\ell}(\cdot;x_{0,i},t_k): i = 1,\ldots,N_1(\tau),\,t_k \le T_{x_{0,i}}, k \in \{1,\ldots,N_2(\tau)\} \bigg\}.
\end{align*}
It is clear that the number of functions in $\widetilde {\cal F}_{j,\ell}(\tau)$ is bounded by $C\,\big(\frac{1}{\tau}\big)^2.$ Without loss of generality we can assume that $T_{x_{0,i}} = \max\{T_{x_0}:\,\|x_0 - x_{0,i}\| \le \tau\},\,i = 1,\ldots,N_1(\tau).$ Otherwise define $x^*_{0,i} = \argmax\{T_{x_0}:\,\|x_0 - x_{0,i}\| \le \tau\}$, replace $x_{0,i} $ by $x^*_{0,i}$ and change $\tau$ to $2\tau$.\\

We show that $\widetilde{\cal F}_{j,\ell}(\tau)$ serves as an approximating class of functions assuring (\ref{VCDef}). To see that let $\omega_{j\ell}(\cdot;x_{0},t) \in {\cal F}_{j,\ell}$ and let $\omega_{j,\ell}(\cdot;x^*_{0},t^*) \in \widetilde {\cal F}_{j,\ell},$ where $\|x_0 - x_0^*\| \le \tau$ and $|t - t^*| \le \tau.$ To verify (\ref{VCDef}) we show that for some constant $C>0$ the following two bounds hold for $t \in [0,T_{x_0}]$:
\begin{align}
&d_{\infty}(\omega_{j,\ell}(\cdot;x_{0},t) , \omega_{j,\ell}(\cdot;x^*_{0},t)) \le C \tau,\label{bound1}\\
&d_{\infty}(\omega_{j,\ell}(\cdot;x^*_{0},t) , \omega_{j,\ell}(\cdot;x^*_{0},t^*)) \le C \tau\label{bound2},
\end{align}
where $d_\infty(f,g) $ denotes the supremum distance of functions $f$ and $g$. First we show (\ref{bound1}). Recall that $\widetilde G_j(u) = \nabla G_j(d^2  f(u)),$ and denote  $\widetilde G_{j,\ell} (u) = \frac{\partial}{\partial u_\ell}\widetilde{G}_j(u),\,\ell = 1,2,3.$ We obtain
\begin{align}
&d_{\infty}(\omega_{j,\ell}(\cdot;x_{0},t) , \omega_{j,\ell}(\cdot;x^*_{0},t))\nonumber \\[5pt]
  = &\;\;\sup_{x \in {\mathbb R}^2}\Big| \int_0^{t} \Big[\widetilde G_{j,\ell}(\X_{x_0}(s)) K_\ell\Big(\frac{\X_{x_0}(s)-x}{h}\Big) 
 %
 - \widetilde G_j(\X_{x^*_0}(s)) K_\ell \Big(\frac{\X_{x^*_0}(s)-x}{h}\Big)\Big]ds\Big|\nonumber \\[5pt]
 %
  %
  %
  \le &\;\;\sup_{x \in {\mathbb R}^2} \int_0^t \Big| \widetilde G_{j,\ell}(\X_{x_0}(s))-  \widetilde G_{j,\ell}(\X_{x^*_0}(s))\Big|  K_\ell\Big(\frac{\X_{x_0}(s)-x}{h}\Big) ds\label{double1}\\[5pt]
  &\hspace*{1cm}+  \sup_{x \in {\mathbb R}^2} \int_0^t   \Big|  K_\ell\Big(\frac{\X_{x_0}(s)-x}{h}\Big) - K_\ell\Big(\frac{\X_{x^*_0}(s)-x}{h}\Big)\Big|\; \big|\widetilde G_{j,\ell}(\X_{x^*_0}(s))\big|  ds.\label{double2}
\end{align}
Now, with $M_\ell = \sup_u K_\ell(u)$, the term in (\ref{double1}) can further be bounded by
\begin{align*}
&M_\ell \; \int_0^t \big| \widetilde G_{j,\ell}(\X_{x_0}(s))-  \widetilde G_{j,\ell}(\X_{x^*_0}(s))\big| ds  \le c^\prime\,M_\ell T_{\mathcal{G}}\, \tau,
\end{align*}
for some $c^\prime > 0$, where we are using Lipschitz continuity of $\widetilde G_{j,\ell}$ along with (\ref{Lipschitz}). To bound (\ref{double2}) we first use the fact that the functions $\widetilde G_{j,\ell}$ are bounded, so that the integral in (\ref{double2}) is less than or equal to
\begin{align*}
 \sup_{u}|\widetilde G_{j,\ell}(u)|\;\;\sup_{x \in {\mathbb R}^2}\int_0^t\;\Big| K_\ell\Big(\frac{\X_{x_0}(s)-x}{h}\Big) - K_\ell\Big(\frac{\X_{x^*_0}(s)-x}{h}\Big)\Big| ds \le C \sup_{u}|\widetilde G_{j,\ell}(u)|\, \tau,
\end{align*}
by using (\ref{fact1}). This shows (\ref{bound1}). The bound (\ref{bound2}) follows by using the boundedness of the integrant in the definition of the functions $\omega_{j,\ell}$ (see (\ref{def-varpi})). This completes the proof of (\ref{VCDef}).\hfill$\square$

\section*{Proposition 5.1 and its proof}
First we show uniform consistency of $\hat\theta_{x_0}$ which is needed for the proof of Proposition~\ref{ThetaRate}.

\begin{proposition}\label{ThetaConsistency}
Under assumptions (\textbf{F}1)--(\textbf{F}6), (\textbf{K}1)--(\textbf{K}2) and (\textbf{H}1), we have
\begin{align}
\sup_{x_0 \in \mathcal{G}}\big|\hat\theta_{x_0} - \theta_{x_0}\big|=o_p(1),\label{ThetaConsistency1} 
%
\end{align}
\end{proposition}
\begin{proof}

Fix $\epsilon>0$ arbitrary (and small enough). We want to show that $P(\sup_{x_0\in\mathcal{G}}|\hat \theta_{x_0} - \theta_{x_0}|>\epsilon)\rightarrow 0$ as $n\rightarrow \infty$. Recall that by definition,
\begin{align*}
\theta_{x_0}=\argmin_{t}\{|t|: \blangle \nabla f(\X_{x_0}(t)), V(\X_{x_0}(t)) \brangle = 0, \lambda_2(\X_{x_0}(t))<0\},
\end{align*}
and a similar definition holds for $\hat \theta_{x_0}$ (with $f, \X, V$ and $\lambda_2$ replaced by our estimates). Without loss of generality, consider $\theta_{x_0}>0$. Let $\rho_{x_0}$ be the first time traveling backwards that the trajectory hits the boundary of $\cal H$, i.e., $\rho_{x_0}=\inf\big\{s: s\leq \theta_{x_0},\;\; \{\X_{x_0}(t): t\in[s, \theta_{x_0}]\}\subset \cal H\big\}$. Note that $\X_{x_0}(\rho_{x_0})\in\cal H$ since $\cal H$ is compact. Under assumption ({\bf F}3), $\theta_{x_0}$ defined in (\ref{theta_def}) is unique, therefore for $\epsilon$ small enough, $\X_{x_0}(\theta_{x_0})$ is the only filament point on $[(-\theta_{x_0}-\epsilon)\vee \rho_{x_0},\theta_{x_0}]$. For any $x_0 \in\mathcal{G}$  let $\mathcal{C}_{x_0,\epsilon}=\{t\in[(-\theta_{x_0}-\epsilon)\vee \rho_{x_0}, \theta_{x_0}-\epsilon]: \blangle\nabla f(\X_{x_0}(t)), V(\X_{x_0}(t))\brangle=0\}$. Assume for now that $\mathcal{C}_{x_0,\epsilon} \ne \emptyset.$ Note that $\mathcal{C}_{x_0,\epsilon}$ is a compact set. For $\eta>0$ let $\mathcal{C}_{x_0,\epsilon}^\eta$ denote  the $\eta$-neighborhood of $\mathcal{C}_{x_0,\epsilon}$ intersected with $[(-\theta_{x_0}-\epsilon)\vee \rho_{x_0},\theta_{x_0}-\epsilon]$. It suffices to show that
\begin{itemize}
\item[(i)] $P(\forall x_0\in \mathcal{G}, \exists \; t_{x_0}\in[\theta_{x_0}-\epsilon, \theta_{x_0}+\epsilon]$ s.t. $\blangle\nabla \hat f(\hat \X_{x_0}(t_{x_0})), \hat V(\hat \X_{x_0}(t_{x_0}))\brangle=0)\rightarrow 1$,
\item[(ii)] $P(\sup_{x_0\in\mathcal{G}, t\in [\theta_{x_0}-\epsilon, \theta_{x_0}+\epsilon]}\hat\lambda_2(\hat \X_{x_0}(t))<0)\rightarrow 1$,
\item[(iii)] There exists an $\eta>0$ such that $P(\inf_{x_0\in\mathcal{G}, t\in \mathcal{C}_{x_0,\epsilon}^\eta} \hat\lambda_2(\hat \X_{x_0}(t))>0)\rightarrow 1$ and
\item[(iv)] $P(\inf_{x_0\in\mathcal{G}, t\in [(-\theta_{x_0}-\epsilon)\vee \rho_{x_0}, \theta_{x_0}-\epsilon)\backslash\mathcal{C}_{x_0,\epsilon}^\eta} |\blangle\nabla \hat f(\hat \X_{x_0}(t)), \hat V(\hat \X_{x_0}(t))\brangle|>0) \rightarrow 1$.
\end{itemize}
By our regularity assumptions $\{a_{x_0}(t) = \blangle\nabla f (\X_{x_0}(t)), V(\X_{x_0}(t))\brangle,\,x_0 \in {\mathcal{G}}\}$ and $\{\lambda_2(\X_{x_0}(t)), x_0 \in {\mathcal{G}}\}$ are classes of equi-continuous functions on $t \in [\rho_{x_0}, \theta_{x_0} + a^*].$ Further, under assumption (\textbf{F}5) $a_{x_0}(t)$ is strictly monotonic at $\theta_{x_0}$. Also the derivatives $a^\prime_{x_0}(t), x_0 \in {\mathcal{G}}$ form an equi-continuous class of functions, and thus for any $\epsilon > 0$ sufficiently small there exists a $\delta > 0$ such that $\inf_{x_0 \in \mathcal{G}}\inf_{t \in [\theta_{x_0} - \epsilon, \theta_{x_0} + \epsilon]}\big|a^\prime_{x_0}(t)\big|  > \delta.$\\[8pt]
Moreover, since $\theta_{x_0}$ corresponds to the first filament point, we have that for any $x_0\in\mathcal{G}$ and $t\in \mathcal{C}_{x_0,\epsilon}$, $\lambda_2(\X_{x_0}(t))>0$. Note that here we have used assumption (\textbf{F}6). Since both $\mathcal{G}$ and $\mathcal{C}_{x_0,\epsilon}$ are compact, there exist $\eta, \zeta>0$ with${\displaystyle \inf_{x_0\in\mathcal{G}, t\in [(-\theta_{x_0}-\epsilon)\vee \rho_{x_0}, \theta_{x_0}-\epsilon)\backslash\mathcal{C}_{x_0,\epsilon}^\eta} }|\blangle\nabla f(\X_{x_0}(t_{x_0})),  V(\X_{x_0}(t_{x_0}))\brangle|>\zeta$ and $\inf_{x_0\in\mathcal{G}, t\in \mathcal{C}_{x_0,\epsilon}^\eta} \lambda_2( \X_{x_0}(t))>\zeta$.\\[8pt]
The proofs of (ii) - (iv) are straightforward by using uniform consistency of $\hat\lambda_2(\hat \X_{x_0}(t))$ and $\blangle\nabla \hat f(\hat \X_{x_0}(t)), \hat V(\hat \X_{x_0}(t))\brangle$ along with the fact that the corresponding theoretical quantities satisfy the inequalities corresponding to the three probability statements from (ii) - (iv). Uniform consistency of  $\blangle\nabla \hat f(\hat \X_{x_0}(t)), \hat V(\hat \X_{x_0}(t))\brangle$ can be shown by using Theorem~\ref{UniformPath} and uniform consistency results for the kernel estimators of $\nabla f$ and $V$ under our assumptions. Uniform consistency of $\hat\lambda_2(\hat \X_{x_0}(t))$ is inherited from uniform consistency of the second derivatives of the kernel estimator and uniform consistency of $\hat\X_{x_0}(t)$ by observing that $\hat\lambda_2(\hat \X_{x_0}(t)) = J(d^2(\hat f(\hat \X_{x_0}(t)))$ with $J(\cdot)$ being Lipschitz-continuous. Further details are omitted. To see (i) first observe that since $\inf_{x_0 \in \mathcal{G}}\inf_{t \in [\theta_{x_0} - \epsilon, \theta_{x_0} + \epsilon]}\big|a^\prime_{x_0}(t)\big| > \delta$ there exists an $\eta >0$ and $t_1,t_2 \in [\theta_{x_0} - \epsilon,\theta_{x_0} + \epsilon]$ with $a_{x_0}(t_1) \ge \eta$ and $a_{x_0}(t_2) \le - \eta$ for all $x_0 \in \mathcal{G}$. Uniform consistency of \\[-20pt]
\begin{align}\label{ANX}
{\wh a}_{x_0}(t)=\blangle\nabla\hat f(\hat \X_{x_0}(t)), \hat V(\hat \X_{x_0}(t))\brangle
\end{align}
as an estimator of $a_{x_0}(t)$ implies that the probability of the event $B_n := \{ \text{for all } x_0 \in {\mathcal{G}}: \hat a_{x_0}(t_1) \ge \eta/2 \text{ and } \hat a_{x_0}(t_2) \le - \eta/2\}$ tends to one as $n\to \infty.$ Since $\hat a_{x_0}(t)$ is continuous, we have that on $B_n$ that for each $x_0 \in \mathcal{G}$ there exists a $t \in [\theta_{x_0} - \epsilon, \theta_{x_0} + \epsilon]$ with $\hat a_{x_0}(t) = 0.$ This completes the proof of (\ref{ThetaConsistency1}) in case $\mathcal{C}_{x_0,\epsilon} \ne \emptyset.$ If $\mathcal{C}_{x_0,\epsilon} = \emptyset$, then we can ignore (iii) and the result follows from (i),(ii) and (iv). 
%
%
%
%
\hfill$\square$\\
\end{proof}
The above proof also shows that the probability of  $\wh \Theta_{x_0} = \emptyset$ or $\hat \theta_{x_0}$ is not unique for all $x_0\in\mathcal{G}$ tends to zero as $n\to \infty.$ We will thus only consider the case of $\wh \Theta_{x_0}\neq \emptyset$ and unique $\hat \theta_{x_0}$ in what follows. Next we derive the convergence rate of $\sup_{x_0\in\mathcal{G}}|\hat \theta_{x_0}-\theta_{x_0}|$.\\

Now we prove Proposition~\ref{ThetaRate}. For the convenience of the reader we restate this propositon here.\\

{\bf Proposition~\ref{ThetaRate}} {\em Under assumptions (\textbf{F}1)--(\textbf{F}6), (\textbf{K}1)--(\textbf{K}2) and (\textbf{H}1), we have
\begin{align*}
\sup_{x_0\in\mathcal{G}}|\hat \theta_{x_0}-\theta_{x_0}|=O_p(\alpha_n),
\end{align*}
where $\alpha_n=\sqrt{\frac{\log{n}}{nh^6}}$, and if in addition $\sup_{x_0\in\mathcal{G}}\|\nabla f(\X_{x_0}(\theta_{x_0}))\|=0$, then $\alpha_n=\sqrt{\frac{\log{n}}{nh^5}}$.\\
}

{\sc Proof.} Note that $\wh a_{x_0}(t)$ is the directional derivative of $\hat f$ at $t$ when traversing the support of $\hat f$ along the curve $\hat \X_{x_0}(t).$ By definition of a filament point we have $\wh a_{x_0}(\hat \theta_{x_0}) = 0$ for all $x_0 \in {\mathcal{G}}.$ A similar interpretation holds for the population quantity $a_{x_0}(t).$ We will use the behavior of $\wh a_{x_0}(t) - a_{x_0}(t)$ around $t =  \theta_{x_0}$ to determine the behavior of $\hat \theta_{x_0}-\theta_{x_0}$. \\[8pt]
By using chain rule and noting that $\nabla \blangle\nabla f(x), V(x)\brangle=\nabla^2 f(x)V(x)+\nabla V(x)\nabla f(x)$ we have
\begin{align}\label{BX}
a_{x_0}^\prime(t) &= V(\X_{x_0}(t))^T \nabla^2 f(\X_{x_0}(t))V(\X_{x_0}(t))+\langle \nabla f(\X_{x_0}(t)), V( \X_{x_0}(t)) \rangle_{\nabla V(\X_{x_0}(t))}\nonumber \\
& = \lambda_2(\X_{x_0}(t)) \|V(\X_{x_0}(t))\|^2 + \langle \nabla f(\X_{x_0}(t)), V( \X_{x_0}(t)) \rangle_{\nabla V(\X_{x_0}(t))}
\end{align}
and similarly
\begin{align}\label{BXn}
\wh a_{x_0}^\prime(t) =  \hat \lambda_2(\hat \X_{x_0}(t)) \|\hat V(\hat\X_{x_0}(t))\|^2 +\langle \nabla \hat f(\hat \X_{x_0}(t)), \hat V( \hat \X_{x_0}(t))\rangle_{\nabla \hat V(\hat \X_{x_0}(t))}.
\end{align}
Notice that $a_{x_0}^\prime(t) = \tilde{a}^\prime(\X_{x_0}(t))$ with $\tilde a^\prime(x)$ from (\ref{def-atildeprime}). We can write
\begin{align}
0 = \wh a_{x_0}(\hat \theta_{x_0}) &= \blangle\nabla\hat f(\hat \X_{x_0}(\hat\theta_{x_0})), \hat V(\hat \X_{x_0}(\hat\theta_{x_0}))\brangle =\wh a_{x_0}(\theta_{x_0})+ \wh a_{x_0}^\prime(\hat \xi_{x_0})(\hat\theta_{x_0}-\theta_{x_0}) \label{ttt}
\end{align}
with some $\hat \xi_{x_0}$ between $\hat \theta_{x_0} $ and $\theta_{x_0}$. We next show that for some $\eta > 0$
\begin{align}\label{hatbn-conv}
P\big(\inf_{x_0 \in {\mathcal{G}}}|\wh a_{x_0}^\prime(\hat \xi_{x_0})|  \ge \eta\big) \to 1.
\end{align}
%
%
%
To this end we prove that
\begin{align}\label{anprime-Conv}
\sup_{x_0\in\mathcal{G}}\sup_{t \in (\theta_{x_0} - \epsilon,\theta_{x_0} + \epsilon)}|\wh a_{x_0}^\prime (t)-a_{x_0}^\prime(t)|=o_p(1)\quad\text{for }\; \epsilon > 0\;\;\text{sufficiently small},\\[-25pt]\nonumber
\end{align}
and
\begin{align}\label{aprime-cont}
\tilde{a}^\prime(x) = \lambda_2(x) \|V(x)\|^2 +\langle \nabla f(x), V(x)\rangle_{\nabla V(x) }\quad\text{is uniformly continuous in $x \in {\cal H}$}.
\end{align}
This then implies (\ref{hatbn-conv}) 
by using standard arguments. Assertion (\ref{aprime-cont})  is a direct consequence of our regularity assumptions that assure continuity of $\tilde{a}^\prime(x)$ by using compactness of $\cal H$. The consistency property (\ref{anprime-Conv}) follows from uniform consistency of $\hat \X_{x_0}(t)$ as an estimator for $\X_{x_0}(t)$ (Lemma~\ref{XConsist}) and uniform consistency of  $\hat V(x)$, $\nabla \hat V(x)$, $\nabla \hat f(x)$ and $\nabla^2 \hat f(x)$ (Lemmas~\ref{LX} and \ref{VXX}) by using a continuous mapping argument or arguments similar to the ones presented in the following proof of (\ref{nablafvx}).\\[8pt]
To prove the assertion of the proposition it remains to show that
\begin{align}\label{nablafvx}
\sup_{x_0\in\mathcal{G}}|\hat a_{x_0}(\theta_{x_0})|=O_p(\alpha_n).
\end{align}
To see this write

\begin{align*}
\hat a_{x_0}(\theta_{x_0}) = \hat a_{x_0}(\theta_{x_0}) - a_{x_0}(\theta_{x_0}) &= \blangle\nabla \hat f(\hat \X_{x_0}(\theta_{x_0})),\hat V(\hat \X_{x_0}(\theta_{x_0}))\brangle - \blangle\nabla f(\X_{x_0}(\theta_{x_0})),V(\X_{x_0}(\theta_{x_0}))\brangle \\
&= \blangle\big[ \nabla \hat f(\hat \X_{x_0}(\theta_{x_0})) - \nabla f(\X_{x_0}(\theta_{x_0})) \big], \hat V(\hat \X_{x_0}(\theta_{x_0}))\brangle \\
&\hspace*{2cm} + \blangle\nabla f(\X_{x_0}(\theta_{x_0})),\big[ \hat V(\hat \X_{x_0}(\theta_{x_0}))  -V(\X_{x_0}(\theta_{x_0})) \big]\brangle.
\end{align*}

The rate (\ref{nablafvx}) now follows from the following facts:\

\begin{align}
&\sup_{x_0 \in \mathcal{G}}\big\|\nabla {\hat f}(\hat \X_{x_0}(\theta_{x_0})) - {\mathbb E}\nabla \hat f (\X_{x_0}(\theta_{x_0}))\big\| = O_p\Big( \sqrt{\frac{\log n}{nh^5}}\Big), \label{rate1}\\[5pt]
&\sup_{x_0 \in \mathcal{G}}\big\| {\mathbb E}\nabla \hat f (\X_{x_0}(\theta_{x_0})) - \nabla f(\X_{x_0}(\theta_{x_0}))\big\|  = O(h^2),\label{rate1a}\\[2pt]
&\sup_{x_0 \in \mathcal{G}}\big\|\hat V(\hat \X_{x_0}(\theta_{x_0}))  -V(\X_{x_0}(\theta_{x_0})) \big\| = O_p\Big( \sqrt{\frac{\log n}{nh^6}}\Big), \label{rate2}\\[8pt]
&\text{ both $V(\X_{x_0}(\theta_{x_0}))$ and $ \nabla f(\X_{x_0}(\theta_{x_0}))$ are bounded uniformly in $x_0 \in \mathcal{G}$.}\label{boundedness}\\[-8pt] \nonumber
\end{align}

Properties (\ref{rate1}) - (\ref{rate2}) follow from Theorem~\ref{UniformPath}, well-known properties of kernel estimators (see Lemma~\ref{LX}) and Lemma~\ref{VXX}. Property (\ref{boundedness}) follows immediately from our regularity assumptions. The proof of Proposition~\ref{ThetaRate} is complete.\hfill$\square$

\section*{Proofs of Theorem 3.4 and Lemma 3.1}
{\bf Proof of Theorem~\ref{Approx-FilaDiff}} First write with $\tilde\theta_{x_0}$ between $\theta_{x_0}$ and $\hat \theta_{x_0},$
\begin{align}
\hat \X_{x_0}(\hat\theta_{x_0})-\X_{x_0}(\theta_{x_0}) &=\hat \X_{x_0}(\hat\theta_{x_0})-\hat \X_{x_0}(\theta_{x_0})+\hat \X_{x_0}(\theta_{x_0})-\X_{x_0}(\theta_{x_0})\nonumber\\
&=\hat V(\hat \X_{x_0}(\tilde\theta_{x_0}))[\hat\theta_{x_0}-\theta_{x_0}]+\hat \X_{x_0}(\theta_{x_0})-\X_{x_0}(\theta_{x_0})\nonumber\\
&= V(\X_{x_0}(\theta_{x_0})) [\hat\theta_{x_0}-\theta_{x_0}] + \big[ \hat V(\hat \X_{x_0}(\tilde\theta_{x_0})) - V(\X_{x_0}(\theta_{x_0})) \big]\,[\hat\theta_{x_0}-\theta_{x_0}] \nonumber\\
&\hspace*{2cm}+\hat \X_{x_0}(\theta_{x_0})-\X_{x_0}(\theta_{x_0}).\label{uu1}
\end{align}
We need a convergence rate of $\sup_{x_0\in\mathcal{G}}\|\hat V(\hat \X_{x_0}(\tilde\theta_{x_0}))-V(\X_{x_0}(\theta_{x_0}))\|.$ Let $a^*$ be as in assumption ($\mathbf F$3), and $\epsilon > 0$ arbitrary. On the set $\{\sup_{x_0 \in \mathcal{G}}|\wh \theta_{x_0} - \theta_{x_0}| < a^*\} \cap \{ \sup_{x_0 \in {\mathcal{G}}; t \in [\theta_{x_0} - a*, \theta_{x_0} + a*]}\|\hat \X_{x_0}(t) - \X_{x_0}(t)\| < \epsilon\}$ we have by recalling that ${\cal H}^\epsilon$ denotes the $\epsilon$-enlargement of ${\cal H},$
\begin{align}\label{VXTheta}
&\sup_{x_0\in\mathcal{G}}\|\hat V(\hat \X_{x_0}(\tilde\theta_{x_0}))-V(\X_{x_0}(\theta_{x_0}))\|\nonumber\\
&\hspace{0.5cm}\leq\; \sup_{x_0\in\mathcal{G}, t\in[\theta_{x_0} - a*, \theta_{x_0} + a*]}\|\hat V(\hat \X_{x_0}(t))-V(\X_{x_0}(t))\|+\sup_{x_0\in\mathcal{G}}\|V(\X_{x_0}(\tilde\theta_{x_0}))-V(\X_{x_0}(\theta_{x_0}))\|\nonumber\\
&\hspace{0.5cm}\leq\; \sup_{x_0\in\mathcal{G}, t\in[\theta_{x_0} - a*, \theta_{x_0} + a*]}\|\hat V(\hat \X_{x_0}(t))-V(\X_{x_0}(t))\|+\sup_{x\in\mathcal {H}}\|\nabla V(x)V(x)\|\sup_{x_0\in\mathcal{G}}|\tilde\theta_{x_0}-\theta_{x_0}|\nonumber\\
&\hspace{0.5cm}\leq\;\sup_{x\in\mathcal {H}^\epsilon}\|\hat V(x)-V(x)\| + \sup_{x_0 \in {\mathcal{G}}\atop t \in [\theta_{x_0} - a*, \theta_{x_0} + a*]}|V(\hat \X_{x_0}(t)) - V(\X_{x_0}(t))| \nonumber \\
&\hspace*{5cm} +\sup_{x\in\mathcal {H}}\|\nabla V(x)V(x)\|\sup_{x_0\in\mathcal{G}}|\tilde\theta_{x_0}-\theta_{x_0}|  
%
=\;O_p\bigg(\sqrt{\frac{\log{n}}{nh^6}}\bigg)
\end{align}
by using Theorem~\ref{UniformPath}, Proposition~\ref{ThetaRate} and Lemma~\ref{VXX}. It follows from (\ref{uu1}), (\ref{VXTheta}), Theorem~\ref{UniformPath} and Proposition~\ref{ThetaRate} that
\begin{align*}
\sup_{x_0 \in \mathcal{G}}\big\| \big[\hat \X_{x_0}(\hat\theta_{x_0})-\X_{x_0}(\theta_{x_0})\big] -  V(\X_{x_0}(\theta_{x_0})) [\hat\theta_{x_0}-\theta_{x_0}] \big\| = O_p\Big(\textstyle{\frac{\log{n}}{nh^6}}\Big) + O_p\Big( \textstyle{\sqrt{\frac{\log{n}}{nh^5}}}\Big) = O_p\Big(\textstyle{ \sqrt{\frac{\log{n}}{nh^5}}}\Big).
\end{align*}
\hfill$\square$\\ 

{\bf Proof of Lemma~\ref{Theta}}
We continue using the notation introduced in (\ref{ANX}) - (\ref{BXn}) and (\ref{aprime-cont}). Since assumption (\textbf{F}5) implies that $0 < \inf_{x_0 \in \mathcal{G}}|a_{x_0}^\prime(\theta_{x_0})| < \infty$ (see discussion given after the assumptions) we obtain from (\ref{ttt}) and (\ref{anprime-Conv}) that
\begin{align*}
\hat \theta_{x_0}-\theta_{x_0}=-\frac{\wh a_{x_0}(\theta_{x_0})}{a_{x_0}^\prime(\theta_{x_0})}  + O_p(R_n),
\end{align*}
where $|R_n| \le \sup_{x_0\in\mathcal{G}}|\wh a_{x_0}^\prime(\hat \xi_{x_0})-a_{x_0}^\prime(x_0)|\,|\hat \theta_{x_0}-\theta_{x_0} |.$ Using Proposition~\ref{ThetaRate}, the assertion of the lemma is now a consequence of assumption ({\bf H}1) as well as\\[-15pt]
\begin{align}\label{sss2}
\sup_{x_0\in\mathcal{G}}|\wh a_{x_0}^\prime(\hat \xi_{x_0})-a_{x_0}^\prime(\theta_{x_0})| =O_p\bigg(\sqrt{\frac{\log{n}}{nh^8}}\,\bigg),\\[-25pt]\nonumber
\end{align}
%
\begin{align}
\sup_{x_0 \in \mathcal{G}}\Big|\, \wh a_{x_0}(\theta_{x_0}) - \blangle\nabla f(\X_{x_0}(\theta_{x_0})), \hat V(\X_{x_0}(\theta_{x_0}))\brangle\,\Big| = O_p\bigg(\sqrt{\frac{\log{n}}{nh^5}}\,\bigg) \label{ttt3}
\end{align}
and\\[-30pt]
\begin{align}
\sup_{x_0 \in \mathcal{G}}\Big|\, \hat \varphi_{1n}(\X_{x_0}(\theta_{x_0})) -  \frac{\blangle \nabla f(\X_{x_0}(\theta_{x_0})), \hat V(\X_{x_0}(\theta_{x_0}))\brangle}{a^\prime_{x_0}(\theta_{x_0})}\,\Big| = O_p\Big(\frac{1}{nh^7} \Big). \label{ttt2}
\end{align}
To see (\ref{sss2}), observe that on $A_n(\epsilon) = \{ |\hat \theta_{x_0} - \theta_{x_0} | \le \epsilon\},\; \epsilon > 0,$ uniformly in $x_0 \in \mathcal{G}$
\begin{align*}
\big| \wh a_{x_0}^\prime(\hat \xi_{x_0}) - a_{x_0}^\prime(\theta_{x_0})\big| &\le  \sup_{t \in (\theta_{x_0} - \epsilon,\theta_{x_0} + \epsilon)}|\wh a_{x_0}^\prime(t)-a_{x_0}^\prime(t)| + \sup\limits_{\substack{|s -t| \le \epsilon,\\ s,t\in(\theta_{x_0} - \epsilon,\theta_{x_0}+\epsilon)}}|a_{x_0}^\prime(s) - a_{x_0}^\prime(t)|\\
&\le  \sup_{t \in (\theta_{x_0} - \epsilon,\theta_{x_0} + \epsilon)}|\wh a_{x_0}^\prime(t)-a_{x_0}^\prime(t)| + C \epsilon
\end{align*}
for some $C > 0$, where the last inequality follows from the fact that the derivatives of $a_{x_0}^\prime(t)$ are continuous in both $t$ and $x_0$ (note here that Lipschitz continuity of $\X_{x_0}(s)$ in $x_0$ is shown in (\ref{Lipschitz}), 
 and thus $a_{x_0}^\prime(t)$ is uniformly bounded for $t \in (\theta_{x_0} - \epsilon,\theta_{x_0} + \epsilon)$ and $x_0 \in \cal G$). Proposition~\ref{ThetaRate} 
 implies that with $\alpha_n$ from Proposition~\ref{ThetaRate} 
 we have $P\big( A_n(\alpha_n)\big) \to 1$ as $n \to \infty,$ and clearly $\alpha_n = o\Big(\sqrt{\frac{\log{n}}{nh^8}}\;\Big).$ It thus remains to show that\\[-15pt]
\begin{align}\label{sss1}
\sup_{x_0\in\mathcal{G}}\sup_{t \in (\theta_{x_0} - \epsilon,\theta_{x_0} + \epsilon)}|\wh a_{x_0}^\prime(t)-a_{x_0}^\prime(t)| = O_P\bigg(\sqrt{\frac{\log{n}}{nh^8}}\bigg).
\end{align}
We have
\begin{align}
&|\wh a_{x_0}^\prime(t) - a^\prime_{x_0}(t)| \label{ttt1}\\
&\hspace*{1cm}\le \Big|\, \big\| \hat V(\hat\X_{x_0}(t))\big\|_{\nabla^2 \hat f(\hat\X_{x_0}(t))} - \big\|V(\X_{x_0}(t))\big\|_{\nabla^2 f(\X_{x_0}(t))} \Big| \nonumber\\
%
&\hspace*{1.5cm}+\big|\blangle \nabla \hat f(\hat \X_{x_0}(t)),\,\hat V( \hat \X_{x_0}(t)) \brangle_{\nabla \hat V(\hat \X_{x_0}(t)) } - \blangle\nabla f(\X_{x_0}(t)),\, V( \X_{x_0}(t))\brangle_{\nabla V(\X_{x_0}(t))} \big| \nonumber
\end{align}
and we will derive the rate of convergence for each of the terms on the right-hand side. As for the first term, we have by using a telescoping argument:
\begin{align*}
&\Big|\, \big\| \hat V(\hat\X_{x_0}(t))\big\|_{\nabla^2 \hat f(\hat\X_{x_0}(t))} - \big\|V(\X_{x_0}(t))\big\|_{\nabla^2 f(\X_{x_0}(t))} \Big| \\
&\hspace*{2cm} \le \big|\, \blangle \hat V(\hat\X_{x_0}(t)),\,\hat V(\hat\X_{x_0}(t)) - V(\X_{x_0}(t))\brangle_{\nabla^2 \hat f(\hat\X_{x_0}(t)) }\,\big|\\
&\hspace*{3cm}+ \big|\, \blangle\hat V(\hat\X_{x_0}(t)),\,V(\X_{x_0}(t))\brangle_{\nabla^2 \hat f(\hat \X_{x_0}(t)) - \nabla^2 f(\X_{x_0}(t)) } \,\big|\\
& \hspace*{4cm}+ \big| \, \blangle \hat V(\hat\X_{x_0}(t)) - V(\X_{x_0}(t)),\,V(\X_{x_0}(t))\brangle_{\nabla^2 f(\X_{x_0}(t)) }  \big|.
\end{align*}
Now we can use similar arguments as in the proof of (\ref{nablafvx}). 
The asserted rate then is the slowest of the rates of convergence of $\sup_{x_0 \in \cal G}\sup_{t \in (\theta_{x_0} - \epsilon,\theta_{x_0}+\epsilon)}\big\|\hat \X_{x_0}(t) - \X_{x_0}(t)\big\|$ and the rates of $\sup_{x \in {\cal H}^\epsilon} \big\| \nabla^2 \hat f(x) - \nabla^2 f(x)\,\big\|_F$  and $\sup_{x \in {\cal H}^\epsilon} \big\| \hat V(x) - V(x)\,\big\|$, respectively (see Theorem 3.3 and Lemmas~\ref{LX} and~\ref{VXX}), where $\epsilon > 0$ is arbitrary. 
 The resulting rate for the last term on the right-hand side of the above inequality is $O_P\big(\sqrt{\frac{\log n}{nh^6}}\big).$ A similar argument applied to the second term on the right-hand side of (\ref{ttt1}) gives the (slower) asserted rate $O_P\big(\sqrt{\frac{\log n}{nh^8}}\big)$ in (\ref{sss1}) inherited by the rate of $\sup_{x \in {\cal H}^\epsilon}\big\|\nabla \hat V(x) - \nabla V(x)\big\|_F$ (see Lemma~\ref{VXX}), 
 which in turn is inherited from the rate of convergence of $\sup_{x \in {\cal H}^\epsilon}\big\|\nabla^2 \hat f(x) - \nabla^2 f(x)\big\|_F$. This proves (\ref{sss2}).\\

In order to see (\ref{ttt3}), first observe that
\begin{align*}
&\wh a_{x_0}(\theta_{x_0}) - \blangle\nabla f(\X_{x_0}(\theta_{x_0})), \hat V(\X_{x_0}(\theta_{x_0}))\brangle \\
&\hspace*{2cm}= \blangle\nabla \hat f(\hat \X_{x_0}(\theta_{x_0})), \hat V(\hat \X_{x_0}(\theta_{x_0}))\brangle -  \blangle\nabla f(\X_{x_0}(\theta_{x_0})), \hat V(\X_{x_0}(\theta_{x_0}))\brangle.
\end{align*}
Since $\sup_{x_0 \in \mathcal{G}}\big\|\hat \X_{x_0}(\theta_{x_0}) - \X_{x_0}(\theta_{x_0})\| = O_p\Big(\sqrt{\frac{\log n}{nh^5}}\Big)$  (Theorem~\ref{UniformPath}),  and $\sup_{x_0 \in \mathcal{G}}\big\|\nabla \hat f(\hat \X_{x_0}(\theta_{x_0})) - \nabla f(\X_{x_0}(\theta_{x_0}))\big\| = O_p\Big(\sqrt{\frac{\log n}{nh^5}}\Big) $ (see (\ref{rate1}), (\ref{rate1a}) and ({\bf H}1)), and since $d^2\hat{f}$ is uniformly consistent,  it is straightforward to see that
\begin{align*}
\sup_{x_0 \in \mathcal{G}}\Big| \wh a_{x_0}(\theta_{x_0})  -  \blangle\nabla f(\X_{x_0}(\theta_{x_0})), \hat V(\X_{x_0}(\theta_{x_0}))\brangle\Big| = O_p\Big(\sqrt{\frac{\log n}{nh^5}}\Big),
\end{align*}
This is (\ref{ttt3}). To see (\ref{ttt2}) observe that with
\begin{align*}
\widehat W_n(x)&=\langle \nabla f(x), d^2 (\hat f - f)  (x)\rangle_{ \nabla G(d^2  f(x))},\,
\end{align*}
we have $\hat \varphi_{1n}(\X_{x_0}(\theta_{x_0})) = \frac{\widehat W_n(\X_{x_0}(\theta_{x_0})) - {\mathbb E}\widehat W_n(\X_{x_0}(\theta_{x_0}))}{a^\prime_{x_0}(\theta_{x_0})},$ so that the assertion follows from
\begin{align}\label{t1}
\sup_{x_0 \in \mathcal{G}}\Big|\, {\mathbb E} \widehat W_n(\X_{x_0}(\theta_{x_0})) \,\Big| = O\Big(\frac{1}{nh^7}\Big)\\[-23pt]\nonumber
\end{align}
and\\[-25pt]
\begin{align}\label{t2}
\sup_{x_0 \in \mathcal{G}}\big|\, \widehat W_n(\X_{x_0}(\theta_{x_0})) - \blangle\nabla f(\X_{x_0}(\theta_{x_0})), \hat V(\X_{x_0}(\theta_{x_0}))\brangle\,\big| = O_p\Big(\frac{\log n}{nh^6}\Big).
\end{align}
To see (\ref{t1}) we use  $\sup_{x \in {\cal H}^\epsilon}\|{\mathbb E}d^2\hat f(x) - d^2f(x)\| = O(h^2)$, which follows by standard arguments. Since ${\mathbb E} \widehat W_n(\X_{x_0}(\theta_{x_0}))$ is a linear combination of the components of bias vector it is of the same order. Assumptions ({\bf H}1) assures that $h^2 = O(\frac{1}{nh^7}).$\\

As for (\ref{t2}), recall our notation $V(x) = G(d^2 f(x))$ and $\hat V(x) = G(d^2\hat f(x)).$ We see that
\begin{align*}
\blangle\nabla f(\X_{x_0}(\theta_{x_0})), \hat V(\X_{x_0}(\theta_{x_0}))\brangle&=\nabla f(\X_{x_0}(\theta_{x_0}))^T\;[\hat V(\X_{x_0}(\theta_{x_0}))-V(\X_{x_0}(\theta_{x_0}))]\\
&\hspace*{-4cm}=\nabla f(\X_{x_0}(\theta_{x_0}))^T\Big[ \int_0^1 \nabla G\big[\big(d^2 \hat f + \lambda d^2 (\hat f - f)\big)(\X_{x_0}(\theta_{x_0}))\big]\,d\lambda\;\Big] \;d^2 (\hat f - f)  (\X_{x_0}(\theta_{x_0})).\nonumber
\end{align*}
%
%
 %
%
Using standard arguments we obtain
\begin{align*}
&\sup_{x_0 \in \mathcal{G}}\Big|\widehat W_n(\X_{x_0}(\theta_{x_0})) - \blangle\nabla f(\X_{x_0}(\theta_{x_0})), \hat V(\X_{x_0}(\theta_{x_0}))\brangle\Big| \\
&\hspace*{3cm} = O_p\Big( \sup_{x_0 \in \mathcal{G}}\big| d^2\hat f(\X_{x_0}(\theta_{x_0})) - d^2f(\X_{x_0}(\theta_{x_0}))\big|^2 \Big) = O_p\Big(\frac{\log n}{nh^6}\Big).
\end{align*}
This completes the proof of (\ref{ThetaNormal}). To prove (\ref{ThetaPart}) where $\sup_{x_0\in\mathcal{G}}\|\nabla f(\X_{x_0}(\theta_{x_0}))\|=0$, we approximate $\hat a_{x_0}(\theta_{x_0})$ by $\blangle\nabla f(\hat\X_{x_0}(\theta_{x_0})), V(\X_{x_0}(\theta_{x_0}))\brangle$ rather than by $\blangle\nabla f(\X_{x_0}(\theta_{x_0})), V(\hat \X_{x_0}(\theta_{x_0}))\brangle$ as we did above. We also will have to deal with the bias of $\nabla \hat f(\X_{x_0}(\theta_{x_0}))$ because this is not negligible here. Let $\mu_n(x) = {\mathbb E} \nabla \hat f(x)$. Then a simple telescoping argument gives
\begin{align}
&\sup_{x_0\in\mathcal{G}}\Big|\hat a_{x_0}(\theta_{x_0})-\blangle\mu_n(\hat\X_{x_0}(\theta_{x_0})), V(\X_{x_0}(\theta_{x_0}))\brangle\Big|\nonumber \\
%
%
&\hspace{1cm}\leq \sup_{x_0\in\mathcal{G}}\Big\|\nabla \hat f(\hat \X_{x_0}(\theta_{x_0}))\Big\|\;\sup_{x_0\in\mathcal{G}}\Big\|\hat V(\hat \X_{x_0}(\theta_{x_0}))-V(\X_{x_0}(\theta_{x_0}))\Big\|\nonumber \\
%
%
&\hspace{1.5cm}+\sup_{x_0\in\mathcal{G}}\Big\|\nabla \hat f(\hat \X_{x_0}(\theta_{x_0}))-\mu_n(\hat\X_{x_0}(\theta_{x_0}))\Big\|\;\sup_{x_0\in\mathcal{G}}\Big\|V(  \X_{x_0}(\theta_{x_0}))\Big\| 
%
%
=O_p\bigg(\sqrt{\frac{\log{n}}{nh^4}}\bigg)\label{Teil1}
\end{align}
%
by using (\ref{rate1}) and (\ref{rate2}) and Lemma~\ref{LX}.
%
%
Further, a one-term Taylor expansion gives
\begin{align}
\blangle\nabla f(\hat\X_{x_0}(\theta_{x_0})), V(\X_{x_0}(\theta_{x_0}))\brangle &= \blangle\nabla f(\hat \X_{x_0}(\theta_{x_0})), V(\X_{x_0}(\theta_{x_0}))\brangle - \blangle\nabla f(\X_{x_0}(\theta_{x_0})), V(\X_{x_0}(\theta_{x_0}))\brangle\nonumber \\[3pt]
%
%
& \hspace*{-1.5cm} =  \blangle V(\X_{x_0}(\theta_{x_0})) , \hat \X_{x_0}(\theta_{x_0}) - \X_{x_0}(\theta_{x_0})\brangle_{\nabla^2 f(\X_{x_0}(\theta_{x_0}))} + r_n
\end{align}
where $ r_n = \blangle V(\X_{x_0}(\theta_{x_0})) , \hat \X_{x_0}(\theta_{x_0}) - \X_{x_0}(\theta_{x_0})\brangle_{\hat{A}_n}$ %
with $\hat A_n = \int_0^1\big[\nabla^2 f(\X_{x_0}(\theta_{x_0}) + \lambda(\X_{x_0}(\hat \theta_{x_0}) - \X_{x_0}(\theta_{x_0}))) - \nabla^2 f(\X_{x_0}(\theta_{x_0}))\big]\,d\lambda,$ and we have\\[-20pt]
\begin{align}\label{Teil2}
r_n = O_p\Big(\, \frac{\log n}{nh^5} \,\Big).
\end{align}
The rate  is uniform in $x_0 \in \mathcal{G}$ and follows from Theorem~\ref{UniformPath} and our regularity assumptions.  Using (\ref{Teil1}) and (\ref{Teil2}) the assertion now follows by using similar arguments as in the first part of the proof.\hfill$\square$\\

\section*{Proofs of Corollaries 3.1 - 3.3} 
 We first prove Corollary 3.2. It follows from Theorem 3.2 that as $n\rightarrow\infty$,
\begin{align}\label{Pointwise}
\sqrt{nh^5}(\hat\X_{x_0}(\theta_{x_0}) - \X_{x_0}(\theta_{x_0}))\rightarrow \mathscr{N}(m(\theta_{x_0}), \Sigma(\theta_{x_0}))
\end{align}
where $m(\cdot)\in\mathbb{R}^2$ and $\Sigma(\cdot)\in \mathbb{R}^{2\times2}$ satisfy the ODE
\begin{align*}
\dot{m}(t)&=\frac{\sqrt{\beta}}{2}\tilde G(t)\mathcal {B}(\X_{x_0}(t))+\tilde G(t)m(t),\\
\dot{\Sigma}(t)&=\nabla V(\X_{x_0}(t))\Sigma(t)+\Sigma(t)[\nabla V(\X_{x_0}(t))]^T\\
&\hspace{3cm}+\tilde G(t)\bigg[\int\int\Psi(\X_{x_0}(t),\tau,z)f(\X_{x_0}(t))dzd\tau\bigg]\tilde G(t)^T
\end{align*}
with $m(0)=x_0$ and $\Sigma(0)=0$. Here $\tilde G(t)$, $\mathcal {B}(\X_{x_0}(t))$ and $\Psi(\X_{x_0}(t),\tau,z)$ have been defined in Theorem 3.2.\\

Observe that  by assumption (\textbf{H}1) we have $\log n /nh^{13/2} = o(1/\sqrt{nh^5})$ and $\sqrt{nh^5} h^{2} \to \beta \ge 0.$ Thus the assertion follows from (\ref{Pointwise}) above and (\ref{XDiffApp}), 
by observing that standard arguments show that
\begin{align}\label{Teil3}
h^{-2}\big(({\mathbb E}\nabla \hat f - \nabla f)(\hat\X_{x_0}(\theta_{x_0})) \big) \to \tilde{b}(\X_{x_0}(\theta_{x_0})),
\end{align}
where $\tilde{b}(x) = \frac{1}{2}\mu_2(K)\,\big(f^{(3,0)}(x) +f^{(1,2)}(x),\,f^{(0,3)} + f^{(2,1)}(x)\big)^T.$\\

Next we prove Corollary 3.1. 
By definition of $\hat \varphi_{1n}(x)$ (see (\ref{def-a1n})) 
  the approximation  (\ref{XDiffApp2}) 
 shows that in this case the convergence properties of $\hat \X_{x_0}(\hat \theta_{x_0})-\X_{x_0}(\theta_{x_0})$ are essentially determined by the deviation $d^2\hat f(\X_{x_0}(\theta_{x_0}))-\mathbb{E}d^2 \hat f(\X_{x_0}(\theta_{x_0}))$. The behavior of this difference is well known. With $\mathbf{R}:=\mathbf{R}(d^2 K)$ where $\mathbf{R}({\cdot})$ defined in assumption (\textbf{K}2), we have that under our assumptions (e.g. see Duong et al. 2008)
\begin{align}\label{as-norm-nabla2f}
\sqrt{nh^6}(d^2\hat f(x)-\mathbb{E}d^2 \hat f(x))\rightarrow \mathscr{N}(0,f(x)\mathbf{R}) \quad \text{ as }\; n\rightarrow\infty, \quad\text{in distribution.}
\end{align}
It follows from (\ref{XDiffApp2}) 
and assumption $({\bf H}1)$ that $\hat \X_{x_0}(\hat\theta_{x_0})-\X_{x_0}(\theta_{x_0})$ has the same asymptotic distribution as $\hat \varphi_{1n}(\X_{x_0}(\theta_{x_0}))\,V(\X_{x_0}(\theta_{x_0}))$. The assertion now follows by standard arguments.\\

Corollary 3.3 
follows immediately from (\ref{XDiffApp2}) 
by observing that
%
%
\begin{align*}
O_P\Big (\frac{\log n}{nh^7}\Big) &= \sup_{x_0 \in {\cal G}}\Big|\Big( \hat \X_{x_0}(\hat\theta_{x_0})-\X_{x_0}(\theta_{x_0})+\hat\varphi_{1n}(\X_{x_0}(\theta_{x_0}))V(\X_{x_0}(\theta_{x_0}))\Big)^T V^\bot(\X_{x_0}(\theta_{x_0}))\Big| \\
&= \sup_{x_0 \in {\cal G}}\Big|\Big( \hat \X_{x_0}(\hat\theta_{x_0})-\X_{x_0}(\theta_{x_0})\Big)^T V^\bot(\X_{x_0}(\theta_{x_0}))\Big|.
\end{align*}
\hfill$\square$

\section*{Proof of continuity of $\boldsymbol{\theta_{\scriptscriptstyle x_0}}$ as a function in $\boldsymbol{{\scriptstyle x_0}\in\mathcal{G}}$.} 

In our formal proofs we repeatedly apply Theorem 3.3 
with $T_{x_0}^{min} = \theta_{x_0} - a^*$ and $T_{x_0}^{max} = \theta_{x_0} + a^*$, and so we need to know that $x_0 \to \theta_{x_0}$ is continuous. We will show this here. First we indicate that filament points are continuous in the starting points of paths, namely, $x_0 \to \X_{x_0}(\theta_{x_0}),\,x_0 \in {\cal G}$ is continuous. To see this, recall that the filament ${\cal L} = \{\X_{x_0}(\theta_{x_0}),\,x_0 \in {\cal G}\} $ by assumption is a smooth curve with bounded curvature. For a given $x^*_0 \in {\cal G}$ and $\epsilon > 0$, consider the set ${\cal L}_{x^*_0}(\epsilon) = \{x \in {\cal L}:\,\|\X_{x^*_0}(\theta_{x^*_0}) - x\| \le \epsilon\}.$ This curve (piece of the filament) has finite length. Now let ${\cal G}_{x^*_0} = \{x_0 \in {\cal G}: x_0 \rightsquigarrow {\cal L}_{x^*_0}(\epsilon)\}$ denote the set of all points on integral curves $\X_{x_0}(t)$ passing through a filament point on ${\cal L}_{x_0^*}(\epsilon).$  By definition, $x^*_0 \in {\cal G}_{x_0^*}$. Since integral curves are non-overlapping, the set ${\cal G}_{x_0^*}$ is delineated by two `boundary curves' $\X^L$ and $\X^U$, say, corresponding to the filament points on the endpoints of the curve ${\cal L}_{x_0^*}(\epsilon).$ Let $\delta = \inf\{\|x_0^* - x\|, x \in \X^L \cup \X^U\}.$ Then, $\delta > 0$. (Otherwise $x_0^*$ would lie on one of the boundary curves.) With this $\delta$ we have, $\|x - x_0^*\| < \delta \Rightarrow \|\X_{x_0^*}(\theta_{x_0^*}) - \X_{x}(\theta_x)\| < \epsilon$ by construction.\\

Now let $x_0, x_0^\prime \in \mathcal{G}$ and without loss of generality assume $\theta_{x_0} \geq \theta_{x_0^\prime}$. Then we have
\begin{align*}
\Big(\X_{x_0}(\theta_{x_0}) - \X_{x_0^\prime}(\theta_{x_0^\prime})\Big)+\Big(\X_{x_0^\prime}(\theta_{x_0^\prime}) - \X_{x_0}(\theta_{x_0^\prime}) \Big) & = \X_{x_0}(\theta_{x_0}) - \X_{x_0}(\theta_{x_0^\prime})\\
& = (\theta_{x_0} - \theta_{x_0^\prime}) \bigg( \frac{1}{\theta_{x_0} - \theta_{x_0^\prime}} \int_{\theta_{x_0^\prime}}^{ \theta_{x_0}} V(\X_{x_0}(t))dt \bigg).
\end{align*}
Taking $L^2$ norm on both sides it follows that
\begin{align}\label{ThetaCont}
|\theta_{x_0} - \theta_{x_0^\prime}| \bigg\| \frac{1}{\theta_{x_0} - \theta_{x_0^\prime}} \int_{\theta_{x_0^\prime}}^{ \theta_{x_0}} V(\X_{x_0}(t))dt \bigg\| \leq \Big\|\X_{x_0}(\theta_{x_0}) - \X_{x_0^\prime}(\theta_{x_0^\prime})\Big\|+\Big\|\X_{x_0^\prime}(\theta_{x_0^\prime}) - \X_{x_0}(\theta_{x_0^\prime})\Big\|.
\end{align}
Similar to (\ref{Lipschitz}), 
we can show that there exists a constant $C>0$ such that $\|\X_{x_0^\prime}(\theta_{x_0^\prime}) - \X_{x_0}(\theta_{x_0^\prime})\|\leq C\|x_0 - x_0^\prime\|$. Considering the statement above that $\X_{x_0}(\theta_{x_0})$ as a function of $x_0$ is continuous in $\mathcal{G}$, we have that there exists another constant $C^\prime>0$ such that R.H.S. of (\ref{ThetaCont}) $\leq C^\prime \|x_0 - x_0^\prime\|$. Using assumption (\textbf{F}4), we complete the proof.\hfill$\square$\\

\section*{Proof of (5.24)} 

A Taylor expansion of $A_h(x+y)$ gives\\[-15pt]
\begin{align*}
A_h(x+y) = A_h(x) +\nabla A_h(x)y+\frac{1}{2}\nabla^{\otimes2} A_h(x)y^{\otimes2}+o(\|y^{\otimes2}\|),\\[-23pt]
\end{align*}
where the little-$o$ term can be chosen to be independent of $h$ due to assumptions (\textbf{F}1)--(\textbf{F}2) and the fact that $h$ is bounded. Similarly, a Taylor expansion of $a_h(x+y)$ leads to\\[-20pt]
%
%
%
%
\begin{align}\label{AATaylor}
a_h(x+y) &= a_h(x) -a_h(x)^3A_h(x)^T \mathbf{R}\nabla A_h(x)y-\frac{1}{2}a_h(x)^3\|\nabla A_h(x)y\|_{\mathbf{R}}\nonumber\\
&\hspace{1cm}-\frac{1}{2}a_h(x)^3 A_h(x)^T \mathbf{R}\nabla^{\otimes2}A_h(x)y^{\otimes2}+o(\|y^{\otimes2}\|),
\end{align}

where again the little-$o$ term is independent of $h$ due to the same reason as above. A Taylor expansion of $\int d^2  K(x+y-s)[d^2  K(x-s)]^T ds$ about $y = 0$ gives 
%
$\int d^2  K(x+y-s)[d^2  K(x-s)]^T ds =\mathbf{R}+\frac{1}{2}\int\nabla^{\otimes2}d^2  K(s)y^{\otimes2}[d^2  K(s)]^T ds+o(\|y^{\otimes2}\|),$ 
%
so that
\begin{align}\label{PartTaylor}
&A_h(x+y)^T  \int d^2  K(x+y-s)[d^2  K(x-s)]^T ds\;A_h(x)\nonumber\\
&\hspace{1cm}=(a_h(x))^{(-2)}+(\nabla A_h(x)y)^T \mathbf{R}A_h(x)+\frac{1}{2}(\nabla^{\otimes2} A_h(x)y^{\otimes2})^T \mathbf{R}A_h(x)\nonumber\\
&\hspace*{2cm} +\frac{1}{2}A_h(x)^T \int\nabla^{\otimes2}d^2  K(s)y^{\otimes2}[d^2  K(s)]^T dsA_h(x)+o(\|y^{\otimes2}\|).
\end{align}
Plugging all these expansions into (\ref{covU}) leads to
\begin{align*}
r_h&(x+y,x)\nonumber\\
&=a_h(x+y)a_h(x)\;A_h(x+y)^T  \int d^2  K(x+y-s)[d^2  K(x-s)]^T ds\;A_h(x)\nonumber\\
&=\Big\{1-(a_h(x))^2A_h(x)^T \mathbf{R}\nabla A_h(x)y-\frac{1}{2}(a_h(x))^2\big\{(\nabla A_h(x)y)^T \mathbf{R}\nabla A_h(x)y\\
&\hspace*{1cm}+A_h(x)^T \mathbf{R}\nabla^{\otimes2}A_h(x)y^{\otimes2}\big\}+o(\|y^{\otimes2}\|)\Big\}\bigg\{1+(a_h(x))^2(\nabla A_h(x)y)^T \mathbf{R}A_h(x)\\
&\hspace*{2cm}+\frac{1}{2}(a_h(x))^2\bigg\{A_h(x)^T \int\nabla^{\otimes2}d^2  K(s)y^{\otimes2}[d^2  K(s)]^T dsA_h(x)\\
&\hspace*{3cm}+(\nabla^{\otimes2} A_h(x)y^{\otimes2})^T \mathbf{R}A_h(x))\bigg\}+o(\|y^{\otimes2}\|)\bigg\}\\
&=1-\frac{1}{2}(a_h(x))^2(\nabla A_h(x)y)^T \mathbf{R}\nabla A_h(x)y-[(a_h(x))^2A_h(x)^T \mathbf{R}\nabla A_h(x)y]^2\nonumber\\
&\hspace*{1cm}+\frac{1}{2}(a_h(x))^2A_h(x)^T \int\nabla^{\otimes2}d^2  K(s)y^{\otimes2}[d^2  K(s)]^T dsA_h(x)+o(\|y^{\otimes2}\|)\nonumber\\
&= 1-y^T\Lambda_1(h,x)y-y^T\Lambda_2(hx)y+o(\|y^{\otimes2}\|)\nonumber\\
&= 1-y\Lambda(h,x)y^T+o(\|y^{\otimes2}\|)
\end{align*}
This is (\ref{RExpression}).\hfill$\square$

\section*{Proof of the fact that (5.22) implies (5.19)  } 

In the proof of Theorem~\ref{ConfBand} as presented in the main article it is claimed that (\ref{Replace1}) follows from (\ref{sufficiency}). Here we show in some details why this in fact is the case.\\

Write a two-dimensional standard Brownian bridge $B(x), x \in [0,1]^2$ as
\begin{align*}
B(x)=W(x)-x_1x_2W(1,1),
\end{align*}
where $W$ is a two-dimensional Wiener process. Let $X = (X_1,X_2)^T$ be a random vector in ${\mathbb R}^2$. Let $\mathfrak{T}: \mathbb{R}^2\mapsto[0,1]^2$ denote the Rosenblatt transformation (Rosenblatt 1952) defined as
$$\mathfrak{T} \binom{x_1}{x_2}= \binom{x_1^\prime}{x_2^\prime} = \binom{P(X_1 \leq x_1)}{P(X_2 \leq x_2 | X_1 = x_1)}.$$
Define\\[-20pt]
\begin{align*}
&\lefts{0}{Y_n(x)}=h^{-1}\left\langle \frac{a(x)A(x)}{\sqrt{f(x)}},\int_{\mathbb{R}^2}d^2  K\Big(\frac{x-s}{h}\Big)dB(\mathfrak{T}s)\right\rangle,\\
&\lefts{1}{Y}_n(x)=h^{-1}\left\langle \frac{a(x)A(x)}{\sqrt{f(x)}},\int_{\mathbb{R}^2}d^2  K\Big(\frac{x-s}{h}\Big)dW(\mathfrak{T}s)\right\rangle,\\
&\lefts{2}{Y}_n(x)=h^{-1}\left\langle \frac{a(x)A(x)}{\sqrt{f(x)}},\int_{\mathbb{R}^2}d^2  K\Big(\frac{x-s}{h}\Big)\sqrt{f(s)}dW(s)\right\rangle,\\
&\lefts{3}{Y}_n(x)=h^{-1}\left\langle a(x) A(x),\int_{\mathbb{R}^2}d^2  K\Big(\frac{x-s}{h}\Big)dW(s)\right\rangle.
\end{align*}
Following the arguments in the proof of Theorem~1 in Rosenblatt (1976), we similarly have
\begin{align*}
&\sup_{x\in\mathcal{L}}|Y_n(x)-\lefts{0}{Y_n(x)}|=O_p(h^{-1}n^{-1/6}(\log{n})^{3/2}),\\
&\sup_{x\in\mathcal{L}}|\lefts{0}{Y}_n(x)-\lefts{1}{Y}_n(x)|=O_p(h),\\
&\sup_{x\in\mathcal{L}}|\lefts{2}{Y}_n(x)-\lefts{3}{Y}_n(x)|=O_p(h).
\end{align*}
The two Gaussian fields $\lefts{1}{Y}_n(x)$ and $\lefts{2}{Y}_n(x)$ have the same probability structure. Note that the Gaussian fields $\lefts{3}{Y}_n(hx)$ and $U_h(x)$ have the same probability structure on $\mathcal{H}_h$. Hence in order to prove (\ref{Replace1}) it suffices to prove (\ref{sufficiency}).\hfill$\square$

\section{Miscellaneous results}
We will also need to estimate the derivative of $V(x)$ given by 
$$\nabla V(x)=\nabla G(d^2 f(x))\nabla d^2 f(x),\quad x\in{\cal H}^\epsilon.$$ 
The corresponding plug-in kernel estimators of $\nabla V(x)$ then is 
$$\nabla \hat V(x)=\nabla G(d^2\hat f(x))\nabla d^2\hat f(x),\quad x\in{\cal H}^\epsilon.$$
\begin{lemma}\label{LX}
Under assumptions (\textbf{F}1),(\textbf{K}1)--(\textbf{K}2) and (\textbf{H}1), we have for $\epsilon > 0$ that
\begin{align*}
&\sup_{x\in {\cal H}^\epsilon}\big\|\nabla \hat f(x)-{\mathbb E}\big[\nabla \hat f(x)\big]\big\|=O_p\bigg(\sqrt{\frac{\log{n}}{nh^4}}\bigg)\\
&\sup_{x\in {\cal H}^\epsilon}\big\|\nabla^2 \hat f(x)-\nabla^2 f(x)\big\|_F=O_p\bigg(\sqrt{\frac{\log{n}}{nh^6}}\bigg),\\
&\sup_{x\in {\cal H}^\epsilon}\big\|\nabla d^2\hat f(x)-\nabla d^2 f(x)\big\|_F=O_p\bigg(\sqrt{\frac{\log{n}}{nh^8}}\bigg).
\end{align*}
\end{lemma}
The same rate holds for $\nabla^2 f(x)$ replaced by ${\mathbb E}\nabla^2 \hat f(x)$.\\

\begin{lemma}\label{VXX}
Under assumptions (\textbf{F}1)--(\textbf{F}2), (\textbf{K}1)--(\textbf{K}2) and (\textbf{H}1), we have for $\epsilon > 0$ small enough that
\begin{align*}
&\sup_{x\in {\cal H}^\epsilon}\big\|\hat V(x)- V(x)\big\|=O_p\bigg(\sqrt{\frac{\log{n}}{nh^6}}\bigg),\\
&\sup_{x\in {\cal H}^\epsilon}\big\|\nabla^2\hat f(x) \hat V(x)- \nabla^2 f(x) V(x)\big\|=O_p\bigg(\sqrt{\frac{\log{n}}{nh^6}}\bigg), \\
&\sup_{x\in {\cal H}^\epsilon}\big\|\nabla G(d^2 \hat f(x))- \nabla G(d^2  f(x))\big\|_F=O_p\bigg(\sqrt{\frac{\log{n}}{nh^6}}\bigg),  \\
&\sup_{x\in {\cal H}^\epsilon}\big\|\nabla \hat V (x)- \nabla V (x)\big\|_F=O_p\bigg(\sqrt{\frac{\log{n}}{nh^8}}\bigg),  \\
&\sup_{x\in {\cal H}^\epsilon}\big\|\nabla \hat V (x)\hat V(x)- \nabla V (x)V(x)\big\|=O_p\bigg(\sqrt{\frac{\log{n}}{nh^8}}\bigg)
\end{align*}
\end{lemma}


The following result shows the uniform consistency of the estimator $\hat \X_{x_0}(t)$.
\begin{lemma}\label{XConsist}
For $x_0\in \mathcal{G}$ let $T_{x_0}^{min}, T_{x_0}^{max}\geq0$ be such that $T_{x_0}^{min}+ T_{x_0}^{max}>0$  and $T_{\mathcal{G}}:=\max\{\sup_{x_0\in \mathcal{G}}T_{x_0}^{min}, \sup_{x_0\in \mathcal{G}}T_{x_0}^{max}\}<\infty$ and $\{\X_{x_0}(t),\,t \in [-T_{x_0}^{min},T_{x_0}^{max}]\} \subset {\cal H}.$ Under assumptions (\textbf{F}1)--(\textbf{F}2), (\textbf{K}1)--(\textbf{K}2) and (\textbf{H}1), we have that
\begin{align*}
\sup_{x_0\in \mathcal{G}, t\in[-T_{x_0}^{min},T_{x_0}^{max}]}\|\hat \X_{x_0}(t)-\X_{x_0}(t)\|=o_p(1).
\end{align*}
\end{lemma}

\begin{proof}
Following the proof on page 1584 of Koltchinskii et al. (2007), we obtain that for all $x_0\in \mathcal{G}$ and $t\in[0,T_{x_0}^{max}]$ and for some constant $L>0$
\begin{align*}
\|\hat \X_{x_0}(t)-\X_{x_0}(t)\|\leq T_{x_0}^{max}\sup_{x\in\mathbb{R}^2}\|\hat V(x)-V(x)\|e^{Lt}.
\end{align*}
Therefore by Lemma~\ref{LX}, and the fact that $G$ is Lipschitz continuous (recall the definitions $\hat{V}(x) = G(d^2\hat{f}(x))$ and $V(x) = G(d^2f(x))$\,)
\begin{align*}
\sup_{x_0\in \mathcal{G}, t\in[0,T_{x_0}^{max}]}\|\hat \X_{x_0}(t)-\X_{x_0}(t)\|\leq T_{\mathcal{G}}\sup_{x\in\mathbb{R}^2}\|\hat V(x)-V(x)\|e^{LT_{\mathcal{G}}}=o_p(1).
\end{align*}
Similarly we can prove $\sup_{x_0\in \mathcal{G}, t\in[-T_{x_0}^{min},0]}\|\hat \X_{x_0}(t)-\X_{x_0}(t)\|=o_p(1)$ and therefore the lemma is proved.\hfill$\square$
\end{proof}

\subsection*{The function $G$ and some of its properties}

In what follows we show that with
\begin{align}\label{GDef}
G(u,v,w)=\left(\begin{array}{c} 2u-2w+2v-2\sqrt{(w-u)^2+4v^2} \\w-u+4v-\sqrt{(w-u)^2+4v^2} \end{array}\right)
\end{align}
we have that 
\begin{align}\label{G-and-V}
V(x)=G(d^2f(x))\quad\text{is a second eigenvector of the Hessian $\nabla^2 f(x)$}.
\end{align}
We also write $G=(G_1, G_2)^T$. Let $\lambda_2(x)$ be the second eigenvalue of $\nabla^2 f(x)$. Then $\lambda_2(x)$ is the smaller root of equation
\begin{align*}
\left|\begin{array}{cc}
f^{(2,0)}(x)-\lambda_2(x) & f^{(1,1)}(x)\\
f^{(1,1)}(x) & f^{(0,2)}(x)-\lambda_2(x)\\
\end{array}\right|
=0.
\end{align*}
Calculation shows that the second eigenvalue of  the Hessian is 
\begin{align*}
\lambda_2(x)=J(d^2 f(x))
\end{align*}
where\\[-20pt]
\begin{align}\label{FDef}
J(u,v,w)=\frac{u+w-\sqrt{(u-w)^2+4v^2}}{2}.
\end{align}
We denote $V(x)=(u,v)^T$. Since $V(x)$ is the second eigenvector of the Hessian, we have $(\nabla^2f(x)-\lambda_2(x)\mathbf{I})V(x)=0$, i.e.,
\begin{align}\label{EigenEqu}
\left(\begin{array}{cc}
f^{(2,0)}(x)-\lambda_2(x) & f^{(1,1)}(x)\\
f^{(1,1)}(x) & f^{(0,2)}(x)-\lambda_2(x)\\
\end{array}\right)
\left(\begin{array}{c}
u\\
v\\
\end{array}\right)=0.
\end{align}
Note that the two equations above are linearly dependent. Also notice that both $V_1(x):=\Big(\lambda_2(x)-f^{(0,2)}(x),\; f^{(1,1)}(x)\Big)^T$ and $V_2(x):=\Big(f^{(1,1)}(x), \; \lambda_2(x)-f^{(2,0)}(x)\Big)^T$ are solutions to one of (and therefore both of) the equations in (\ref{EigenEqu}). For any $c_1(x),\; c_2(x)\in\mathbb{R}$, if
\begin{align*}
V(x)&=c_1(x)V_1(x)+c_2(x)V_2(x),
\end{align*}
then $V(x)$ satisfies the equations in (\ref{EigenEqu}). We want to find $c_1(x)$ and $c_2(x)$ such that $V(x)\neq0$ if the two eigenvalues of $\nabla^2 f(x)$ are not equal, i.e. $f^{(2,0)}(x) \neq f^{(0,2)}(x)$ and $f^{(1,1)}(x)\neq0$. For this purpose we choose $c_1(x)\equiv4$ and $c_2(x)\equiv2$. As a result, $V(x)=G(d^2f(x))$ with $G$ defined in (\ref{GDef}), which is (\ref{G-and-V}) \\

There are other ways of choosing $c_1(x)$ and $c_2(x)$ but  all that matters for our results is that $V(x)$ is smooth and that $\|V(x)\|$ is bounded away from zero (and infinity) as long as the two eigenvalues of Hessian $\nabla^2 f(x)$ are distinct.\\

{\em Lipschitz continuity of $G$ and $\nabla G$.} First observe that by assumption (\textbf{F}2), there exists a $\delta>0$ such that $\{d^2 f(x): x\in {\cal H}\}\subset \mathcal {Q}_{\delta}$ where
\begin{align}\label{Q-def}
\mathcal {Q}_{\delta}=\{(u,v,w)\in\mathbb{R}^3:|u-w|>\delta\;\; \textrm{or} \;\; |v|>\delta\},
\end{align}
since two eigenvalues of a $2\times2$ symmetric matrix are equal iff the matrix is a scaled identity matrix (see also the discussion given after the assumptions in the paper). \\

To verify Lipschitz continuity of $G(u,v,w)$  on $\mathbb{R}^3$, it suffices to notice that
\begin{align*}
&\Big|\sqrt{(u_1-w_1)^2+4v_1^2}-\sqrt{(u_2-w_2)^2+4v_2^2}\Big|\\
&=\frac{\big|(u_1-u_2-w_1+w_2)(u_1-w_1+u_2-w_2)+4(v_1-v_2)(v_1+v_2)\big|}{\sqrt{(u_1-w_1)^2+4v_1^2}+\sqrt{(u_2-w_2)^2+4v_2^2}}\\
&\leq |u_1-u_2|+|w_1-w_2|+2|v_1-v_2|.
\end{align*}
As for Lipschitz continuity of $\nabla G(u,v,w)$ on $\mathcal {Q}_{\delta}$, it suffices to notice that the components of $\nabla^2 G_i (u,v,w)$ are all bounded on $\mathcal {Q}_{\delta}$, where $((u-w)^2+4v^2)^{3/2}$ in the denominator is bounded away from zero on $\mathcal {Q}_{\delta}$.
\end{appendix}
\end{document}